\documentclass[a4paper,11pt]{article}
\usepackage[utf8]{inputenc} 
\usepackage[ngerman,american]{babel}

\usepackage{amsmath}
\usepackage{amsfonts}
\usepackage{amssymb}
\usepackage{enumerate}
\usepackage{color}
\usepackage{fullpage}
\usepackage{dirtytalk}
\usepackage{mathtools}

\usepackage[table,dvipsnames]{xcolor}
\usepackage{tikz}
\usetikzlibrary{decorations.pathreplacing}

\usepackage{amsthm}

\theoremstyle{plain}
\newtheorem{theorem}{Theorem}
\newtheorem{lemma}[theorem]{Lemma}

\newtheorem{claim}[theorem]{Claim}
\newtheorem{corollary}[theorem]{Corollary}
\theoremstyle{definition}
\newtheorem{definition}[theorem]{Definition}
\newtheorem{remark}[theorem]{Remark}
\newtheorem{conjecture}{Conjecture}

\usepackage[colorlinks]{hyperref}
\definecolor{lightblue}{rgb}{0.5,0.5,1.0}
\definecolor{darkred}{rgb}{0.5,0,0}
\definecolor{darkgreen}{rgb}{0,0.5,0}
\definecolor{darkblue}{rgb}{0,0,0.5}

\hypersetup{colorlinks,linkcolor=darkred,filecolor=darkgreen,urlcolor=darkred,citecolor=darkblue}

\usepackage{algorithm}
\usepackage{algorithmic}

\definecolor{gray}{gray}{0.3}

\makeatletter
\providecommand*{\toclevel@algorithm}{0}
\makeatother

\usepackage{hyphenat}
\newcommand{\leftmset}{\{\!\!\{}
\newcommand{\rightmset}{\}\!\!\}}

\DeclareMathOperator{\maxdeg}{\Delta}

\DeclareMathOperator{\bulk}{bulk}
\DeclareMathOperator{\Span}{span}
\DeclareMathOperator{\nucl}{coat}
\DeclareMathOperator{\partition}{part}

\DeclareMathOperator{\Vmaxdeg}{\mathit{V}_{\Delta}}

\newcommand{\NTildeZero}[2]{\Span_{#1}{(#2)}}
\newcommand{\NTildeOne}[2]{{N}_{#1}(\Span{(#2)})}
\newcommand{\NTildeTwo}[2]{\overline{N_{#1}\left[\Span{(#2)}\right]}}
\newcommand{\NTwo}[2]{\overline{N_{#1}\left[{#2}\right]}}

\title{Interval Graphs are Reconstructible}
\author{Irene Heinrich \\
{\tt heinrich@mathematik.tu-darmstadt.de}
 \and 
Masashi Kiyomi\\
{\tt kiyomi@st.seikei.ac.jp}
 \and Yota Otachi, \\
{\tt otachi@nagoya-u.jp}
\and 
Pascal Schweitzer \\ 
{\tt schweitzer@mathematik.tu-darmstadt.de}
}

\begin{document}

\maketitle

\begin{abstract}
A graph is reconstructible if it is determined up to isomorphism by the multiset of its proper induced subgraphs. The reconstruction conjecture postulates that every graph of order at least~3 is reconstructible.

We show that interval graphs with at least three vertices are reconstructible.
For this purpose, we develop a technique to handle separations in the context of reconstruction. This resolves a major roadblock to using graph structure theory in the context of reconstruction.
To apply our novel technique, we also develop a resilient combinatorial structure theory for interval graphs. 

A consequence of our result is that interval graphs can be reconstructed in polynomial time.

\end{abstract}
\section{Introduction}

Due to its fundamental nature, the reconstruction conjecture is one of the most prominent open problems in graph theory~\cite{Lauri2016}. 
First posed by Kelly and Ulam in 1942~\cite{Kelly57}, it inquires to what extent the structure of a graph is determined by its parts. Intuitively it asks whether the~$(n-1)$-vertex induced subgraphs of a graph~$G$ of order~$n$ determine the graph~$G$ up to isomorphism.

More formally, for a vertex~$v\in V(G)$ of a finite simple undirected graph~$G$, we define the \emph{card}~$G_v$ to be the class $[G-v]_{\cong}$ of graphs isomorphic to $G-v$. Thereby, a graph of order~$n$ gives rise to a multiset of~$n$ cards~$\leftmset G_v\colon v\in V(G)\rightmset$, which
is
called the \emph{deck} of~$G$.
The central question is now whether non-isomorphic graphs can have the same deck. A small example consists of
the complete graph on two vertices and its complement.
The deck of each of these graphs consists of two 1-vertex graphs.

The reconstruction conjecture postulates that this is the only counterexample in the sense that every graph of order at least 3 is determined (up to isomorphism) by its deck.

The conjecture succinctly captures
the fundamental question of whether the global structure of a graph is determined by its induced substructures. 

Since it is readily stated, many graph theorists have approached the conjecture over time. Yet, despite 70 years of research, the reconstruction conjecture remains wide open and one of the most tenacious unresolved problems in structural graph theory. 
Overall this highlights gaps in our understanding of graph structure theory. Indeed, the conjecture remains open even for many basic graph classes that are generally considered well understood (see related work below).

The fundamental question has long attracted significant interest (we refer to surveys and books~\cite{MR360368,MR0480189,bondy_1991,BabaiHandbook,LauriInHandbook,DBLP:journals/arscom/AsciakFLM10,Lauri2016}). It also has numerous applications on a broad range of domains. These include finite model theory~\cite{DBLP:journals/jgt/EgrotH22,tuprints26387}, machine learning~\cite{DBLP:journals/tnn/Shawe-Taylor93,DBLP:journals/pami/BouritsasFZB23,DBLP:conf/nips/CottaMR21}, isomorphism problems~\cite{Zemlyachenko1985}, and complexity questions surrounding them~\cite{DBLP:journals/mst/KratschH94,HEMASPAANDRA2007103,Kiyomi2010,DBLP:journals/jct/Huber11,DBLP:journals/cc/HemaspaandraHSW20}. 
Related forms of reconstruction problems arise for other combinatorial objects~\cite{DBLP:journals/jct/KrasikovR97,e85c79e0-d93c-36bd-9f89-07652a5d63a3,DBLP:journals/jct/AlonCKR89}, and also in related topics such as constraint satisfaction problems~\cite{DBLP:journals/siamdm/MontanariRT11} and topology~\cite{PITZ_SUABEDISSEN_2017}. Non-discrete variants are related to computed tomography (CT) (see~\cite{doi:10.1179/1743280413Y.0000000023}).

At first sight, the conjecture may seem obviously true. However, Stockmeyer found an infinite family of tournament pairs which show that for directed graphs the conjecture is false~\cite{https://doi.org/10.1002/jgt.3190010108}. This rules out various approaches to reconstruction that do not take the undirected nature of simple graphs into account.

This may explain why, even today, there seems to be no clear consensus on whether to believe in the conjecture. Another contributing factor could be that, in general, progress towards resolving the conjecture has been slow. Recent research has focused on variants and related aspects of reconstruction (see related work below), but regarding the original problem only limited insights have been reported. 

Let us discuss which graph classes are known to be reconstructible.
All graphs of order at least~3 in the following natural classes are reconstructible:
trees~\cite{Kelly57},
separable graphs without degree-1 vertices~\cite{Bondy69},
outerplanar graphs~\cite{Giles1974},
regular graphs (see~\cite{MR0480189}),
unit interval graphs~\cite{DBLP:journals/dm/Rimscha83},
non-trivial Cartesian products, and
bipartite permutation graphs \cite{kiyomi2012}. 
We should highlight the fact that, apart from the 2012 result on bipartite permutation graphs, all of the other results are
 more than 40 years old.

Recently it has also been shown computationally that 
graphs with up to 13 vertices are reconstructible~\cite{MR4446370}. 

In this paper, we take a novel approach to the reconstruction conjecture using isomorphism-related concepts and structure theory. Indeed, we are interested in showing larger graph classes to be reconstructible. For this, we develop techniques to handle separations in the context of reconstruction. We then apply our techniques 
to the class of interval graphs.
These are known to be recognizable~\cite{DBLP:journals/dm/Rimscha83} (i.e., we can infer from the deck whether the graph is an interval graph). However, they were only known to be reconstructible in the special case of unit length intervals. 
Using our new techniques, we prove that interval graphs are reconstructible.

\begin{theorem}\label{main:thm}
Every interval graph on at least three vertices is reconstructible.
\end{theorem}

We should remark that when it comes to which graphs are challenging, the reconstruction conjecture behaves differently to many typical graph theoretical problems.
Indeed, note that reconstructibility of regular graphs is trivial. However, conversely, showing reconstructibility of trees required effort. Overall it is generally not clear how to apply graph structure theory to the problem.
In fact, when it comes to classes that do not have built-in regularity, we provide one of the largest classes of graphs for which reconstructibility is known.

As a direct consequence, by combining our theorem with previous work on problems related to decks~\cite{Kiyomi2010}, we conclude that reconstruction of interval graphs is efficiently possible.
\begin{corollary}
Interval graphs can be reconstructed in polynomial time.
\end{corollary}

\paragraph{Techniques.} To prove the theorem, we develop three key techniques: (1) a reconstruction method based on separations, (2) a suitable method to annotate structural information in parts of the graph, and (3) a strategy to identify resilient parts of the graph that serve as anchors for reconstruction. 

We believe these techniques have applications beyond our main result. In particular, we believe that our techniques for reconstruction may be adapted to graphs of bounded treewidth and beyond. In any case, our technique is the first method for reconstruction that is compatible with traditional methods of decomposing a graph with a separator of more than one vertex. This resolves a major roadblock to using graph structure theory in the context of reconstruction.

We also develop a combinatorial structure theory for interval graphs that may have applications beyond reconstruction. The important property is that the structure is, to some extent, resilient to vertex deletion.

In addition to the three techniques, we also develop various tricks to handle edge cases. We describe all of these techniques in more detail next. 

\subparagraph{1. Reconstruction-by-separation.}
Our first new technique (called Reconstruction-by-Separation, Lemma~\ref{lem:reconst:clique:sep:implies:reconst}), which is
central to our approach, is to decompose the graph along a structurally simple separator.
The hope is that we can reassemble the graph (uniquely up to isomorphism) from sufficient information on the parts. To make this work, it is clear that we need to understand how vertices in the different parts are interconnected via the vertices in the separator.

This general technique is not specific to interval graphs. In the concrete case of interval graphs, we  separate the graph along a so-called ``clean clique separation'' into two parts. 
To 
reconstruct how
the separator is attached to the parts, we ensure that, within each of the parts, the neighborhoods of the separator-vertices
are linearly ordered by inclusion. 

\subparagraph{2. Annotated graphs and distant vertices.}
Our second technique consists of using annotated graphs and the Distant Vertex Lemma (Lemma~\ref{lemma:farthest:away:and:at:least:two:reconstruct:other:side}).
For this, since the neighborhoods of separator-vertices within the parts are linearly ordered,
in order to reassemble the graph, we only need to know how many neighbors a vertex has within a part. This is the case because, due to the linearly ordered neighborhoods, separator vertices with an equal number of neighbors within a part are adjacent to exactly the same vertices. 
We formalize this information by annotating the graphs that are induced by the parts
as to contain sufficient information on their connections in the rest of the graph. 

To determine the annotation of the graph, we choose cards of vertices that are ``farthest away'' from the separator. Our concept of ``farthest away'' must be defined in a combinatorial, isomorphism-invariant manner, which complicates the matter. 

(For generalizations of our technique to other graph classes, such as graph of bounded threewidth, the linearly ordered neighborhoods would be replaced by an upped bound on the separator.)

\begin{figure}
	\begin{center}
		
		\begin{tikzpicture}[>=stealth]
			
			\draw[thick, LimeGreen] (-0.5,0) -- (2,0);
			\draw[thick, LimeGreen] (-0.5,0.125) -- (2,0.125);
			\draw[thick, LimeGreen] (-0.5,0.25) -- (2,0.25);
			\draw[thick, LimeGreen] (1.5,0.5) -- (4,0.5);
			\draw (-0.5,-0.5) -- (1,-0.5);
			\draw (-0.5,-0.75) -- (0.5,-0.75);
			\draw [decorate,Plum, decoration={brace, amplitude=5pt}](6.75,-1.5) -- node[below=1ex](I0){flank}(5.75,-1.5);
			\draw (2.5,0) -- (4,0);
			\draw (4.5,0) -- (6,0);
			\draw [decorate,Plum, decoration={brace, amplitude=5pt}](1,-1.5) -- node[below=1ex](I0){outsiders}(-0.5,-1.5);
			
			\draw (3.75,-0.25) -- (4.75,-0.25);
			\draw (0.75,-0.25) -- (1.75,-0.25);
			\draw[thick, LimeGreen] (3.5,0.25) -- (5.25,0.25);
			\draw (1,-0.625) -- (0,-0.625);
			\node (v1) at (2.5,2.75) {bordering maximum degree classes};
			\draw[bend right,->] (v1) edge (0.75,0.425);
			\draw[bend left,->] (v1) edge (4.5,0.425);
			\draw (5.75,-0.25) -- (6.75,-0.25);
			\draw[decorate,Plum, decoration={brace, amplitude=5pt}] (-0.5,1.5) -- node[above=1ex](I0){bulk}(5.25,1.5);
			\draw[dashed] (5.5,1.75) -- (5.5,-2.25) node (v3) {};
			\draw[dashed] (1.25,1.75) -- (1.25,-2.25) node (v4) {};
			\node (v2) at (3.5,-3.5) {clean clique separations};
			\draw[->, bend right=10] (v2) edge (v3);
			\draw[->, bend left=10] (v2) edge (v4); 
			\draw (6.75,-0.5) -- (6.25,-0.5);
			\draw[dashed]  (-0.75,0.75) rectangle (7,-1);
			\draw (5,-0.25) -- (5.25,-0.25);
			\draw (4.5,-0.5) -- (5.125,-0.5);
			\draw (5,-0.75) -- (6,-0.75);
		\end{tikzpicture}
	\end{center}
	
	\caption{The figure illustrates our structure theorem on an interval representation shown inside the dashed box. The green thick intervals correspond to maximum degree vertices. While the figure helps to convey the intuition behind the definitions of outsiders, bulk, and flank, crucially, the actual definitions are combinatorial and independent of the interval representation.}\label{fig:generic:structure:int:graph}
\end{figure}
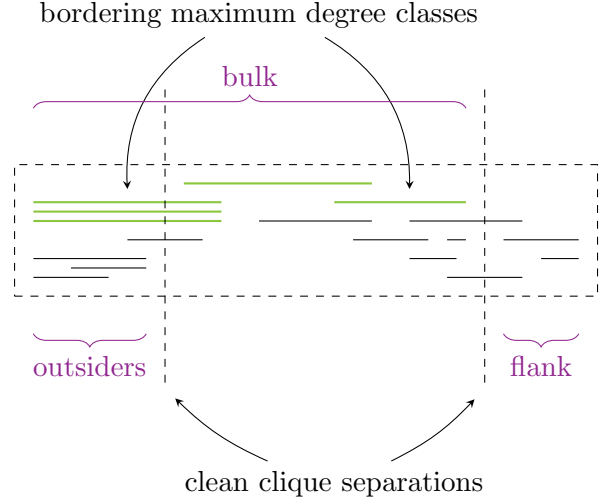

\subparagraph{3. Structure theory for interval graphs.}
The core problem that remains 
is that we need to locate the separator
in the cards. The challenge is that we do not know where a vertex has been deleted.
It seems very difficult to describe
such a
separator
of the graph in combinatorial terms that allow us to locate it in cards without fail. For example, a separator-vertex might have been removed in a card. To address this, we allow the separator under consideration to vary according to the structure of the graph.
Indeed, we identify for interval graphs certain resilient parts of the graph which we can often locate in cards as follows (see Figure~\ref{fig:generic:structure:int:graph}).

We define the \emph{bulk} of the graph as the graph spanned (in a well-defined combinatorial sense) by all the maximal degree vertices. 
If the bulk does not span the entire graph, there will be vertices that are not in the closed neighborhood of the bulk.
These vertices naturally partition into one or two combinatorially well-defined components (intuitively left and right) which we call \emph{flanks}. We can use the common neighbors of a flank and the bulk as the separator, which turns out to induce a clean clique separation.
However, it may be the case that there are no flanks. In this case, we define the so-called \emph{outsiders} which, again in a combinatorially well-defined sense, are certain leftmost or rightmost vertices. Here, we can use the open neighborhood of an outsider class as a separator.

\subparagraph{4. Edge cases.}
The three techniques described so far outline our general approach. However, in  the concrete applications there are several conceptual problems that repeatedly occur, especially in edge cases. Let us, at a high level, describe the crux of the recurring issues.

Intuitively, due to possible left-right symmetries of interval graphs, it is often challenging to determine whether a vertex was deleted on the left or the right ``side'' of the graph. This can happen for example when the side from which the vertex was deleted is exactly one vertex larger than the other side and, after the deletion, the sides look similar.
In this case, our solution is often to delete a vertex from the smaller side.

A similar issue appears when the sides have the same order (and look similar) in the original graph. When a vertex~$v$ is then deleted and we inspect the card~$G_v$, we can distinguish the sides in~$G_v$ due to their size. 
However, we may not be able to infer how the sides of~$G_v$ correspond to the sides of~$G$. To address this issue, we typically use several cards besides~$G_v$ to recover sufficient information by observing that some subgraph is simultaneously present in all the cards considered.

Finally, another issue may be that when a vertex is removed, no vertices are left on a side. It may then be difficult to locate the side within the card at all. This issue typically appears in edge cases and is resolved using case distinctions for concrete situations.

Overall, these issues make many parts of the proof somewhat technical. Many edge cases turn out to require careful attention to detail or different techniques.

\paragraph{Outline.} After some related work (Section~\ref{sec:related:work}) and preliminaries (Section~\ref{sec:prelims}), we develop a structure theory for interval graphs (Section~\ref{sec:structure:theory}). This structure theory discusses clean clique separations and linearly ordered neighborhoods. It also introduces the decomposition of a graph into the bulk, the (possibly empty) flanks, and the vertices connecting the bulk with the flanks. We then discuss annotated subgraphs (Section~\ref{sec:annotated:subgraphs}) and develop two techniques to reconstruct a graph decomposed along a clean clique separation. These techniques are Reconstruction-by-Separation (Lemma~\ref{lem:reconst:clique:sep:implies:reconst}) and the Distant Vertex Lemma (Lemma~\ref{lemma:farthest:away:and:at:least:two:reconstruct:other:side}).

We then argue first that the sizes of the flanks are reconstructible (Section~\ref{sec:flank:sizes}), which
allows us to treat the case with two flanks (Section~\ref{sec: two-flanks}).
We then introduce outsiders (Section~\ref{sec:outsiders}), which allow us to resolve the case with no flanks (Section~\ref{sec:no:flank}). The case with exactly one flank (Section~\ref{sec:one:flank}) is in some sense a combination of the techniques of the other two cases.

\section{Related Work}\label{sec:related:work}

Reconstruction problems have been widely studied in combinatorics and theoretical computer science.
We refer to surveys on the reconstruction conjecture~\cite{MR360368},~\cite{MR0480189},~\cite[Chapters~8--11]{Lauri2016} for a broader context.
In the following, we summarize key results in addition to those mentioned in the introduction.

Various graph invariants are known to be reconstructible such as planarity, connectivity, the characteristic polynomial, and the Tutte polynomial~\cite{MR538033}.

Regarding reductions, the reconstruction conjecture is true if all 2-connected graphs are reconstructible~\cite{Yang88}. A recent strengthening of this involves additional restrictions on the diameter, co-diameter, and triangle-freeness~\cite{MR4658425}.

Many variants of reconstructibility have been considered over the years. These include reconstructibility from smaller subgraphs, for example, cubic graphs are reconstructible from the deck of two-vertex-deleted cards \cite{Kostochka2021}. They also include edge reconstruction, which, for example, is possible for claw-free graphs~\cite{Ellingham1988}. Finally, there are several different versions of switching reconstruction~\cite{MR787322,MR2839371,MR3218273}.

The \emph{reconstruction number} is the minimum number of well-chosen cards that suffice to reconstruct the graph. The \emph{adversary reconstruction number} of a graph is the minimum number~$k$ of cards such that every $k$-multiset of cards uniquely determines the graph. It has been shown that almost every graph has adversary reconstruction number~3~\cite{Bollobas1990, Muller1976, Myrvold1988}.

Motivated by applications in graph machine learning, reconstruction based only on color refinement information (also known as the~1-dimensional Weisfeiler-Leman algorithm) of the cards was recently investigated~\cite{DBLP:journals/corr/abs-2406-09351}.
Finally, there are algebraic formulations of the conjecture~\cite{DBLP:journals/jgt/OliveiraT16}, as well as a formulation for the edge reconstruction conjecture in terms of combinatorial K-theory~\cite{calle2024combinatorialktheoryperspectiveedge}.

Regarding interval graphs, we should mention that there are many efficient algorithms to decide whether a given graph is an interval graph and also compute an interval representation. We refer to the introduction of~\cite{DBLP:journals/siamcomp/KoblerKLV11} for an overview.

\section{Preliminaries}\label{sec:prelims}
All graphs in this paper are finite and without loops or parallel edges.
Let $G$ be a graph. We denote the vertex set and the edge set of $G$ by $V(G)$ and $E(G)$, respectively.
The~\emph{order} of~$G$ is~$|V(G)|$.
By~$\maxdeg(G)$ we denote the maximum degree of~$G$.
We set $\Vmaxdeg(G)\coloneqq \{v\in V(G)\colon \deg_G(v) = \maxdeg(G) \}$.
For two disjoint vertex subsets~$U$ and~$W$ of~$G$, we denote the set of edges with one end in $U$ and the other end in~$W$ by $E(U,W)$.
We use~``$\leftmset$'' and~``$\rightmset$'' to indicate multisets.

\paragraph*{Neighborhoods and neighborhood partitions.} 
By~$N_G(v)\coloneqq \{w\in V(G)\colon \{v,w\} \in E(G) \}$ and~$N_G[v]\coloneqq N_G(v)\cup \{v\}$, we denote the \emph{open} and \emph{closed neighborhood} of~$v$, respectively.
For a set $S \subseteq V(G)$, the set
$N_G[S] \coloneqq \bigcup_{s \in S} N_G[s]$ is the \emph{closed neighborhood} of~$S$,
$N_G(S) \coloneqq N_G[S] \setminus S$ is the \emph{open neighborhood} of~$S$, and
we set~$\NTwo{G}{S}\coloneqq V(G) \setminus N_G[S]$ as the complement of the closed neighborhood.
We define
\[\partition_G(S) \coloneqq (S,N_G(S),\NTwo{G}{S} ).\] Note that $\partition_G(S)$ is an ordered partition of $V(G)$.
A vertex~$v$ of~$G$ is \emph{universal} if $N_G[v] = V(G)$.

\paragraph{Twin equivalence classes.}
Let~$G$ be a graph and let~$u$ and~$v$ be two vertices of~$G$.
If~$N_G(u) = N_G(v)$, then $u$ and $v$ are \emph{true twins}.
If $N_G[u] = N_G[v]$, then $u$ and $v$ are \emph{false twins}.
If $u$ and $v$ are true or false twins, we say that $u$ and $v$ are \emph{twins}.
Note that true twins are non-adjacent and false twins are adjacent.
The relation which consists of all vertex pairs that are twins in~$G$ is an equivalence relation and the equivalence classes are called \emph{twin equivalence classes}.

\paragraph*{Interval graphs.} A graph $G$ is an \emph{interval graph} if there exists a family $\{[\ell_v, r_v] \colon v \in V(G)\}$  of closed intervals over~$\mathbb{R}$, one interval for each vertex of $G$, such that two vertices $u$ and $v$ of~$G$ are adjacent precisely if $[\ell_u, r_u] \cap [\ell_v, r_v] \neq \emptyset$.
In this case, $\{[\ell_v, r_v] \colon v \in V(G) \}$ is an \emph{interval representation} of~$G$.
For a vertex subset $W \subseteq V(G)$ and a given interval representation of $G$, we set $\ell_W \coloneqq \min \{\ell_w\colon w \in W\}$ and $r_W \coloneqq \max \{r_w\colon w \in W \}$.
If $G$ has an interval representation for which all intervals are of the same length, then $G$ is a \emph{unit interval graph}.
If~$G$ has an interval representation for which no interval is properly contained in another interval, then~$G$ is a \emph{proper interval graph}.
A graph is a unit interval graph if and only if it is a proper interval graph (see~\cite{Bogart99, Gardi07, Roberts69}).

A graph is \emph{chordal} if it is free of induced cycles of length at least~4.
A 3-vertex subset $S$ of a graph $G$ is an \emph{asteroidal triple} if $S$ is an independent set of $G$ and for each $v \in S$, removing~$N_G[v]$ from $G$ does not disconnect the two other vertices of $S$.
\begin{lemma}[\cite{Lekkerkerker1962} and \cite{Fulkerson1965}] \label{lem: conditions of interval graphs}
	A graph is an interval graph if and only if it is chordal and has no asteroidal triple.
\end{lemma}

\paragraph*{Graph reconstruction.}
Let $G$ be a graph. We denote by~$ [G]_{\cong}$ the class of all graphs isomorphic to~$G$.
For each vertex $v$ in $V(G)$, we call~$G_v\coloneqq [G-v]_{\cong}$ a \emph{card} of~$G$.
We will frequently think of~$G_v$ as a canonical representative from~$[G-v]_{\cong}$, so we can talk about its vertices, edges, and similar.
The \emph{deck} of~$G$ is the multiset
\[\mathcal{D}(G) \coloneqq \leftmset G_v \colon v \in V(G) \rightmset.\]
We say that $G$ is \emph{reconstructible} if every graph $H$ with~$\mathcal{D}(G) = \mathcal{D}(H)$ is isomorphic to~$G$.

\begin{conjecture}[Reconstruction conjecture]
	Every graph of order at least~3 is reconstructible.
\end{conjecture}

A graph invariant $f$ is \emph{reconstructible}	if $f(G) = f(H)$ for every pair of graphs $G$ and $H$ with~$\mathcal{D}(G) = \mathcal{D}(H)$.
For a given graph $G$, we say that we can \emph{reconstruct} $f(G)$ if $f(G) = f(H)$ for every graph~$H$ with~$\mathcal{D}(H)= \mathcal{D}(G)$. 

A graph class $\mathcal{G}$ is \emph{recognizable} if the invariant of being in the class is reconstructible. It is \emph{reconstructible} if every graph in $\mathcal{G}$ is reconstructible.
A \emph{vertex property} is a function~$f'$ that associates to every pair consisting of a graph~$G$ and a vertex~$v\in V(G)$ a mathematical object~$f'(G,v)$. An example would be the degree~$f'(G,v)\coloneqq \deg_G(v)$.
A vertex property $f'$ is \emph{reconstructible} if the multiset~$\leftmset [(G-v,f'(G,v))]_{\cong}\colon v\in V(G)  \rightmset$ is reconstructible. Here~$[(G-v,f'(G,v))]_{\cong}$ is the isomorphism class of the pair~$(G-v,f'(G,v))$, where
two pairs~$(G_1-v_1,f'(G_1,v_1))$ and~$(G_2-v_2,f'(G_2,v_2))$ are isomorphic 
if there is a bijection from $V(G_1)$ to $V(G_2)$ which restricts to an isomorphism between the first entries and to an isomorphism between the second entries.
For the degree, this means that the first entries are isomorphic graphs and~$\deg_{G_1}(v_1) = \deg_{G_2}(v_2)$.

We want to highlight that there is also a stronger notion that is sometimes useful in reconstruction.
We say that a vertex property $f'$ is \emph{strongly reconstructible} if  $f'(G,v) = f'(H,x)$ whenever $v \in V(G)$, $x \in V(H)$, $\mathcal{D}(G) = \mathcal{D}(H)$, and $G_v = H_x$.

Finally, a map associating with every vertex~$v$ of a graph~$G$ a subset~$f''(G,v)\subseteq V(G-v)$ of the vertices of~$G-v$ (e.g., associating with a vertex all vertices of distance 2) is \emph{reconstructible} if~$\leftmset [(G-v,f''(G,v))]_{\cong}\colon v\in V(G)  \rightmset$ is reconstructible. Intuitively this means that in the card~$G_v$, we can mark the set~$f''(G,	v)$.

We collect some known facts regarding reconstructibility and interval graphs.

\begin{theorem}[Von Rimscha 1983 \cite{DBLP:journals/dm/Rimscha83}]
	The class of interval graphs is recognizable.
\end{theorem}

\begin{lemma}[See~\cite{MR0480189}] \label{lem: reconstruct-degree-sequence}
	If $G$ is a graph of order at least~3, then
	$\leftmset \deg_G(u)\colon u \in V(G) \rightmset$ is reconstructible. For every vertex $v$ of $G$, the degree~$\deg_G(v)$ and $\leftmset \deg_G(u)\colon u \in N_G(v) \rightmset$ are strongly reconstructible.
\end{lemma}

\begin{lemma}[cf.~\cite{MR0480189}]
	\label{lem: reconstructible-graph-classes}
	The following graph classes are reconstructible:
	\begin{enumerate}[(i)]
		\item disconnected graphs,
		\item graphs with a universal vertex.
	\end{enumerate}
\end{lemma}

\section{The structure of interval graphs}\label{sec:structure:theory}
\subsection{Compact and tidy representations and bordering vertices}

	Let $G$ be an interval graph with an interval representation $\mathcal{I} \coloneqq \{[\ell_v, r_v ]\colon v\in V(G) \}$.
	The interval representation $\mathcal{I}$ is \emph{compact} if $|\{\ell_v \colon v\in V  \} \cup \{r_v \colon v \in V \}|$ is minimal among all interval representations of $G$. Note that every interval graph has a compact interval representation.
	We call~$\mathcal{I}$ \emph{tidy} if for every~$v\in V(G)$ and every~$M\subseteq V(G)$ with~$N_G[v]\subseteq N_G[M]$ we have~$[\ell_v, r_v ]\subseteq [\min\{\ell_m\colon m\in M\}, \max\{r_m\colon m\in M\} ]$.

\begin{lemma}\label{existence:tidy}
If $\mathcal{I} \coloneqq \{[\ell_v, r_v ]\colon v\in V(G) \}$ is an interval representation of a graph~$G$, then there is a tidy interval representation  $\mathcal{I'} \coloneqq \{[\ell'_v, r'_v ]\colon v\in V(G) \}$ of~$G$ for which~$[\ell'_v, r'_v ]\subseteq [\ell_v, r_v ]$ for all~$v\in V(G)$. Furthermore, every compact interval representation is tidy.
\end{lemma}
\begin{proof}
	Let~$\mathcal{I} =\{[\ell_v, r_v ]\colon v\in V(G) \}$ be an interval representation of $G$. 
	For $M \subseteq V(G)$, we set~$\ell_M \coloneqq \min \{\ell_m\colon m\in M\}$ and $r_M \coloneqq \max \{r_m\colon m\in M\}$.
	Assume that there is a vertex~$v\in V(G)$ and a vertex set~$M\subseteq V(G)$ with~$N_G[v]\subseteq N_G[M]$ but $[\ell_v, r_v ]\nsubseteq [\ell_M, r_M]$.
	By symmetry, we can assume~$\ell_v < \ell_M$. 
	We obtain from $N_G[v]\subseteq N_G[M]$ that~$r_u\notin [\ell_v,\ell_M)$ for all~$u\in V(G)$.
	We alter the interval representation by replacing every interval of the form~$[\ell_w,r_w]$ for which~$\ell_w=\ell_v$ by~$[\ell_M,r_w]$.
	This gives us an interval representation of~$G$ with strictly fewer endpoints of intervals. The lemma now follows by induction.
	
	The same argument implies that compact interval representations are tidy.
\end{proof}

\begin{lemma}
	If two maximum degree vertices~$u$ and~$v$ of an interval graph $G$ are twins, then either $E(G) = \emptyset$ or $u$ and $v$ are false twins.
\end{lemma}

\begin{proof}
	Let~$u$ and~$v$ be true twins of maximum degree in an interval graph $G$ with an interval representation $\{[l_x, r_x]\colon x \in V(G) \}$.
	By symmetry, we may assume that $r_u < \ell_v$.
	If there is a common neighbor $w$  of $u$ and $v$, then $[r_u, \ell_v] \subseteq [\ell_w, r_w]$.
	In particular, the common neighbors of $u$ and $v$ form a clique.
    Since $u$ and $v$ are true twins, we have $N_G(u) = N_G(u) \cap N_G(v) \subseteq (N_G(w) \cup \{w\})$ and $\{u,v\} \subseteq N_G(w)\setminus N_G(u)$.
    This yields~$|N_G(w)| \geq |N_G(u)|+1 = \maxdeg(G) + 1$, which is a contradiction. Hence, $N_G(u) = N_G(u) \cap N_G(v) = \emptyset$. Since $\deg_G(u) = \maxdeg(G)$, we obtain~$E(G) = \emptyset$.
\end{proof}

For two real intervals $[\ell,r]$ and $[\ell', r']$, we write  $[\ell,r] \preceq [\ell', r']$ if $\ell \leq \ell'$ and $r \leq r'$.
We say a vertex~$v$ of a graph $G$ is \emph{neighborhood-contained} if there is a vertex~$w \in V(G)$ such that~$N_G[v]\subsetneq N_G[w]$. Otherwise it is \emph{non-neighborhood-contained}.

\begin{lemma}\label{lem:order:of:max:deg:vertices}
	Let $G$ be a connected interval graph with an interval representation $\{I_v \colon v \in V(G) \}$. Let~$V'$ be a set of non-neighborhood-contained vertices.
	\begin{enumerate}[(i)]
		
		\item There is a linear order $v_1, v_2, \ldots, v_k$ on $V'$ such that for all $i \in [k-1]$:
  $I_{v_i} \preceq I_{v_{i+1}}$ or~$N_G[v_i] =N_G[v_{i+1}]$.  \label{item:are:ordered}

		\item Every linear order	obtained in this way from an interval representation of~$G$ is equivalent to~$v_1,v_2, \ldots,v_k$ or its reverse~$v_k,v_{k-1},\ldots,v_1$ up to permuting vertices from the same twin equivalence class.  \label{item:order:always same}
		
		 \item For every tidy interval representation $\{I'_v \colon v \in V(G) \}$ of~$G$, the linear order~$v_1,\dots,v_k$
		 satisfies $I'_{v_i} \preceq I'_{v_{i+1}}$ for each $i \in [k-1]$ or it satisfies~$I'_{v_i} \succeq I'_{v_{i+1}}$ for each $i \in [k-1]$.		  \label{item:tidy:is:oredered}
	\end{enumerate}
\end{lemma}
\begin{proof}
	For Part~\eqref{item:are:ordered}, we may assume without loss of generality that~$V'$ contains exactly one vertex from each twin equivalence class of non-neighborhood-contained vertices of~$G$.
	Then~$G[V']$ is connected. 
	By the choice of~$V'$, we know that $G[V']$ is a proper interval graph. 	Hence, the assumptions for Theorem~1 of~\cite{DBLP:journals/dam/Bang-JensenHI07} (originally stated in~\cite{Deng1996}) are satisfied and the result follows.
	
	For Part~\eqref{item:order:always same}, we may again assume without loss of generality that~$V'$ contains exactly one vertex from each twin equivalence class of non-neighborhood-contained vertices of~$G$. 	
	We will argue that if~$v_i,v_{i+1}, v_{i+2}$ are three consecutive vertices in the order from Part~\ref{item:are:ordered} then in every order obtained from an interval representation of~$G$, it is the case that~$v_{i+1}$ is between~$v_i$ and~$v_{i+2}$. This suffices to show Part~\eqref{item:order:always same}. 
	Note that by the choice of~$V'$ we know that~$\{v_i,v_{i+1}\}\in E(G)$ and~$\{v_{i+1},v_{i+2}\}\in E(G)$.
	We distinguish two cases. 
	\begin{itemize}
	\item 
	If~$\{v_i,v_{i+2}\}\notin E(G)$ then~$v_{i+1}$ has to be between~$v_i$ and~$v_{i+2}$, since these vertices are both adjacent to~$v_{i+1}$ but are represented by disjoint intervals.
	\item If~$\{v_i,v_{i+2}\}\in E(G)$ then there is a vertex~$x$ adjacent to~$v_i$ and~$v_{i+1}$ but not to~$v_{i+2}$ and there is a vertex~$y$ adjacent to~$v_{i+1}$ and~$v_{i+2}$ but not to~$v_{i}$. However, there is no vertex adjacent to~$v_{i}$ and~$v_{i+2}$ but not adjacent to~$v_i$, since~$N[v_{i+1}]\subseteq N[v_{i}]\cup N[v_{i+2}]$. 
	\end{itemize}
	Thus, being the vertex in the middle in the linear ordering is a combinatorial property and thus independent of the interval representation.
		
	For Part~\eqref{item:tidy:is:oredered}, we simply observe that for a tidy interval representation for every pair of twins~$v,v'$ we have $I_{v} = I_{v'}$.
\end{proof}
Note that the previous lemma in particular applies to the set of vertices of maximum degree.

\begin{definition}[Bordering vertex]
	Let $G$ be an interval graph. We call vertex $v \in \Vmaxdeg(G)$ \emph{bordering} if there is a linear order of $\Vmaxdeg(G)$ as in Lemma~\ref{lem:order:of:max:deg:vertices} such that~$v$ is extremal with respect to the ordering.
\end{definition}
Note that if a vertex is bordering then every false twin of the vertex is bordering. Also note that two maximum degree vertices that are twins are adjacent (i.e., they are false twins). We conclude that the set of bordering vertices consists of one or two equivalence classes of maximum degree vertices.

\subsection{Separations with linearly ordered neighborhoods}

We now introduce the concept of a clean clique separation along which we can reassemble the graph and which is crucial for our technique Reconstruction-by-Separation.

\begin{definition}[Clean clique separation] \label{def: clean-clique-separation}
	A \emph{separation} of a graph~$G$ is an ordered partition $(A,C,B)$ of $V(G)$ such that~$E_G(A,B) = \emptyset$.
	We say that the separation is a \emph{clean clique separation} if there is an interval representation 
	$\{[\ell_v, r_v ]\colon v\in V(G) \}$ of~$G$ such that 
		\[\left(\bigcap_{c \in C}[\ell_c, r_c]\right)\cap (\max_{a \in A} r_a, \min_{b \in B}\ell_b ) \neq \emptyset.\]
	
\end{definition}

We say that the separation $(A,C,B)$ has \emph{linearly ordered neighborhoods} if for every pair of vertices~$c_1$ and $c_2$ in $C$, we have~$N_G(c_1)\cap A \subseteq N_G(c_2)\cap A$ or~$N_G(c_2)\cap A \subseteq N_G(c_1)\cap A$ and, similarly, we also have that~$N_G(c_1)\cap B \subseteq N_G(c_2)\cap B$ or~$N_G(c_2)\cap B \subseteq N_G(c_1)\cap B$.

\begin{lemma}\label{lem: clean implies linearly ordered nbhds}
	If~$(A,C,B)$ is a clean clique separation of an interval graph,  
	then $(A,C,B)$ has linearly ordered neighborhoods.
\end{lemma}

\begin{proof}
	Let $(A, C, B)$ be a clean clique separation of an interval graph	$G$ and let $\{[\ell_v, r_v] \colon v \in V(G)\}$ be an interval representation of $G$ as in the definition of clean separations.  
	By definition, there exists a~$c^{\star} \in (\max_{c \in C}\ell_c, \min_{c\in C}r_c)\cap (\max_{a \in A} r_a, \min_{b \in B}\ell_b )$.
	By symmetry it suffices to show that neighborhoods are linearly ordered with respect to~$A$.
	Let~$c_1$ and~$c_2$ be vertices in~$C$.
	
	Suppose towards a contradiction that neither $N_G(c_1)\cap A \subseteq N_G(c_2)\cap A$ nor~$N_G(c_2)\cap A \subseteq N_G(c_1)\cap A$.
	This implies that there are vertices~$a_1$ and $a_2$ in $A$ such that~$a_1\in N_G(c_1)\setminus N_G(c_2)$ and~$a_2\in N_G(c_2)\setminus N_G(c_1)$. The vertices are non-adjacent since otherwise~$G[\{a_1,a_2,c_1,c_2\}]$ is an induced 4-cycle, which is a contradiction since every interval graph is chordal (see~Lemma~\ref{lem: conditions of interval graphs}).
	In particular, the intervals~$[\ell_{a_1}, r_{a_1}]$ and $[\ell_{a_2}, r_{a_2}]$ are disjoint, say, $\ell_{a_1} \leq r_{a_1} < \ell_{a_2}\leq r_{a_2}$.
	Since $\ell_{c_1} \leq r_{a_1} < r_{a_2} < c^{\star} \leq r_{c_1}$ (the second-to-last inequality holds since $(A,C, B)$ is clean) it follows that $c_1$ is adjacent to~$a_2$, which is a contradiction.
\end{proof}

\subsection{The sides of a vertex in an interval graph}

Let~$G$ be an interval graph and let~$S\subseteq V(G)$ be a set of vertices. We define an equivalence relation~$\approx_{G,S}$ on the vertices in~$V(G)\setminus N[S]$ as follows:
Define~$u\sim_{G} u'$ if~$N_G[u]\cap N_G[u']$ is non-empty.
We let~$\approx_{G,S}$ be the equivalence relation on~$V(G)\setminus N[S]$ that is  the transitive closure of~$\sim_{G}$. 

\begin{lemma} \label{lem: flanks}
	Let~$G$ be a connected interval graph  and let $\{[\ell_v, r_v]\colon v \in V(G) \}$ be an interval representation of~$G$. Further let $S\subseteq V(G)$ be a non-empty connected set of vertices.
	Set \[L_G(S) \coloneqq \{u \in V(G)\colon r_u < \min_{s \in S}\ell_s \} ~\text{and}~
	R_G(S) \coloneqq \{w \in V(G)\colon \ell_w > \max_{s \in S}r_s \}.\]
	\begin{enumerate}[(i)]
		\item Each two vertices of $L_G(S)$ are equivalent and each two vertices of $R_G(S)$ are equivalent with respect to $\approx_{G,S}$.
		In particular, there are at most two equivalence classes with respect to~$\approx_{G,S}$ on $V(G)\setminus N_G[S]$.
		\item \label{itm: intuitive-sides} If there exists a vertex $s \in S$ with $\deg_G(s) \geq \maxdeg(G)-1$, then both sets~$L_G(S)$ and~$R_G(S)$ are equivalence classes with respect to $\approx_{G,S}$.
		In particular, $L_G(S) \cap R_G(S) = \emptyset$.
	\end{enumerate}
\end{lemma}
\begin{proof}
	Since~$G[S]$ is connected, $v\in N[S]$ if and only if $I_v \cap [\min_{s\in S} \ell_s ,\max_{s \in S}r_s] \neq \emptyset$ and, hence, $V(G)\setminus N[S] = L_G(S) \cup R_G(S)$.
	For the first part, by symmetry, it suffices to show that the vertices in~$L_G(S)$ are equivalent with respect to~$\approx_{G,S}$. 
	Let~$u_1$ and~$u_2$ be in $L_G(S)$ and consider a shortest $u_1$-$u_2$~path~$P$ in $G$.
	Choose a vertex $x$ from $N_G[S]$ which minimizes $\ell_x$. 
	Note that $[\ell_x, r_x]$ covers all points~$a\in \bigcup \{I_v\colon v\in N_G[S]\}$ satisfying~$a<\min_{s \in S} \ell_s$. This implies that $P$ contains at most one vertex from~$N_G[S]$ (because if there were two, the subpath between them could be replaced by~$x$ contradicting that~$P$ is shortest). Hence, $u_1 \approx_{G,S} u_2$.
	
	For the second part, suppose there is a vertex $s \in S$ with $\deg_G(s) \geq \maxdeg(G)-1$.
	Towards a contradiction suppose that there exist vertices $u \in L_G(S)$ and $w \in R_G(S)$ with $u \sim_{G} w$.
	Since $u$ and $w$ cannot be adjacent there exists a vertex $x \in N_G[S]\setminus S$ which is adjacent to both~$u$ and~$w$.
	In particular, the interval $[\ell_x, r_x]$ satisfies
	$\ell_x \leq r_u < \ell_s \leq r_s < \ell_w \leq r_x$.
	It follows that every neighbor of $s$ is also a neighbor of $x$.
	Since $x$ also has the two neighbors $u$ and $w$ which are not in the neighborhood of $s$, we obtain $\deg_G(x) \geq \deg_G(s) + 2 \geq \maxdeg(G) + 1$, which is a contradiction. Altogether, no vertex of $R_G(S)$ is related to a vertex of $L_G(S)$ with respect to $\sim_{G}$ and, hence, each of the sets $L_G(S)$ and $R_G(S)$ is an equivalence classes with respect to~$\approx_{G,S}$.
\end{proof}

\begin{definition}[Sides of a vertex]
Let $G$ be an interval graph, $s\in V(G)$ with~$\deg(s)\geq \maxdeg(G)-1$, and let $G'$ be the component of $G$ containing $s$. The \emph{sides of~$s$} are the equivalence classes of~$\approx_{G',\{s\}}$.
\end{definition}

\subsection{The coating and the span of a vertex set}
	
\begin{definition}[Coating of a vertex set]
Let $G$ be a graph. For $S \subseteq V(G)$, we define the \emph{coating} of $S$, denoted by $\nucl_G(S)$, to be the set of vertices~$x$ of~$G$ for which there exist~$p_1$ and~$p_t$ in $S$ and an induced $p_1$-$p_t$ path~$P$ in~$G$ (possibly~$t=1$) with~$x \in V(P)$.
	
\end{definition}

We observe that if~$G$ is connected then for every set~$S$ the \emph{coating} of $S$ induces a connected subgraph: indeed, all vertices of~$S$ are in the same connected component since there is an induced path between every pair of vertices and all other vertices of the coating must be in the same connected component by construction.

\begin{lemma} \label{lem: paths-in-interval-graphs}
	Let $G$ be an interval graph with an interval representation $\{[\ell_v, r_v]\colon v \in V(G) \}$.
	\begin{enumerate}[(i)]
		\item \label{itm: ordered induced paths} If $v_1v_2\cdots v_k$ is an induced path in $G$ and $r_{v_1} \leq r_{v_k}$, then
		$\ell_{v_{i+1}} \leq r_{v_i} < \ell_{v_{i+2}}$ for all $i \in [k-2]$.
		\item \label{itm: max induced path} If $u$ and $v$ are distinct vertices in $\Vmaxdeg(G)$ and $w$ is a vertex of an induced $u$-$v$ path in $G$, then
		$[\ell_w, r_w] \subseteq [\min \{\ell_u, \ell_v \}, \max \{r_u, r_v\}]$. \label{itm: induced-paths-of-int-graphs-2}
		\item If $G$ is connected, then there exist $x$ and $y$ in $\Vmaxdeg(G)$ such that $\bigcup_{v \in \nucl(\Vmaxdeg(G))}[\ell_v, r_v] = [\ell_x, r_y]$.\label{itm: induced-paths-of-int-graphs-3}
	\end{enumerate}
\end{lemma}

\begin{proof}
	We prove Part~\eqref{itm: ordered induced paths} of the lemma.
		We have~$\ell_{v_{i+1}} \leq r_{v_i}$ since~$v_{i+1}$ and~$v_i$ are adjacent.  Towards the second inequality, observe that~$r_{v_1}\leq r_{v_2}\leq \cdots \leq r_{v_k}$: indeed, otherwise there is a~$j$ such that~$r_{v_j}> r_{v_{j+1}}$ and~$r_{v_{j+1}}<r_{v_{j+2}}$, but then~$v_j$ and~$v_{j+2}$ are adjacent. We conclude as follows. If~$r_{v_i} \geq \ell_{v_{i+2}}$, then~$\bigcup_{i+2\leq j\leq k} [\ell_{v_j}, r_{v_j}]$
		is a connected interval containing a point~$p\leq r_{v_i}$ and a point~$p'\geq r_{v_i}$ (for example~$p'=r_k$). Thus some~$v_j$ with~$j\geq i+2$ is adjacent to~$v_i$.
		
		For Part~\eqref{itm: max induced path}, let $u$ and $v$ be two maximum-degree vertices of~$G$, and let $P$ an induced $u$-$v$ path in~$G$. Let $w \in V(P)$.
		If $w \in \{u,v\}$, the statement follows immediately.
		Hence, assume $|V(P)| \geq 3$ and $w \notin \{u,v\}$.
		Since $P$ is induced, the intervals corresponding to $u$ and $v$ do not intersect.
		By symmetry, we may assume that $\ell_u \leq r_u < \ell_v \leq r_v$.
		
		If $w \notin N_G(u)$, then it follows from the first part of this lemma that $\ell_u\leq r_u< \ell_w$.
		If otherwise $w \in N_G(u)$, then $\ell_w \leq r_u < r_w$.
		Suppose towards a contradiction that $\ell_w < \ell_u$.
		Then $N_G(u) \subseteq N_G(w)$ and since $P$ is induced, $w$ has at least one more neighbor in $V(G)\setminus N_G(u)$, which is a contradiction to $\deg_G(u) = \maxdeg(G)$. We conclude~$\ell_w\geq \ell_u$. Symmetrically, we conclude~$r_w\leq r_v$.

		For Part~\eqref{itm: induced-paths-of-int-graphs-3}, note that by Part~\eqref{itm: induced-paths-of-int-graphs-2} for each~$v \in \nucl(\Vmaxdeg(G))$ there are vertices~$x_v$ and~$y_v$ with~$[\ell_v, r_v]\subseteq [\ell_{x_v}, r_{y_v}]$. By Lemma~\ref{lem:order:of:max:deg:vertices}, there are vertices~$x$ and~$y$ in $\Vmaxdeg(G)$ such that~$[\ell_x,r_y]$ contains all intervals~$[\ell_{x_v}, r_{y_v}]$.
		This shows $\bigcup_{v \in \nucl(\Vmaxdeg(G))}[\ell_v, r_v] \subseteq [\ell_x, r_y]$. For the other inclusion direction it suffices to recall that~$\nucl(\Vmaxdeg(G))$ induces a connected subgraph.
\end{proof}

We define the span of a set of vertices~$S$ to be the set obtained by forming the coating and then including all vertices which are neighborhood-contained by~$\nucl(S)$ as follows. 

\begin{definition}[Span]
Let $G$ be a graph.
For a subset~$S$ of $V(G)$, we define the \emph{span}  of~$S$ to be the set
\[ \Span_G(S) \coloneqq \{v \in N_G[\nucl_G(S)] \colon N_G[v]\subseteq N_G[\nucl_G(S)]\}.\] 

\end{definition}

We will drop the index~$G$ in the notation for the span and the coating if it is apparent from the context.
\begin{lemma}\label{lem:nucleus:and:span}
	Let $G$ be an interval graph and $S \subseteq V(G)$.
	\begin{enumerate}[(i)]
		\item  $S \subseteq \nucl_G(S) \subseteq \Span_G(S)$. \label{itm: nucl-in-span}
		\item $N_G[\Span_G(S)]= N_G[\nucl_G(S)]$. \label{itm: n-of-span-is-n-of-nucl}
		\item  If~$G$ is connected, then~$\Span_G(S)$ induces a connected subgraph of~$G$. \label{itm: span-is-connected}
		\item Suppose~$G$ is connected. If $\{[\ell_v, r_v ]\colon v\in V(G) \}$ is a tidy interval representation of~$G$ and~$S \subseteq \Vmaxdeg(G)$, then
		\label{itm:span-is-intuitive-span}\begin{equation*} \label{eq: span-is-span}
			\Span_G(S) = \{v \in V(G)\colon [\ell_v, r_v] \subseteq [\min_{s \in S}\ell_s, \max_{s \in S}r_s]\}.
		\end{equation*}
	\end{enumerate}
\end{lemma}

\begin{proof}
	We prove Part~\eqref{itm: nucl-in-span}.
	 We have $S \subseteq \nucl_G(S)$ by definition. If $v \in \nucl_G(S)$, then $N_G[v] \subseteq N_G[\nucl_G(S)]$, which yields the second inclusion of the statement.
	
	For Part~\eqref{itm: n-of-span-is-n-of-nucl}, observe that Part~\eqref{itm: nucl-in-span} implies $N_G[\nucl_G(S)] \subseteq N_G[\Span_G(S)]$.
		For the other direction observe that every vertex $v$ in $\Span_G(S)$ satisfies $N_G[v] \subseteq N_G[\nucl_G(S)]$ by definition.
		
	Now we prove Part~\eqref{itm: span-is-connected}. 
		By Part~\eqref{itm: nucl-in-span}, it holds that $\nucl_G(S)\subseteq \Span_G(S)$.
		Recall that~$\nucl_G(S)$ induces a connected subgraph of~$G$.
		If $s'$ is a vertex in $\Span_G(S)\setminus \nucl_G(S)$, then there is a neighbor $t'$ of $s'$ in $\nucl_G(S)$. Thus~$\Span_G(S)$ induces a connected subgraph.
		
		Finally, for Part~\eqref{itm:span-is-intuitive-span}, let
		$\{[\ell_v, r_v ]\colon v\in V(G) \}$ be a  tidy interval representation of~$G$.
		If $u \in \Span_G(S)$, then $N_G[u] \subseteq N_G[\nucl_G(S)]$ by definition. 
		Since the interval representation is tidy, it follows that $[\ell_u, r_u] \subseteq [\min_{v \in \nucl_G(S)}\ell_v, \max_{v \in \nucl_G(S)}r_v ]$.
		By Lemma~\ref{lem: paths-in-interval-graphs} Part~\eqref{itm: max induced path}, each $u' \in \nucl_G(S)$ satisfies $[\ell_{u'}, r_{u'}] \subseteq [\min_{s \in S}\ell_s, \max_{s \in S}r_s]$.
		Altogether, we obtain that
		$\Span_G(S) \subseteq  \{v \in V(G)\colon [\ell_v, r_v] \subseteq [\min_{s \in S}\ell_s, \max_{s \in S}r_s]\}$.
		
		It remains to show the other inclusion. For this, note that~$\nucl_G(S)$ is connected and thus $\bigcup_{v \in \nucl_G(S)}[\ell_v, r_v]$ contains~$[\min_{s \in S}\ell_s, \max_{s \in S}r_s]$. The statement now follows from the definition of the span.
\end{proof}

\begin{lemma}\label{lem:tilde:sets:give:clean:clique:sep}
	Let~$G$ be a connected interval graph,~$S\subseteq V(G)$, and $\{[\ell_v, r_v ]\colon v\in V(G) \}$ a tidy interval representation of $G$ such that 
	\[\min_{s \in S}\ell_s = \min_{v \in V(G)}\ell_v~~\text{or}~~\max_{s \in S}r_s =\max_{v \in V(G)}r_v.\]
	
	\begin{enumerate}[(i)]
	\item If $N_G[S] \neq V(G)$ and~$S= \{v \in V(G)\colon [\ell_v, r_v] \subseteq [\min_{s \in S}\ell_s, \max_{s \in S}r_s]\}$, then~$\partition_G(S)$ is a clean clique separation.\label{lem:item:no:tilde}
	\item If $N_G[\Span_G(S)] \neq V(G)$ and~$S \subseteq \Vmaxdeg(G)$, then~$\partition_G(\Span(S))$ is a clean clique separation.\label{lem:item:with:tilde}
	\end{enumerate}
\end{lemma}

\begin{proof}
We may assume that $\max_{s \in S}r_s =\max_{v \in V(G)}r_v$.
For the first part, we observe that all vertices in~$N_G(S)$ contain~$\min_{v \in S}\ell_v$.
It follows that the vertices in $N_G(S)$ form a clique and~$\partition_G(S)$ is a clean clique separation.

For the second part, 
	by Lemma~\ref{lem:nucleus:and:span} Part~\eqref{itm:span-is-intuitive-span}, we know that
	$\Span_G(S)  = \{v \in V(G)\colon [\ell_v, r_v] \subseteq [\min_{s \in S}\ell_s, \max_{s \in S}r_s]\}$.
	By definition, $E(\Span_G(S) , \NTildeTwo{G}{S} ) = \emptyset$.
	For $i \in \{1,2\}$, let $c_i$ be a vertex in $\NTildeOne{G}{S} $.
	Since $c_i \in \NTildeOne{G}{S} $ and $\max_{s \in S}r_s =\max_{v \in V(G)}r_v$, we obtain that $\ell_{c_i} < \min_{s \in S}\ell_s \leq r_{c_i} \leq \max_{s \in S}r_s$.
	In particular, the intervals corresponding to $c_1$ and $c_2$ intersect in $\min_{s \in S}\ell_s$. It follows that $\NTildeOne{G}{S} $ is a clique.
 \end{proof}

\subsection{Decomposing an interval graph into a bulk and up to two flanks}

In this subsection, we assume that~$G$ is a connected interval graph.

\begin{definition}
	We define the \emph{bulk} of~$G$ by setting
	\[\bulk(G)\coloneqq \Span_G(\Vmaxdeg(G)).\]
	We call the equivalence classes of~$\approx_{G,\bulk(G)}$ the \emph{flanks} of~$G$
	and denote by
	\[F_G\coloneqq V(G) \setminus N_G[\bulk(G)]\] the set of vertices in the flanks.
\end{definition}

Note that an interval graph has at most two flanks by Lemma~\ref{lem: flanks}.

\begin{lemma}\label{lem: bordering-means-flank-is-side}
	Let $F_1$ and $F_2$ be the two (possibly empty) flanks of $G$.
	For each $b \in \Vmaxdeg(G)$ the (possibly empty) sides of $b$ can be labeled $S^b_1$ and $S^b_2$ such that $F_1 \subseteq S^b_1$ and $F_2 \subseteq S^b_2$.
	For this labeling it holds for $i \in [2]$ that $F_i = \bigcap_{b \in \Vmaxdeg(G)}S^b_i$.
	
	Moreover, a vertex $v \in \Vmaxdeg(G)$ is bordering in $G$ if and only if $S_1^v = F_1$ or $S_2^v = F_2$.
	Conversely, every flank is a side of some bordering vertex.
\end{lemma}

\begin{proof}
	Consider a tidy interval representation $\{[\ell_v, r_v]\colon v \in V(G) \}$ of~$G$.
	For each~$b\in \Vmaxdeg(G)$, define~$S_1^b \coloneqq  \{u \in V(G)\colon r_u < \ell_b \}$ and
	$S_2^b \coloneqq \{w \in V(G)\colon \ell_w > r_b \}$. By Lemma~\ref{lem: flanks}, we know that~$S_1^b$ and~$S_2^b$ are the sides of~$b$.
	
	Furthermore, define~$F_1 \coloneqq \{u \in V(G)\colon r_u < \min_{s \in \Vmaxdeg(G)}\ell_s \}$ and
	$F_2\coloneqq \{w \in V(G)\colon \ell_w > \max_{s \in \Vmaxdeg(G)}r_s \}$. Again by Lemma~\ref{lem: flanks}, we know that~$F_1$ and~$F_2$ are the flanks of~$G$. We obtain that $F_i = \bigcap_{b \in \Vmaxdeg(G)}S^b_i$, showing the first part of the lemma.

	We also obtain that if~$b$ is bordering, then $S_1^b = F_1$ or $S_2^b = F_2$. On the other hand, if~$b$ is not bordering, then by Lemma~\ref{lem:order:of:max:deg:vertices}, there is a vertex~$v \in \Vmaxdeg(G)$ and a vertex~$x\in N(v)$ with~$r_x < \ell_b$. Thus~$x \in S_1^b \setminus S_1^v \subseteq S_1^b\setminus F_1$ (we use the first part of the statement to obtain the latter inclusion).
	Similarly there is a vertex~$x'\in  S_2^b\setminus F_2$, which shows that neither $S_1^b = F_1$ nor $S_2^b = F_2$.
	
	Finally, we obtain the last sentence of the lemma by considering vertices in~$\Vmaxdeg(G)$ that minimize~$\ell_v$ or maximize~$r_v$.
\end{proof}

\begin{lemma}\label{lem:side:versus:flank} 
	Let $G$ be a connected interval graph. Let~$v \in \Vmaxdeg(G)$ and let~$M$ be a side of $v$. The following are equivalent:
	\begin{enumerate}[(i)] 
		\item $M$ is a flank of $G$. \label{itm: one}
		\item $N_G[M] \cap \Vmaxdeg(G) = \emptyset$, and \label{itm: two}
		\item $v$ is bordering and no vertex in~$N_G[M]$ is bordering. \label{itm: three}
	\end{enumerate}
	
\end{lemma}
\begin{proof} (\eqref{itm: one} $\Rightarrow$ \eqref{itm: two})
	If~$M$ is a flank of~$G$, then~$N_G[M] \cap \bulk(G) =\emptyset=  M\cap N_G[\bulk(G)]$. Thus~$N_G[M]\cap \Vmaxdeg(G)= \emptyset$ since~$\Vmaxdeg(G)\subseteq \bulk(G)$.
	
	\medskip \noindent
	(\eqref{itm: two} $\Rightarrow$ \eqref{itm: three})
	We assume~$N_G[M] \cap \Vmaxdeg(G) = \emptyset$ and consider a tidy interval representation of~$G$.
	Suppose towards a contradiction that~$v$ is not bordering. Then every side of~$v$ contains a vertex adjacent to a vertex of~$\Vmaxdeg(G)$ (Lemma~\ref{lem: bordering-means-flank-is-side}), a contradiction. Vertices in~$M$ are not adjacent to bordering vertices since these are of maximum degree.

	\medskip
	\noindent
	(\eqref{itm: three} $\Rightarrow$ \eqref{itm: one})
	By Lemma~\ref{lem: flanks} Part~\eqref{itm: intuitive-sides}, the two sides of~$v$ are disjoint.
	If no vertex of~$N_G[M]$ is bordering, then by Lemma~\ref{lem:order:of:max:deg:vertices}, no vertex of~$N_G[M]$ has degree~$\Delta(G)$. It follows by Lemma~\ref{lem:nucleus:and:span} Part~\eqref{itm:span-is-intuitive-span} and the previous lemma that~$M\subseteq F_G$. Then, by Lemma~\ref{lem: bordering-means-flank-is-side}, we conclude that~$M$ is a flank of~$G$. 
\end{proof}

\begin{lemma} \label{lem: connected-vmaxdeg}
	Let $G$ be an interval graph.
	\begin{enumerate}[(i)]
		\item If $\Vmaxdeg(G)$ is connected, then $N[\bulk(G)] = N[\Vmaxdeg(G)]$.
		\item If $H$ is an induced subgraph of $G$ with $\maxdeg(H) = \maxdeg(G)$, then $N_H[\bulk(H)] \subseteq N_G[\bulk(G)]$.
		\item If $v \in F_G 
		$, then $N_{G-v}[\bulk(G-v)] = N_G[\bulk(G)]$. \label{itm:minus-flank-vertex}
	\end{enumerate}
\end{lemma}

\begin{proof}
	We prove the first statement.
	Assume that $\Vmaxdeg(G)$ is connected.
	Since $\Vmaxdeg(G) \subseteq \bulk(G)$, we obtain $N_G[\Vmaxdeg(G)] \subseteq N_G[\bulk(G)]$.
	We argue the other direction.
	For~$b\in \nucl_G(\Vmaxdeg(G))$, we show that~$N_G[b]\subseteq N_G[\Vmaxdeg(G)]$.
	Let $P$ be an induced $p_1$-$p_2$ path containing~$b$ as a vertex with~$p_1,p_2\in \Vmaxdeg(G)$. Consider an interval representation of~$G$.
	By Lemma~\ref{lem: paths-in-interval-graphs} Part~\eqref{itm: induced-paths-of-int-graphs-2}, we have that
	$[\ell_b, r_b] \subseteq [\min\{\ell_{p_1}, \ell_{p_2}\}, \max\{r_{p_1}, r_{p_2}\}]$. 
	This means that the interval of every neighbor~$x$ of~$b$ intersects $[\min\{\ell_{p_1}, \ell_{p_2}\}, \max\{r_{p_1}, r_{p_2}\}]$. Since~$G[\Vmaxdeg(G)]$ is connected, every point that is contained in $[\min\{\ell_{p_1}, \ell_{p_2}\}, \max\{r_{p_1}, r_{p_2}\}]$ is contained in some interval corresponding to a vertex of maximum degree. Thus~$b,x\in N_G[\Vmaxdeg(G)]$, which shows~$N_G[b]\subseteq N_G[\Vmaxdeg(G)]$.
	
	\medskip
	\noindent
	For the second statement let $H$ be an induced subgraph of $G$ with $\maxdeg(H) = \maxdeg(G)$.
	Since $H$ is induced, every induced path of $H$ is also an induced path of $G$.
	In particular $\nucl_H(\Vmaxdeg(H))\subseteq \nucl_G(\Vmaxdeg(G))$.
	Combining this with Lemma~\ref{lem:nucleus:and:span} Part~\eqref{itm: n-of-span-is-n-of-nucl}, we obtain 
	\[N_H[\bulk(H)] = N_H[\nucl_H(\Vmaxdeg(H))] \subseteq N_G[\nucl_G(\Vmaxdeg(G))]= N_G[\bulk(G)].\]
	
	\medskip
	\noindent
	For the third statement, the second part implies~$N_{G-v}[\bulk(G-v)] \subseteq N_G[\bulk(G)]$.
	We have $N_G[\bulk(G)] = N_G[\nucl_G(\Vmaxdeg(G))]$ by Lemma~\ref{lem:nucleus:and:span} Part~\eqref{itm: n-of-span-is-n-of-nucl}.
	So for~$x\in N_G[\bulk(G)]$ there exists~$y\in \nucl_G(\Vmaxdeg(G))$ adjacent to~$x$ or $x \in \nucl_G(\Vmaxdeg(G))$. In the former case~$y\in \nucl_{G-v}(\Vmaxdeg(G-v))$ since~$\nucl_{G}(\Vmaxdeg(G))= \nucl_{G-v}(\Vmaxdeg(G-v))$ and~$v \in F_G$.
	Hence,~$x\in N_{G-v}[\bulk(G-v)]$.
\end{proof}

\section{Reconstruction from annotated induced subgraphs}\label{sec:annotated:subgraphs}

Our main technique to reconstruct a given interval graph is to find a clean clique separation that decomposes the graph into two parts.
We determine the parts together with information on the vertices' original degrees and argue that under certain conditions, which turn out to be often satisfied in interval graphs, we can reconstruct the original graph.

An \emph{annotated graph} is a pair~$(H,\lambda)$, consisting of a graph~$H$ and a map~$\lambda\colon V(H)\rightarrow \mathbb{N}\cup \{\bot\}$. An \emph{isomorphism} between two annotated graphs~$(H_1,\lambda_1)$ and~$(H_2,\lambda_2)$ is a graph isomorphism~$\varphi$ from~$H_1$ to~$H_2$ with~$\lambda_2(\varphi(h))= \lambda_1(h)$.
An \emph{annotated induced subgraph} of a graph~$G$ is an annotated graph~$(H,\lambda)$ with~$H$ an induced subgraph of~$G$ where~$\lambda$ satisfies for all~$h\in V(H)$ that~$\lambda(h) = d_G(h)$ if~$\lambda(h) \neq \bot$ and~$d_H(h)=d_G(h)$ if~$\lambda(h) =\bot$.  
The interpretation is that vertices can only be annotated with a number that is their original degree in the graph~$G$. However, vertices that have the same degree in~$H$ and~$G$ do not necessarily have to be annotated.
Slightly abusing notation we will denote by~$G[V']$ the annotated graph induced by a set of vertices~$V'\subseteq V(G)$. The definition of the map~$\lambda$ will always be clear from the context.

Every clean clique separation~$(A,C,B)$ in a graph~$G$ gives rise to two annotated induced subgraphs $H_X \coloneqq G[X\cup C]$ with annotations $\lambda_X$ for $X \in \{A,B\}$ with
\begin{align*}
	\lambda_X(h)= \begin{cases}
		\bot &\text{if ~$h\notin C$ and}\\
		\deg_G(h) &\text{otherwise.}
	\end{cases}
\end{align*}
 In particular, we can recover from each of the graphs $H_A$ and $H_B$ which vertices belong to~$C$ and what their degree in $G$ is.

\begin{lemma}[Reconstruction-by-Separation]\label{lem:reconst:clique:sep:implies:reconst}
If~$(A,C,B)$ and~$(A',C',B')$ are clean clique separations of two graphs~$G$ and~$G'$, respectively, such that~$H_A\cong H_{A'}$ and~$H_B\cong H_{B'}$ as annotated graphs, then~$G\cong G'$.
\end{lemma}

\begin{proof}
By Lemma~\ref{lem: clean implies linearly ordered nbhds} the clean clique separations have linearly ordered neighborhoods.

Let~$\varphi_A\colon V(H_A) \rightarrow V(H_A')$ and~$\varphi_B\colon V(H_B) \rightarrow V(H_B')$ be isomorphisms respecting the annotations.
In particular,~$\varphi_A(C)= C'$ and~$\varphi_B(C)=C'$.

Consider the number of vertices~$v\in C$ for which~$\varphi_A(v)\neq \varphi_B(v)$. For the case in which this number is larger than zero, we explain how to alter~$\varphi_B$ to reduce this number by at least one:
choose $c \in C$ with~$\varphi_A(c)\neq \varphi_B(c)$ and set~$\widetilde\varphi_B\colon V(H_B) \rightarrow V(H_B')$ with \[\widetilde\varphi_B(v) = \begin{cases}
\varphi_A(c) & \text{if } v=c\\
\varphi_B(c) & \text{if } \varphi_B(v)= \varphi_A(c)\\
\varphi_B(v) & \text{otherwise.}
\end{cases}\]

It is easy to see that the number of vertices of~$C$ for which~
$\varphi_A(v) \neq \widetilde\varphi_B(v)$ is smaller than the respective number for~$\varphi_A$ and~$\varphi_B$.
We need to argue that~$\widetilde\varphi_B$ is an annotation-respecting isomorphism. The annotation is respected since~$\deg_{G'}(\varphi_A(c)) = \deg_G(c) = \deg_{G'}(\varphi_{B}(c))$. To argue that the new map~$\widetilde\varphi_B$ is an isomorphism it suffices to argue that~$\varphi_B(c)$ and~$\widetilde\varphi_B(c)= \varphi_A(c)$ are twins. Indeed, both vertices have the same degree~$d$ in~$G$, since the maps~$\varphi_B$ and~$\widetilde\varphi_B$ respect annotations. Moreover~$\deg_{H_{B'}}(\varphi_B(c))= \deg_{H_B}(c) = d-\deg_{G[A]}(c)+|C|-1= d-\deg_{G[A']}(\varphi(c))+|C|-1= \deg_{H_{B'}}(\widetilde\varphi_{B}(c))$.
Since neighborhoods are linearly ordered it follows that~$\varphi_B(c)$ and~$\widetilde\varphi_B(c)$ are twins in $H_B$.

Finally, assume that the number of vertices~$v\in C$ for which~$\varphi_A(v)\neq \varphi_B(v)$ is zero.
Then~$\varphi_A$ and~$\varphi_B$ agree on~$C$ and their common extension to~$V(G)$ is an isomorphism.
\end{proof}

Our main tool to show that interval graphs are reconstructible will be to determine (up to isomorphism) annotated graphs~$H_A$ and~$H_B$ which provably arise as the two subgraphs of a clean clique separation~$(A,C,B)$ of~$G$. This implies that they have linearly ordered neighborhoods. By the previous lemma this implies that the graph is reconstructible. The challenge is of course to determine~$C$ and the corresponding annotated subgraphs from the deck.

A reoccurring strategy will be to locate a card~$G_v$ of a vertex~$v$ in~$B$ that is ``as far away as possible'' from~$C$. 
If~$v$ is not the only vertex in~$B$ and we can locate~$A$ and~$C$ in the card, then there is a method to recover~$H_A$ with annotation as follows.

\begin{lemma}(Distant Vertex Lemma)\label{lemma:farthest:away:and:at:least:two:reconstruct:other:side}
For~$j\in \{1,2\}$, let~$G_j$ be 
an interval graph with a clean clique separation~$(A_j,C_j,B_j)$ with $|B_j|\geq2$.
Let~$b_j\in B_j$ be a vertex such that $G_j-b_j$ is connected and $N_{G_j}(b_j) \cap C_j$ is as small as possible under this condition.  
If~$\leftmset d_{G_1}(u)\colon u\in N_{G_1}(b_1)\rightmset=\leftmset d_{G_2}(u)\colon u\in N_{G_2}(b_2)\rightmset$ and
there is an isomorphism~$\varphi$ from~$G_1-b_1$ to~$G_2-b_2$ with~$\varphi(A_1)= A_2$ and~$\varphi(C_1)= C_2$,
then~$H_{A_1} \cong H_{A_2}$ as annotated subgraphs of~$G_1$ and~$G_2$, respectively.
\end{lemma}
\begin{proof}
Let~$\varphi$ be an isomorphism from~$G_1-b_1$ to~$G_2-b_2$
with~$\varphi(A_1)= A_2$ and~$\varphi(C_1)= C_2$.
If we ignore the annotations, then~$\varphi|_{A_1\cup C_1}\colon H_{A_1} \to H_{A_2}$ is an isomorphism.
If~$\varphi|_{A_1\cup C_1}$ is not an isomorphism of the annotated graphs $H_{A_1}$ and~$H_{A_2}$, then there is a vertex~$c\in C_1$ such that
the annotations of~$c$ and~$\varphi|_{A_1\cup C_1}(c)$ disagree.
This implies that~$c\in N_{G_1}(b_1)$ and~$\varphi|_{A_1\cup C_1}(c)\notin N_{G_2}(b_2)$ or that~$c\notin N_{G_1}(b_1)$ and~$\varphi|_{A_1\cup C_1}(c) \in N_{G_2}(b_2)$. By symmetry, we can assume the former.

Note that~$\deg_{G_1-b_1}(c)=\deg_{G_2-b_2}(\varphi|_{A_1\cup C_1}(c))$.

For~$i\in \{1,2\}$, let~$X_i\subseteq A_i\cup C_i$ be the set of vertices~$x$ for which~$B_i\setminus \{b_i\}\subseteq N_{G_i}(x)$. Since~$\varphi$ is an isomorphism,~$\varphi(X_1)=X_2$. Furthermore, for~$i\in \{1,2\}$, since~$b_i$ is chosen so that $N_{G_i}(b_i) \cap C_i$ is minimal, every vertex of~$C_i$ adjacent to~$b_i$ is adjacent to all vertices of~$B_i$, so~$N_{G_i}(b_i) \subseteq X_i$. 
Moreover, since~$|B_i|\geq 2$, we know that~$X_i\subseteq C_i$.

In particular, we have~$c\in X_1$ and~$\varphi|_{A_1\cup C_1}(c)\in X_2$.

Since~$\leftmset d_{G_1}(u)\colon u\in N_{G_1}(b_1)\rightmset=\leftmset d_{G_2}(u)\colon u\in N_{G_2}(b_2)\rightmset$, there is a vertex~$c'\in X_1$ with~$B_1 \setminus \{b_1\}\subseteq N_{G_1}(c')$ and with~$\deg_{G_1-b_1}(c')=\deg_{G_1-b_1}(c)$ and such that~$c'\notin N_{G_1}(b_1)$ and~$\varphi|_{A_1\cup C_1}(c')\in N_{G_2}(b_2)$.
We alter~$\varphi|_{A_1\cup C_1}$ by mapping~$c$ to~$\varphi|_{A_1\cup C_1}(c')$ and~$c'$ to~$\varphi|_{A_1\cup C_1}(c)$.
Since $(A_1, C_1, B_1)$ and $(A_2, C_2, B_2)$ have linearly ordered neighborhoods by Lemma~\ref{lem: clean implies linearly ordered nbhds}, the vertices~$x\in C_1$ and~$c'\in C_1$ are twins in~$G_1-b_1$ and the vertices~$\varphi|_{A_1\cup C_1}(x)\in C_2$ and~$\varphi|_{A_1\cup C_1}(c')\in C_2$ are twins in~$G_2-b_2$.

This decreases the number of vertices of which the isomorphism~$\varphi|_{A_1\cup C_1}$ does not respect  the annotation. By induction on this number there is an annotation-respecting isomorphism from~$H_{A_1}$ to~$H_{A_2}$.
\end{proof}

\section{Reconstruction of the flank sizes}\label{sec:flank:sizes}

In this section, we argue that the sizes of the flanks are reconstructible.

\begin{lemma}\label{lem:no:flanks:can:be:decided}
The property whether $G$ is an interval graph with~$F_G = \emptyset$ is recognizable.
\end{lemma}

\begin{proof}
Von Rimscha~\cite{DBLP:journals/dm/Rimscha83}
proved that being an interval graph is recognizable and disconnected graphs are reconstructible~\cite{Kelly57}.
Hence, we can assume that~$G$ is a connected interval graph.
We can reconstruct whether $N_G[\Vmaxdeg(G)] = V(G)$ by Lemma~\ref{lem: reconstruct-degree-sequence}. 
In this case~$N_{G}[\bulk(G)] = V(G)$ and, hence,~$|F_G| = 0$.
We can thus assume that $\NTwo{G}{\Vmaxdeg(G)} \neq \emptyset$.
Set
\[\mathcal{J} \coloneqq \leftmset G_v \colon v \in \NTwo{G}{\Vmaxdeg(G)} \rightmset.\]

Note that~$\mathcal{J}$ is non-empty and reconstructible.\footnote{The slightly informal statement ``$\mathcal{J}$ is reconstructible'' would be recast by our formal definition given in Section~\ref{sec:prelims} to say that the map that associates with every graph~$G$ the multiset~$\mathcal{J}(G) \coloneqq \leftmset G_v \colon v \in \NTwo{G}{\Vmaxdeg(G)} \rightmset$ of unlabeled graphs is reconstructible, meaning it satisfies that for pairs of graphs~$G$ and~$G'$ with the same deck we have~$\mathcal{J}(G)=\mathcal{J}(G')$. We will continue to use informal statements from here onwards.} 
Since~$G[\Vmaxdeg(G)] \cong G_v[\Vmaxdeg(G)]$ for every~$G_v\in \mathcal{J}$, the graph~$G[\Vmaxdeg(G)]$ is reconstructible (up to isomorphism).
If~$G[\Vmaxdeg(G)]$ is connected, then~$F_G \neq \emptyset$ by Lemma~\ref{lem: connected-vmaxdeg}.
We can thus assume that~$G[\Vmaxdeg(G)]$ is not connected.
Set
\[\mathcal{K} \coloneqq \leftmset G_v \colon\maxdeg(G_v) = \maxdeg(G) \rightmset.\]
Since we assume that~$G[\Vmaxdeg(G)]$ is not connected, every vertex $v \in \Vmaxdeg(G)$ satisfies $G_v \in \mathcal{K}$. In particular, $\mathcal{K} \neq \emptyset$.

\begin{claim} \label{claim:K}
	$F_G = \emptyset$ if and only if there is no card in~$\mathcal{K}$ in which
	every vertex of maximum degree has two non-empty sides.
\end{claim}

\noindent $\ulcorner$ 
($\Rightarrow$) 
Assume that~$F_G= \emptyset$.
Let~$u_1$ and~$u_2$ be two bordering vertices of maximum degree that are not twins. Note that~$u_1$ and~$u_2$ exist and are not adjacent since~$\mathcal{K}\neq  \emptyset$. Also note that $u_1$ and~$u_2$ both have only one non-empty side by Lemma~\ref{lem: bordering-means-flank-is-side}.

A vertex~$v$ for which~$G_v\in \mathcal{K}$ is not adjacent to both~$u_1$ and~$u_2$, since otherwise it would be adjacent to all maximum degree vertices. However, if~$v$ is not adjacent to~$u_i$, then in~$G_v$ the vertex~$u_i$ only has one non-empty side.

\noindent
($\Leftarrow$) Assume that~$F_G \neq \emptyset$.
In particular, there exists at least one flank $F_1 \neq \emptyset$ of $G$.
By Lemma~\ref{lem: bordering-means-flank-is-side}, there exists a bordering vertex $u_1$ in $\Vmaxdeg(G)$ such that $F_1$ is a side of $u_1$.

Since $\mathcal{K} \neq \emptyset$ there exists a bordering vertex~$u_2$ which is not a twin of~$u_1$ and not adjacent to~$u_1$.
In particular, $u_2$ is contained in a side $S_2$ of $u_1$ and $S_2 \neq F_1$ since vertices in the flank are not of maximum degree.

If~$u_2$ has a twin, then in $G_{u_2}$ every vertex of maximum degree has two sides (one containing~$F_1$ and the other one containing a twin of $u_2$). 
So we can assume that~$u_2$ does not have a twin.

Since graphs of maximum degree at most~2 are reconstructible, we can assume~$\deg_G(u_2) \geq 3$.
By Lemma~\ref{lem:side:versus:flank}, in $G$, exactly one side of~$u_2$ contains other vertices of maximum degree. We denote this side by~$L$.
Let~$x_1,x_2,\ldots,x_t$ be an ordering of the vertices in~$N(u_2)$ so that~$N(x_i) \cap L \subseteq N(x_{i+1}) \cap L$ (such an ordering exists since $G$ is an interval graph).

We prove that~$G_{x_2}$ is in $\mathcal{K}$: If~$x_2 \in \Vmaxdeg(G)$, then~$G_{x_2}$ is in~$\mathcal{K}$ since $\Vmaxdeg(G)$ is not connected. If~$x_2 \in V(G)\setminus \Vmaxdeg(G)$, then~$N_G(x_2) \cap L = \emptyset$, because otherwise all neighbors of~$u_2$ except possibly~$x_1$ have a neighbor in~$L$ and are adjacent to~$u_2$, which in particular means they are all  adjacent and have maximum degree (including~$x_2$). Thus also in this case $G_{x_2}$ is in~$\mathcal{K}$.

It remains to prove that every vertex in $\Vmaxdeg(G_{x_2})$ has two sides in $G_{x_2}$.
Note that~$x_1$ has no neighbors in~$L$ since otherwise every neighbor of~$u_2$ has a neighbor in~$L$, making all neighbors of~$u_2$ pairwise adjacent, which contradicts~$u_2 \in \Vmaxdeg(G)$.
For every~$s \in \Vmaxdeg(G_{x_2})$ the vertex $x_1$ and the flank $F_1$ lie in different sides of~$s$.  \hfill~$\lrcorner$

Since $\mathcal{K}$ is reconstructible by Lemma~\ref{lem: reconstruct-degree-sequence}, the statement follows by Claim~\ref{claim:K}.
\end{proof}

Set
\[\mathcal{E} \coloneqq \leftmset G_v \in \mathcal{D}(G) \colon G_v~\text{is connected},  \maxdeg(G_v) = \maxdeg(G), \lvert N_{G_v}[\bulk(G_v)]\rvert= \lvert N_G[\bulk(G)]\rvert   \rightmset.\]

\begin{lemma}\label{lem: E-recognizable}
The multiset of cards~$\mathcal{E}$ is reconstructible and $\{v \in V(G)\colon G_v \in \mathcal{E}\}\subseteq F_G$.
\end{lemma}
\begin{proof}
Clearly, we can identify the connected cards.
Since the degree sequence is reconstructible we can identify the cards $G_v$ satisfying $\Vmaxdeg(G) = \Vmaxdeg(G_v)$. 

By Lemma~\ref{lem: connected-vmaxdeg}, for a vertex~$v\in F_G$ we have $|N[\bulk(G-v)]|= |N[\bulk(G)]|$ but for a vertex~$v\in V(G)\setminus F_G$ which satisfies $\maxdeg(G_v) = \maxdeg(G)$, we have $|N[\bulk(G-v)]|< |N[\bulk(G)]|$.
Thus $\{v \in V(G)\colon G_v \in \mathcal{E} \} \subseteq F_G$.
In particular, if $F_G = \emptyset$, then $\mathcal{E} = \emptyset$.
If otherwise~$|F_G|>0$, then we can reconstruct~$|N[\bulk(G)]|$ from the deck: it is the maximum~$|N[\bulk(G-v)]|$ among all~$v$ for which the card~$G_v$ satisfies $\maxdeg(G_v) = \maxdeg(G)$.
Since the property~$|F_G|=0$ is reconstructible (Lemma~\ref{lem:no:flanks:can:be:decided}),
we conclude that~$\mathcal{E}$ is reconstructible.
\end{proof}

We now argue that from the set~$\mathcal{E}$ we can recover information about the flanks.

For $i \in \{1,2\}$, we set $\neg i \coloneqq 3-i$.

\begin{lemma}\label{lem:each:flank:represented:in:E}
If~$|F_G|>0$ and ~$F_1$ and~$F_2$ are the two flanks of~$G$ (with one of them possibly the empty set), then for each $i \in \{1,2\}$ the following statements hold. 

\begin{enumerate}
\item \label{itm:flank-sizes} For each~$G_v\in \mathcal{E}$, if~$v\in F_i$, then the flanks of~$G-v$ are~$F_i\setminus\{v\}$ and~$F_{\neg i}$.
\item \label{itm: fi-not-zero-card-in-e-in-flank} If~$|F_i|>0$, then there is a card~$G_v\in \mathcal{E}$ with~$v\in F_i$.
\end{enumerate}
\end{lemma}
\begin{proof}

For Part 1, if~$G_v\in \mathcal{E}$, then the flanks of~$G-v$ are subsets of the flanks of~$G$, which shows the first property.

For Part 2, if flank~$F_i$ has only one vertex~$x$ then~$G_x$ is connected and thus in~$\mathcal{E}$.
If a flank has at least two vertices then, being an interval graph, the induced subgraph~$G[F_i]$ contains (at least) two simplicial vertices~$x$ and~$y$. Then one of the two cards~$G_{x}$ or~$G_{y}$ is connected and thus in~$\mathcal{E}$. 
\end{proof}

\begin{lemma}\label{lem:sizes}
If~$F_1$ and~$F_2$ are the two (possibly empty) flanks of~$G$, then~$\leftmset |F_1|,|F_2|\rightmset$ is reconstructible.
\end{lemma}

\begin{proof}
By Lemma~\ref{lem:no:flanks:can:be:decided}, we can assume that~$\leftmset |F_1|,|F_2|\rightmset \neq \leftmset 0,0\rightmset$ and, hence that $\mathcal{E} \neq \emptyset$ by Lemma~\ref{lem:each:flank:represented:in:E} Part~\eqref{itm: fi-not-zero-card-in-e-in-flank}.
Since $\mathcal{E}$ is reconstructible (Lemma~\ref{lem: E-recognizable}), we obtain from Lemma~\ref{lem:each:flank:represented:in:E} Part~\eqref{itm:flank-sizes} that~$|F_1|+|F_2|$ is reconstructible. 
We distinguish cases according to the value of $|F_1|+|F_2|$.

\medskip
\noindent
(Case $|F_1|+|F_2|\geq 3$) In this case, by the previous lemma, there is a card in~$\mathcal{E}$ for which~$G_v$ has two flanks if and only if~$|F_1| > 0$ and~$|F_2|> 0$. Thus we can determine whether~$|F_1| > 0$ and~$|F_2|>0$ from the deck.
If $|F_1|= 0$ or~$|F_2| = 0$  then~$\leftmset |F_1|,|F_2|\rightmset = \leftmset 0,|F_1|+|F_2|\rightmset$ and the multiset is thus reconstructible.
Otherwise, if~$|F_1|> 0$ and~$|F_2| > 0$ then by the previous lemma~$\max\{|F_1|,|F_2|\}$ is the size of the largest flank in~$G_v$ over all choices of~$G_v\in \mathcal{E}$. Since the sum and the maximum of~$|F_1|$ and~$|F_2|$ are reconstructible we conclude that~$\leftmset |F_1|,|F_2|\rightmset$ is reconstructible.

(Case $|F_1|+|F_2|= 2$) 
In this case, by the previous lemma, there are at most two cards in~$\mathcal{E}$.
Furthermore, if there is only one card in~$\mathcal{E}$ then we are in the case~$\leftmset |F_1|,|F_2|\rightmset = \leftmset 0,2\rightmset$.
We can thus assume otherwise.  Then~$\mathcal{E}$ is exactly the set~$\{G_{m},G_{m'}\}$ where~$\{m,m'\} = F_1\cup F_2$. This implies that we can reconstruct~$\leftmset \deg(m),\deg(m')\rightmset$ from the deck.
Assume w.l.o.g. that~$\deg(m)\leq \deg(m')$. Then~$G_{m'}$ has exactly one flank, and this flank contains exactly one vertex, namely~$m$. If~$m$ and~$m'$ are adjacent in~$G$, this vertex has degree~$\deg(m)-1$. 
However, if~$m$ and~$m'$ are not-adjacent in~$G$ then no card~$G_v\in \mathcal{E}$  has a flank vertex of degree~$\deg(m)-1$.
We can thus determine whether~$m$ and~$m'$ are adjacent in~$G$. 

If~$m$ and~$m'$ are adjacent then~$\leftmset |F_1|,|F_2|\rightmset = \leftmset 0,2\rightmset$ and we can thus further assume that~$m$ and~$m'$ are not adjacent.
Under this assumption,~$m$ and~$m'$ are in the same flank of~$G$  exactly if~$m$ and~$m'$ have a common neighbor.
If they do, all vertices in~$N(m)$ and all vertices in~$ N(m')$ are adjacent. Moreover~$N(m)\subseteq N(m')$.

We can determine the degree sequence~$s$ of the neighborhood of~$m'$. Indeed, if~$N(m)= N(m')$ then the neighborhoods of~$m$ and~$m'$ have the same degree sequence, otherwise~$m'$ and~$m$ do not have the same degree and we can identify the card~$G_{m'}$ and thereby determine the degree sequence of the neighborhood of~$m'$.

It follows now that the vertices~$m$ and~$m'$ have a common neighbor exactly if there is no card in~$\{G_{m'},G_m\}$ in which the unique flank vertex has a neighborhood with a degree sequence equal to~$s$.

Overall, we can determine whether~$m$ and~$m'$ belong to the same flank and thus whether $\leftmset |F_1|,|F_2|\rightmset$ is equal to~$ \leftmset 0,2\rightmset$ or to~$\leftmset 1,1\rightmset$.

(Case $|F_1|+|F_2|= 1$) In this case, we have~$\leftmset |F_1|,|F_2|\rightmset = \{0,1\}$.
 
(Case $|F_1|+|F_2|= 0$)  This case is equivalent to $\leftmset |F_1|,|F_2|\rightmset = \leftmset 0,0\rightmset$ which was already discussed at the beginning of the proof.
\end{proof}

In the following subsections, we distinguish three cases according to the number of flanks of the graph.

\section[Two flanks (|F1|>0 and |F2|>0)]{ Two flanks ($|F_1|>0$ and~$|F_2|>0$)}
\label{sec: two-flanks}

We assume in this section that~$G$ is a connected interval graph with two non-empty flanks~$F_1$ and~$F_2$.
For~$i \in \{1,2\}$, we 
 define~$R_i\coloneqq N_G[\bulk(G)]\cap N_G[F_i]$ and $B_i \coloneqq V(G)\setminus (F_i \cup R_i)$.
 Note that~$\partition_G(F_i) = (F_i,R_i,B_i)$.
\begin{lemma}\label{lem:flanks:create:clean:clique:sep}
The sets~$R_1$ and~$R_2$ are disjoint and each of them is a clique in~$G$. Furthermore, for $i \in \{1,2\}$ no vertex in~$R_i$ has a neighbor in~$F_{\neg i}$.
In particular, for $i \in \{1,2\}$ the triple~$\partition_G(F_i)= (F_i,R_i,B_i)$  is a clean clique separation with $F_i\subseteq B_{\neg i}$.
\end{lemma}
\begin{proof} The second statement follows from the fact that a vertex with a neighbor in both~$F_1$ and~$F_2$ would imply that $F_1 \cup F_2$ forms one equivalence class with respect to $\approx_{G, \bulk(G)}$ contradicting the assumption of this chapter.
Vertices in $R_i$ all have two common neighbors which are not adjacent, namely one in~$F_i$ and a vertex of maximum degree in~$\bulk(G)$ by Lemma~\ref{lem: paths-in-interval-graphs}.
 It follows that they form a clique.
\end{proof}

For each~$i\in \{1,2\}$ the separation~$\partition_G(F_i)$ gives rise to the two annotated graphs~$H_{F_i}\coloneqq G[F_i\cup R_i]$ and~$H_{B_i} \coloneqq G[B_i\cup R_i]$.
Set
$\mathcal{E}_i \coloneqq \{G_v \in \mathcal{E}\colon v \in F_i \}$.
Observe that $\{\mathcal{E}_1, \mathcal{E}_2\}$ is a partition of $\mathcal{E}$ by Lemma~\ref{lem: E-recognizable}.

\begin{lemma}\label{lem:HFi:reconstr}
The multiset~$\leftmset H_{F_1},H_{F_2}\rightmset$ of annotated induced subgraphs is reconstructible.
\end{lemma}

\begin{proof}
Our overall strategy to determine $\leftmset H_{F_1},H_{F_2} \rightmset$ is to consider the flanks of all cards in~$\mathcal{E}$.
According to Lemma~\ref{lem:each:flank:represented:in:E}, for each~$i\in \{1,2\}$ there is a card~$G_v$ in~$\mathcal{E}$ with~$v\in F_{\neg i}$ such that the flanks of~$G-v$ are~$F_i$ and~$F_{\neg i}\setminus \{v\}$.

If $|F_i|\neq |F_{\neg i}| -1$, then we can determine which flank $F'$ of~$G_v$ corresponds to~$F_i$ and conclude~$H_{F_i} \cong G_v[N_{G_v}[F']]$. The graphs $H_{F_i}$ and $G_v[N_{G_v}[F']]$ are isomorphic as annotated induced subgraphs since~$N_G[F_i] \cap N_G[F_{\neg i}] = \emptyset$, and thus in particular~$N_G[F_i] \cap N_G[v] = \emptyset$.

It remains to argue the case~$|F_i|=|F_{\neg i}|-1$.
Note that $\mathcal{E}_{\neg i}$ is reconstructible since it contains exactly the cards in~$\mathcal{E}$ with equally sized flanks.
In each card~$G_v\in \mathcal{E}_{\neg i}$ there are two flanks~$F'$ and~$F''$ giving us two graphs~$H_{F'}$ and~$H_{F''}$. One of these graphs is isomorphic to~$H_{F_i}$ while the other one isomorphic to~$H_{ F_{\neg i}}-v$.
However, we already know that~$H_{ F_{\neg i}}$ is reconstructible (since~$|F_{\neg i}|\neq |F_i|+1$) and we can determine all graphs~$H_{ F_{\neg i}}-v$ with~$G_v\in \mathcal{E}_{\neg i}$. For this, we only need to observe that for a vertex~$v\in F_{\neg i}$ we have~$G_v\in \mathcal{E}_{\neg i}$ exactly if $H_{ F_{\neg i}}-v$ is connected, since~$H_{ F_{\neg i}}-v$ is connected if and only if~$G_v$ is connected and the other properties of the cards in~$\mathcal{E}$ are maintained since $v \in V(G)\setminus N_G[\bulk(G)]$.
If for some~$v\in \mathcal{E}_{\neg i}$ the graph $H_{ F_{\neg i}}-v$ is isomorphic to~$H_{ F_{i}}$, then~$H_{ F_{i}}\cong H_{F'}\cong H_{F''}$ is reconstructible. Otherwise~$H_{ F_i}$ cannot  appear as a graph $H_{ F_{\neg i}}-v$ for some~$v\in \mathcal{E}_{\neg i}$, and thus out of the two options $H_{F'},H_{F''}$ we know that $H_{ F_i}$ is the graph not appearing as $H_{ F_{\neg i}}-v$ for some~$v\in \mathcal{E}_{\neg i}$.
\end{proof}

\begin{lemma}\label{lem:two:flanks:reconstruc}
If~$|F_1|>0$ and~$|F_2|>0$, then~$G$ is reconstructible.
\end{lemma}
\begin{proof}
By symmetry, we can assume that~$|F_1|\leq |F_2|$.
\paragraph{Case 1:~$\mathbf{1 < |F_1| <|F_2|}$.}\leavevmode
In this case~$\mathcal{E}_1$ is reconstructible by Lemma~\ref{lem: E-recognizable} and Lemma~\ref{lem:each:flank:represented:in:E}.
Furthermore, by~Lemma~\ref{lem:flanks:create:clean:clique:sep}, we know that~$\partition_G(F_1)= (F_1,R_1,B_1)$ is a clean clique separation.
Let~$G'$ be a graph with $\mathcal{D}(G) = \mathcal{D}(G')$.
By Lemma~\ref{lem:sizes}, the graph~$G'$ has flanks~$F_1'$ and~$F_2'$ of sizes~$|F_1|$ and~$|F_2|$, respectively. There is thus an analogous separation~$(F_1',R_1',B_1')$ in~$G'$.

We show that~$H_{B_1}\cong H_{B_1'}$ as annotated graphs.
For~$G'$ there are multisets~$\mathcal{E}_1'$ and~$\mathcal{E}_2'$ analogous to~$\mathcal{E}_1$ and~$\mathcal{E}_2$ of~$G$.
Consider a vertex $v\in V(G)$ with $G_v \in \mathcal{E}_1$ for which~$|N_G[v]\cap R_1|$ is minimal. There is a vertex~$v' \in V(G')$ with $G'_{v'}\in \mathcal{E}_1'$ such that there is an isomorphism~$\varphi$ from~$G-v$ to~$G'-v'$ and~$\leftmset d_{G}(u)\colon u\in N_{G}(v)\rightmset=\leftmset d_{G'}(u)\colon u\in N_{G'}(v')\rightmset$.
By Lemma~\ref{lem: connected-vmaxdeg} Part~\eqref{itm:minus-flank-vertex}, we have~$N_{G-v}[\bulk(G-v)] = N_G[\bulk(G)]$ and $N_{G'-v'}[\bulk(G'-v')] = N_{G'}[\bulk(G')]$.
By the choice of $v$ and since $|F_1| > 1$ every vertex of $N_{G-v}[\bulk(G-v)]$ which has a neighbor in $F_1$ in $G$ also has a neighbor in $F_1\setminus \{v \}$ in $G-v$ and, hence $\bulk(G-v) = \bulk(G)$ and~$\bulk(G'-v') = \bulk(G')$.
 It follows that~$\varphi(R_1) = R_1'$ and~$\varphi(B_1)= B_1'$.  Also note that~$|F_1|=|F_1'|\geq 2$.
This means that all requirements of Lemma~\ref{lemma:farthest:away:and:at:least:two:reconstruct:other:side} are satisfied so~$H_{B_1}\cong H_{B_1'}$ as annotated graphs.
We have thus managed to reconstruct~$[H_{B_1}]_{\cong}$.

Using Lemma~\ref{lem:HFi:reconstr}, we can reconstruct the multiset $\leftmset [H_{F_1}]_{\cong},[H_{F_2}]_{\cong}\rightmset$. Since~$|F_1|\neq |F_2|$, we can reconstruct~$[H_{F_1}]_{\cong}$ given $\leftmset [H_{F_1}]_{\cong},[H_{F_2}]_{\cong}\rightmset$.
Overall, we have reconstructed the pair of annotated graphs~$([H_{F_1}]_{\cong},[H_{B_1}]_{\cong})$ coming from a clean clique separation. By Lemma~\ref{lem:reconst:clique:sep:implies:reconst}, the graph $G$ is reconstructible.

\paragraph{Case 2:~ $\mathbf{|F_1| = 1}$ and $\mathbf{2<|F_2|}$.}\leavevmode
 In this case, we use the same technique as in the previous case except that we choose~$G_x$ from~$\mathcal{E}_2$ rather than from~$\mathcal{E}_1$. Since~$|F_2| > |F_1|+1$ by Lemma~\ref{lem: E-recognizable} and Lemma~\ref{lem:each:flank:represented:in:E}, each of the multisets~$\leftmset G_v\colon v\in \mathcal{E}_1\rightmset$ and~$\leftmset G_v\colon v\in \mathcal{E}_2\rightmset$ is reconstructible. In~$G_x$ with~$G_x\in \mathcal{E}_2$, we can locate the unique flank of size~$|F_2|-1$.
 The annotated graphs $(H_{F_2},H_{B_2})$ are recovered in the same manner as~$(H_{F_1},H_{B_2})$ are recovered in the previous case.

\paragraph{Case 3:~$\mathbf{1 < |F_1|=|F_2|}$.}\leavevmode
With the same technique as described before, we can reconstruct $\leftmset [H_{B_1}]_{\cong}, [H_{B_2}]_{\cong} \rightmset$ but it is not clear which element of the multiset belongs to which flank. However, for each $i \in \{1,2\}$ the graph~$H_{B_i}$ has exactly one non-empty flank and that flank together with its neighbors induces the graph~$H_{F_{\neg i}}$. From this, we can determine the graph~$H_{F_i}$ that corresponds to~$H_{B_i}$ by either taking the graph non-isomorphic to~$H_{F_{\neg i}}$ or, if~$H_{F_1}\cong H_{F_2}$ by choosing arbitrarily.
Again, we have reconstructed a pair of annotated graphs~$(H_{F_i},H_{B_i})$ coming from a clean clique separation and then, by Lemma~\ref{lem:reconst:clique:sep:implies:reconst}, the graph is reconstructible.

\paragraph{Case 4:~$\mathbf{|F_1| = 1}$ and $\mathbf{|F_2| =2}$.}\leavevmode
Say~$F_1=\{f_1\}$ and $F_2= \{f_2,f_2'\}$. 
By Lemma~\ref{lem: E-recognizable} and Lemma~\ref{lem:each:flank:represented:in:E}, the multisets~$\mathcal{E}_1$ and~$\mathcal{E}_2$ are reconstructible. In particular $G_{f_1}$ is reconstructible and we can determine whether~$f_2$ and~$f_2'$ are adjacent\footnote{This slightly informal statement means that the map that associates with every graph~$G$ the pair~$(G_{v_1},b)$ consisting of an unlabeled graph~$G_{v_1}$ and the boolean~$b=\text{``}\{v_2,v_3\} \in E(G)\text{''}$ whenever~$G$ has a flanks~$F_1=\{v_1\}$ and~$F_2= \{v_2,v_3\}$ of sizes~$1$ and~$2$ and the empty set, say, otherwise, is reconstructible.}.
We may assume without loss of generality that~$\deg_G(f_2)\geq \deg_G(f_2')$. This means
$N[f_2]\supseteq N[f_2']$ if~$f_2$ and~$f_2'$ are adjacent and $N(f_2)\supseteq N(f_2')$ if they are not. In particular, if $\deg_G(f_2) = \deg_G(f_2')$, then~$f_2$ and~$f_2'$ are true or false twins.

If~$\deg_{G-f_2'}(f_2)\neq \deg_G(f_1)$, then in~$G_{f_2'}$ there are two flanks each containing precisely one vertex and these vertices are of different degree. 
We can reconstruct the graph~$G$ from~$G_{f_2'}$: we add a vertex~$x$ to~$G_{f_2'}$ and connect~$x$ and~$f_2$ if~$f_2$ and~$f_2'$ are adjacent\footnote{More formally, this means that the graph obtained from the unlabeled graph~$G_{f_2'}$ by the operation of adding~$x$ in the described fashion must be isomorphic to~$G$. This renders~$G$ reconstructible.}. For each neighbor that~$f_2'$ has in~$G$ of degree~$d$, we join~$x$ to a neighbor of~$f_2$ of degree~$d-1$. Since neighbors of~$f_2$ of equal degree in~$G-f_2'$ are twins, this reconstructs the graph~$G$.
We can thus assume that~$\deg_{G-f_2'}(f_2)= \deg_G(f_1)$.

\begin{itemize}
\item \textit{Case a: $\Vmaxdeg(G)$ is a twin equivalence class of $G$ and contains more than one vertex.} 
We argue that $\max\{\deg_G(f_1), \deg_G(f_2), \deg_G(f_2')\} \leq \maxdeg(G)-2$. Since~$\deg(f_2')\leq \deg(f_2)$ it suffices to show this for~$f_1$ and~$f_2$. Both the vertices~$f_i$ have a neighbor~$r_i$ of degree at most~$\maxdeg(G)-1$ that is adjacent to at least two vertices of maximum degree.
We obtain $|N_G(f_1)| \leq |N_G[r_1]\setminus (\Vmaxdeg(G) \cup \{f_1\}) | \leq \maxdeg(G)-3$
and $|N_G(f_2)| \leq |(N_G[r_1] \cup \{f_2'\}) \setminus (\Vmaxdeg(G) \cup \{f_2\})| \leq \maxdeg(G)-2$.

Let $v \in \Vmaxdeg(G)$.
Since there is only one twin equivalence class of vertices of maximum degree in $G$, the degree of every vertex in $V(G)\setminus F_G$ is reduced by exactly one in~$G_v$.
Altogether, we obtain $\Vmaxdeg(G) = \Vmaxdeg(G-v) \cup \{v\}$.
Hence,
we can reconstruct~$G$ by duplicating a vertex of degree~$\maxdeg(G)-1$ in $G_v$.

\item \textit{Case b: $\Vmaxdeg(G)$ partitions into at least two twin equivalence classes of $G$.}
For~$i\in \{1,2\}$ let~$w_i$ in $\Vmaxdeg(G)$ be a bordering vertex such that
$|N_G(w_i) \cap R_i|$ is maximized.
The set~$V(G)$ can be partitioned into~$A_i\coloneqq \{v\in V(G)\colon N[v]\subseteq N[w_i] \cup F_i\}$,~$C_i\coloneqq N[w_i] \setminus A_i$ and~$B_i\coloneqq V(G)\setminus (A_i\cup C_i)$.

Note that~$B_i$ contains at least two vertices: Indeed~$B_1 \supseteq \{f_2, f_2'\}$, while~$B_2$ contains~$f_1$ and a vertex that is adjacent to a vertex of maximum degree but not to~$w_2$, since otherwise there would be only one twin equivalence class of vertices of maximum degree. Note that~$(A_i,C_i,B_i)$ is a clean clique separation.

From $G_{f_1}$, we can reconstruct~$H_{A_2}$ as annotated graph: The graph~$G-f_1$ has exactly one non-empty flank, namely~$F_2$. It also has exactly one equivalence class of bordering vertices whose side is that flank. We can thus locate the vertices corresponding to~$A_2$ in~$G_{f_1}$ and then locate the vertices corresponding to~$C_2$. Now reconstructibility of~$H_{A_2}$ follows from Lemma~\ref{lemma:farthest:away:and:at:least:two:reconstruct:other:side}.

 From~$G_{f_1}$, we can also reconstruct the graph~$H_{B_1}= G[C_1\cup B_1]$ as annotated subgraph using similar arguments with the other class of bordering vertices. 
 Using exactly the same method we can reconstruct~$\widehat{H}= (G-f'_2)[C_1\cup B_1]$ as annotated subgraph by additionally deleting the flank vertex of smallest degree.

From~$G_{f_2'}$, we can reconstruct the multiset~$\leftmset H_{C_2}, \widehat{H} \rightmset$ of annotated subgraphs of~$G-{f_2'}$. 
Note that for~$H_{B_2}$ the annotation as a subgraph of~$G$ and the annotation as a subgraph of~$G-f_2'$ are the same.

Overall, we can reconstruct~$(H_{A_2},H_{B_2})$ and thus~$G$ is reconstructible by Lemma~\ref{lem:reconst:clique:sep:implies:reconst}.

\item \textit{Case c: $|\Vmaxdeg(G)| = 1$.}
Let $a$ be the vertex in $\Vmaxdeg(G)$.

If~$G-a$ is not connected, then we can reconstruct~$G$:
 In~$G_{a}$, we observe the multiset of isomorphism types of connected components of $G-a$.
In~$G_{f_1}$, we can observe the component~$C$ containing~$f_2$ and~$f_2'$ and how it is attached to~$a$.
From~$G_{f_2'}$ and~$G_{f_2}$ we can then observe at most two components containing exactly one special vertex not adjacent to~$a$. If there is only one component or the two components are isomorphic mapping the special vertices to each other, we have reconstructed all components with multiplicity.
If there are two components not isomorphic via an isomorphism respecting the special vertices, then one of them is~$C-\{f_2\}$ or~$C-\{f_2'\}$ and we can determine which one it is. The other component is the last component we need.

We can thus assume that~$G-a$ is connected.
Let~$x_1, x_2,\ldots,x_t$ be a maximal sequence of vertices forming a path with~$x_1=f_2$ and for~$i>1$ the vertex~$x_i$ has degree at most~$2$ in~$G-a-f_2'$ (possibly~$t=1$).
If this sequence contains all vertices of~$G-a-f_2'$ then~$G-f_2'$ has an automorphism interchanging~$f_1$ and~$f_2$ and hence~$G$ is reconstructible.
Otherwise let~$w$ be a neighbor of~$x_t$ of maximum degree in~$G-a-f_2'$ and~$w'$ a neighbor of~$x_t$ of minimum degree different from~$w$.

By assumption~$\deg_{G-f_2'}(f_2)=\deg_G(f_1)$ and, hence,
neither~$f_1$ nor a neighbor of~$f_1$ is contained in~$\{x_1,\ldots,x_t,w\}$. 
We can reconstruct~$G$ as follows: We use a card~$G_{v'}$ with~$v'$ a twin of~$w'$ or~$w'$ itself. Consider the triple~$(A,B,C)$ with~$A= \{v\colon N[v]\subseteq 
N[\{x_1,\ldots,x_t,w\}]\}$, with~$B= N[A]\setminus A$, and with~$C$ the rest of the graph.
From~$G_{v'}$ we can determine the annotated graph~$H_C$ and from~$G_{f_1}$ we can determine~$H_{A}$. We obtain a reconstructible clean clique separation and the graph is reconstructible by Lemma~\ref{lem:reconst:clique:sep:implies:reconst}.

\end{itemize}

\paragraph{Case~5:~$\mathbf{|F_1| = |F_2| = 1}$.}\leavevmode
Say $F_1 = \{f_1\}$ and $F_2 = \{f_2\}$.

\begin{itemize}

\item \textit{Case a: there is exactly one twin equivalence class and it has more than one vertex.} 

This case is analogous to Case 4a.

\item \textit{Case b: There are at least two twin equivalence classes of maximum degree vertices.}

Similarly to Case 4b, let $w_i$ be a vertex of maximum degree that has a common neighbor with a vertex in~$F_i$. The vertex set of~$G$ can be partitioned into~$A_i\coloneqq \{v\in V(G)\colon N[v]\subseteq N[w_i] \cup F_i\}$,~$B_i\coloneqq N[w_i] \setminus A_i$ and~$C_i\coloneqq V(G)\setminus (A_i\cup B_i)$.

We argue that~$\leftmset (H_{A_1},H_{C_2}), (H_{A_2},H_{C_1})\rightmset$ is reconstructible.

For~$i\in \{1,2\}$ consider~$G_{f_i}$. In this graph there is exactly one equivalence class of vertices of maximum degree that only has one side. Let~$w$ be a vertex in that equivalence class. The set~$A_i$ consists of~$f_i$ and those vertices~$v$ of~$G_{f_i}$ that satisfy~$N[v]\subseteq N[w]$ in~$G_{f_i}$.
The set~$B_i$ consists of the vertices adjacent to~$w$ but also adjacent to a vertex not in~$A_i$ and the set~$C_i$ consists of all other vertices. Since no vertex in~$B_i$ can be adjacent to~$f_i$ in~$G$, the graph~$H_{C_i}$ is the induced subgraph~$G_w[B_i\cup C_i]$, which can be obtained from~$G_{f_i}$ together with its annotation. The graph~$H_{A_{\neg i}}$ is the graph induced by all vertices not adjacent to~$w$ and all their neighbors in~$G_{f_i}$. Since in~$G$ none of these vertices are adjacent to~$f_i$ we can recover the graph~$H_{A_{\neg i}}$ with annotation.
Overall from the card~$G_{f_i}$ we can recover the pair~$(H_{A_{\neg i}},H_{C_i})$. 
While we may not be able to distinguish between~$G_{f_1}$ and~$G_{f_2}$, we can recover~$\leftmset (H_{A_1},H_{C_2}), (H_{A_2},H_{C_1})\rightmset$ as a multiset. From this, we can recover~$\leftmset (H_{A_1},H_{C_1}),(H_{A_2},H_{C_2})\rightmset$ and thus have recovered a clean clique separation of~$G$. (In fact, we have even recovered two clean clique separations of~$G$.) Lemma~\ref{lem:reconst:clique:sep:implies:reconst} implies that~$G$ is reconstructible.
\item \textit{Case c: $|\Vmaxdeg(G)| = 1$.}

Denote the unique vertex in~$\Vmaxdeg(G)$ by~$a$.
If~$G-a$ is a path, then $f_1$ and $f_2$ are the degree 1 vertices in~$G-a$.
We can reconstruct~$G$ by adding a vertex adjacent to all other vertices.
If~$G-a$ is disconnected, then we can reconstruct~$G$ with arguments similar to what we used at the beginning of Case 4c.

We now assume~$G-a$ is connected.
We consider paths in~$G-a$ starting from the flank vertices~$f_i$ with vertices of degree exactly~$2$ in~$G-a$. 
Since~$G-a$ is connected, there cannot be three vertices of degree~$1$ that have neighbors of degree 2 (because this holds for interval graphs in general, since there are no asteroidal triples).
This means, if the paths have length at least 2, then we can find the neighbor of the vertex~$f_i$ in~$G-f_i-a$. Otherwise we can reconstruct a clean clique separation, as before in Case 4c, using vertices~$w$ and~$w'$ not of degree 2. More precisely we reconstruct $\leftmset (H_{A_1},H_{C_2}), (H_{A_2},H_{C_1})\rightmset$ and thereby~$\leftmset (H_{A_1},H_{C_1}),(H_{A_2},H_{C_2})\rightmset$ for suitable triples~$(A_i,B_i,C_i)$.
Overall the graph is reconstructible.
\qedhere
\end{itemize}
\end{proof}

\section{The outsiders}\label{sec:outsiders}

In order to deal with graphs that do not have two flanks, we use the concept of an outsider. Intuitively an outsider corresponds to an interval that is extremal on an end of the interval graph that does not have a flank.

Throughout this chapter we assume that~$G$ is connected, has no universal vertex, and at most one flank, since otherwise~$G$ is reconstructible by Lemma~\ref{lem: reconstructible-graph-classes} and Section~\ref{sec: two-flanks}, respectively.

Let
\[A \coloneqq A(G) \coloneqq \{z \in \Vmaxdeg(G)\colon z~\text{has at most one side}\}.\]

It follows with Lemma~\ref{lem: bordering-means-flank-is-side} that~$A \neq \emptyset$ if and only if~$G$ does not have two flanks.

\begin{lemma}\label{lem:extreal:max:deg:vertices}
The set~$A$ is composed of at most two twin equivalence classes. The number of non-empty flanks plus the number of twin equivalence classes in~$A$ add up to~2.
\end{lemma}
\begin{proof} 
Let~$T$ be the set of bordering maximum degree vertices of~$G$.
Note that all vertices in~$A$ are bordering and, hence, $A \subseteq T$.
By  Lemma~\ref{lem:order:of:max:deg:vertices}, the set~$T$ is comprised of at most two twin equivalence classes.

If~$T$ consists of only one twin equivalence class, then $\Vmaxdeg(G) = T$ and the sides of vertices in~$T$ are precisely the flanks of $G$. 
Since we assume in this section that $G$ has at most one flank and  no universal vertex there is precisely one flank and~$T=A$.

If~$T$ consists of two equivalence classes, then
for each equivalence class of~$T$, either the vertices have two sides, one of which is a flank by Lemma~\ref{lem: bordering-means-flank-is-side} or the equivalence class is contained in~$A$.  
Finally, note that if~$T$ has two equivalence classes both with vertices that have two sides, then~$G$ has two flanks and~$A$ is empty (Lemma~\ref{lem:side:versus:flank}).
\end{proof}

In the following, we assume~$A_1$ and~$A_2$ to be the twin equivalence classes in~$A$.
Let $i \in \{1,2\}$ with $A_i \neq \emptyset$.
Observe that  $\Span(A_i) = \{v \in V(G)\colon N_G[v] \subseteq N_G[A_i]\}$ since $A_i$ is a twin equivalence class of $G$.
Equivalently stated, for every tidy interval representation $\{ I_u\colon u \in V(G) \}$ of $G$ a vertex $v$ is in $\Span(A_i)$ precisely if $I_v \subseteq I_x$ for $x \in A_i$.
Let us recall that by definition~$\partition_G(\Span(A_i)) \coloneqq (\Span(A_i),N_G(\Span(A_i)),\NTwo{G}{\Span(A_i)} )$.
Observe that none of the sets in $\partition_G(\Span(A_i))$ are empty since $G$ has no universal vertex.

\begin{definition}
	Set \[O_i \coloneqq  \{o \in V(G)\colon N_G[o] \subseteq \Span_G(A_i) \} .\]
	We say that a vertex in $O_i$ is \emph{an outsider of} $A_i$ and $O_i$ is an \emph{outsider class} of $A_i$ (and of $G$).
	The vertices in~$O_1\cup O_2$ are the \emph{outsiders} of~$G$.
\end{definition}

\begin{lemma} \label{lem: properties of o}
	Let $i \in [2]$ with $A_i \neq \emptyset$.
	If $o$ is a vertex of $O_i$, then
		\begin{equation}
			\deg_G(o) < \maxdeg(G)~\text{and}~N_G(o) \cap \Vmaxdeg(G) = A_i.
		\end{equation}
\end{lemma}
\begin{proof}
	Suppose $x \in \Vmaxdeg(G)$. 
	
	If $x \in A_i$, then $x$ has a neighbor in $\NTildeOne{G}{A_i}$ since $x$ is not universal.
	If on the other hand $x \notin A_i$, then $x \notin \Span(A_i)$ since $x$ is of maximum degree and, hence, has at least one neighbor which is not in~$N_G[A_i]$. Altogether, we obtain that $\Vmaxdeg(G)\cap O_i = \emptyset$, which proves the first part of the lemma.
	
	For the second part assume that $x \in N_G(o)$ for some outsider $o \in O_i$. Since $o$ is not adjacent to vertices in $\NTildeOne{G}{A_i} \cup    \NTildeTwo{G}{A_i}$, we obtain $x \in \Span(A_i)$. Every maximum degree vertex in $\Span(A_i)$ is in $A_i$, yielding $N_G(o) \cap \Vmaxdeg(G) \subseteq A_i$.
	The other inclusion follows since $o \in N_G[A_i]$ by definition and $A_i$ is a twin equivalence class of $\Vmaxdeg(G)$.
\end{proof}

\begin{lemma}\label{lem:anonempty-ononempty}
If~$A_i$ is non-empty, then~$O_i$ is non-empty. 
\end{lemma}
\begin{proof}
Since vertices in~$A_i$ have only one side, the vertices in~$\NTildeOne{G}{A_i}$ form a clique and they are all adjacent to each vertex of~$A_i$ (since $A_i$ is a twin equivalence class). However, vertices in~$\NTildeOne{G}{A_i}$ also have a neighbor in~$\NTildeTwo{G}{A_i}$ (which is non-empty since otherwise each vertex of $A_i$ is universal). Since vertices in~$A_i$ have maximum degree, there is some vertex~$o\in N[A_i] \setminus A_i$ that is not adjacent to any vertex in~$\NTildeOne{G}{A_i}$. We conclude that~$o\in O_i$.
\end{proof}

\begin{lemma}\label{lem:one:flank:means:max:deg:rivet}
For~$i\in \{1,2\}$, if~$|O_i|=1$, then there is a vertex of maximum degree~$x\in \NTildeOne{G}{A_i}$ that is adjacent to all vertices of~$N_G[A_i]$ except to the vertex $o_i$ in~$O_i$.

In particular, in $G-o_i$ the twin equivalence class $A_i'$ of $x$ is a class of bordering, maximum degree vertices of degree~$\maxdeg(G)$ and $F_{G-o_i} = F_{G}$.
\end{lemma}
\begin{proof}
Suppose~$|O_i|=1$. The neighborhoods~$N[x]\cap N[A_i]$ for vertices~$x\in \NTildeOne{G}{A_i}$ are linearly ordered (Lemma~\ref{lem:tilde:sets:give:clean:clique:sep}).
Every vertex of~$N[A_i]\setminus O_i$ is adjacent to a vertex of~$ \NTildeOne{G}{A_i}$. It follows that some vertex~$y \in \NTildeOne{G}{A_i}$ is adjacent to all vertices of~$N[A_i]\setminus O_i$. This vertex~$y$ also has a neighbor in~$\NTildeTwo{G}{A_i}$. Thus the degree of~$y$ is at least the degree of the vertices in~$A_i$, so~$\deg_G(y) = \maxdeg(G)$.
\end{proof}

\begin{lemma} \label{lem: steal vertex from outsider class}
	Let $i \in \{1,2\}$ with $|O_i| \geq 2$. 
	If $o \in O_i$ and $\Vmaxdeg(G) \neq A_i$, then
	\begin{enumerate}[(i)]
		\item $G-o$ has precisely one more flank $F$ than $G$ and $O_i\setminus \{o\} \subseteq F$.
		\item $O(G-o) = O_{\neg i}$ (which might be empty).
	\end{enumerate}
\end{lemma}

\begin{proof}	
	Let $o' \in O_i \setminus \{o\}$.
	Note that~$G_o$ is connected. Also note that~$\Vmaxdeg(G_o) = \Vmaxdeg(G)$ since~$o\notin N_G[A_{\neg i}]$. By Lemma~\ref{lem: properties of o}, we have $N_G(o') \cap \Vmaxdeg(G) = A_i$ and, hence, $N_{G-o}(o') \cap \Vmaxdeg(G-o) = \emptyset$
	
	Consider a tidy interval representation of~$G$. It induces an interval representation of~$G-o$. By Lemma~\ref{existence:tidy}, after (possibly) shrinking the intervals, we obtain a tidy interval representation of~$G-o$. Using Lemma~\ref{lem:nucleus:and:span} Part~\eqref{itm:span-is-intuitive-span}, we conclude that
	$O_i \setminus \{o\}$ is contained in a flank of~$G_o$. Since $\Span_G(A_{\neg i}) \subseteq \Span_{G-o}(A_{\neg i})$, we also obtain that every vertex~$v$ in $O_{\neg i}$ satisfies~$N_{G-o}[v] \subseteq \Span_{G-o}(A_{\neg i})$ so vertices in~$O_{\neg i}$ are outsiders of $G-o$.

	It remains to show that in~$G-o$ there are no further outsiders in the same class as~$O_{\neg i}$.

	Suppose~$w$ would be such an outsider in~$G-o$. 
	Then~$w$ is not adjacent to~$A_i$ since~$A_i$ is adjacent to~$O_i\setminus \{o\}$.
	It would also have to be the case that~$w$ and~$o$ have a common neighbor, since otherwise~$w$ was already an outsider in~$G$.
	Together this implies that~$o$ is not an outsider in~$G$, which gives a contradiction.
\end{proof}

\begin{remark}
	Since we assume that no vertex of $G$ is universal, the condition $\Vmaxdeg(G) \neq A_i$ is always satisfied in case $F_G = \emptyset$.
\end{remark}

\begin{lemma}\label{lem: outsider-in-interval-rep}
	Let $\mathcal{I} = \{[\ell_v, r_v]\colon v \in V(G)\}$ be a tidy interval representation of $G$.
	Assume that $A_1 \neq \emptyset$.
	If additionally $A_2 \neq \emptyset$, then we further assume that
	$\ell_{A_1} < \ell_{A_2}$ (and, hence, also $r_{A_1} < r_{A_2}$).
	It holds that
	\begin{enumerate}[(i)]
		\item \label{itm: a2nonempty}  either $A_2 \neq \emptyset$ and
		\begin{align*}
			O_1 &= \{v \in V(G)\colon r_v < \min\{\ell_u\colon u \in N_G(\Span(A_1)) \}  \}~\text{and}\\
			O_{2} &= 
				\{v \in V(G) \colon \ell_v > \max\{r_u\colon u \in N_G(\Span(A_{2} )) \}  \},
		\end{align*}
		\item or $A_2 = O_2 = \emptyset$ and \label{itm: a2empty}
		\begin{align*}
			O_1 &=  \{v \in V(G)\colon r_v < \min\{\ell_u\colon u \in N_G(\Span(A_1)) \} ~\text{or}\\
			O_1 &=  	\{v \in V(G) \colon \ell_v > \max\{r_u\colon u \in N_G(\Span(A_{1} )) \}  \}.
		\end{align*}
	
	\end{enumerate}
	
\end{lemma}

\begin{proof}
	Towards a proof of Part~\eqref{itm: a2nonempty} assume that $A_2 \neq \emptyset$. 
	The set $A_1$ has only one side (by definition) and there exists a vertex $z \in N_G[A_2]\setminus N_G[A_1]$, with $\ell_z > r_{A_1}$. It follows that $\ell_{A_1} = \ell_{V(G)}$ is empty since $\mathcal{I}$ is tidy.
	First, we prove that $\{v \in V(G) \colon r_v < \min\{l_u\colon u \in N_G(\Span(A_1)) \}  \} \subseteq O_1$.
	To this end let $v \in V(G)$ with $r_v < \min \{\ell_u \colon u \in N_G(\Span(A_1)) \}$.
	Note that $r_v < \min \{\ell_u \colon u \in N_G(\Span(A_1)) \} <r_{A_1}$. Hence $v \in N_G[A_1]\setminus N_G(\Span(A_1))$, which is the defining property for being in~$O_1$.
	
	Next, we show that $O_1 \subseteq \{v \in V(G)\colon r_v < \min\{\ell_u\colon u \in N_G(\Span(A_1)) \}  \}$.
	Let $o_1 \in O_1$. 
	Since $N_G[o_1] \subseteq N_G[A_1]$ and $\mathcal{I}$ is tidy, we know that $[\ell_{o_1}, r_{o_1} ] \subseteq [\ell_{A_1}, r_{A_1}]$.
	By the definition of outsiders, $o_1$ is not adjacent to~the vertices of $N_G(\Span(A_1))$. Since $\ell_{A_1} = \ell_{V(G)}$,
	every vertex~$n$ in $N_G(\Span(A_1))$ satisfies $r_n > r_{A_1}$. 
	 This implies $o_1 \in \{v\colon r_v < \min\{\ell_u\colon u \in N_G(\Span(A_1)) \}  \}$.
	
	Altogether, we have shown the first equality of Part~\eqref{itm: a2nonempty}.
	The second equality of Part~\eqref{itm: a2nonempty} follows by symmetry (and interchanging the assumption for $A_1$ and $A_2$).
	
	The proof of Part~\eqref{itm: a2empty} is similar to the proof of Part~\eqref{itm: a2nonempty} with the only difference that we do not know which of the sides of $A_1$ (\say{left} or \say{right}) is the empty side. Depending on that, we obtain the first or the second equality of the statement.
\end{proof}

Our next goal is to reconstruct the multiset of cards that correspond to vertices in~$O_1\cup O_2$. For this let~$Q^{\geq 2}$ be the set of vertices of~$G$ for which
\begin{enumerate}
	\item there is a twin equivalence class $V'$ of $\Vmaxdeg(G)$ with
	$N_G[v] \subseteq \Span(V')$, and 
	\item $G_v$ has more flanks than~$G$ or~$\maxdeg(G_v)< \maxdeg(G)$.\label{prop:more:flanks}
\end{enumerate}

\begin{lemma}\label{lemma:find:Q2}
$v\in Q^{\geq 2}$ if and only if for some~$i\in \{1,2\}$ we have~$v\in O_i$ and~$|O_i|\geq 2$. Moreover~$\leftmset G_v \colon v\in Q^{\geq 2} \rightmset$ is reconstructible.
\end{lemma}

\begin{proof}
	(\emph{$\Rightarrow$})
	Let~$v \in Q^{\geq 2}$.
	Let~$V'$ be an equivalence class of~$\Vmaxdeg(G)$ with~$N_G[v] \subseteq \Span(V')$.
	Observe that $(\Vmaxdeg(G)\setminus V') \cap \Span(V') = \emptyset $.
	In particular, if $\maxdeg(G_v) < \maxdeg(G)$, then $V' = \Vmaxdeg(G)$ and $v$ is an outsider.
	
	Otherwise $G_v$ has more flanks than $G$ and $\maxdeg(G) = \maxdeg(G_v)$.
	Note that~$G-v$ is connected and $\Vmaxdeg(G)\setminus V'	= \Vmaxdeg(G-v)$.
	In particular, the graph $G-v$ inherits the linear order of its maximum degree vertices as well as twin equivalence classes and bordering vertices from~$G$.
	In particular, if~$V' \notin \{A_1, A_2\}$, then 
	the number of flanks in~$G$ and~$G_v$ are the same, which contradicts	 the second condition of $Q^{\geq 2}$.
	It follows that~$V'\in \{A_1,A_2\}$ and thus~$v \in O(G)$.
	
	Towards a contradiction suppose that~$|O_i|= \{v\}$.
	By Lemma~\ref{lem:one:flank:means:max:deg:rivet},
	there are not more flanks in $G_v$ than in~$G$ and~$\maxdeg(G_v)= \maxdeg(G)$, which is a contradiction.
	
	\medskip
	
\noindent	(\emph{$\Leftarrow$})  Let~$v\in O_i$ be an outsider and~$|O_i|\geq 2$.
	Observe first that $A_i$ is the twin equivalence class satisfying the first condition of the definition of $Q^{\geq 2}$.
	If $A_i = \Vmaxdeg(G)$, then $\maxdeg(G_v) < \maxdeg(G)$ and, hence, $v \in Q^{\geq 2}$.
	Otherwise~$G_v$ has more flanks than~$G$ by Lemma~\ref{lem: steal vertex from outsider class}.
	
	\medskip
	
	\noindent(\emph{$Q^{\geq 2}$ is reconstructible}.)

	Consider the set~$W'$ of vertices~$w \in V(G)\setminus \Vmaxdeg(G)$ which satisfy that  $N_G(w)\cap \Vmaxdeg(G)$ is a twin equivalence class of $G$, $G_w$ is connected, and $G_w$ has more flanks than~$G$ or~$\maxdeg(G_w)< \maxdeg(G)$. 
	
    Note that the case that~$\maxdeg(G_w)< \maxdeg(G)$ can only happen if~$G$ only has a single twin equivalence class of vertices of maximum degree. If this is not the case then we set~$W\coloneqq W'$.
    Otherwise, we set~$W$ to be the set of those vertices among~$w\in W'$ for which~$G-w$ has the largest number of twin equivalence classes of degree~$\maxdeg-1$.

	 \begin{claim} 
	  If~$w\in W$, then~$N_G(w)\cap \Vmaxdeg(G) \in \{A_1, A_2\}$.
	 \end{claim}
	 	 \noindent $\ulcorner$ 
	 Suppose towards a contradiction that $A_i \cap N_G(w) = \emptyset$ for $i \in \{1,2\}$. This implies that $\maxdeg(G_v) = \maxdeg(G)$, the bordering vertices of $G-v$ are precisely the bordering vertices of $G$ (since the linear order of the maximum vertices in $G-v$ is inherited from a respective order in $G$, see Lemma~\ref{lem:order:of:max:deg:vertices}), and $F_{G-v} = F_G$, which contradicts $w \in W$.
	 We obtain that $N_G(w) \cap A_i \neq \emptyset$ for some $i \in \{1,2\}$.
	 Since $N_G(w) \cap \Vmaxdeg(G)$ is a twin equivalence class of $G$, we obtain $A_i = N_G(w) \cap \Vmaxdeg(G)$.
	 \\\mbox{}\hfill~$\lrcorner$

	\medskip
	
	By Lemma~\ref{lem:order:of:max:deg:vertices}, there is a unique order (up to twins and reflection) of the non-neighborhood-contained vertices of $G$.
	Let~$D_1$ and $D_2$ be the extremal twin equivalence classes of vertices of degree at least $\maxdeg(G)-1$ in this order.

	Note that, for $i \in \{1,2\}$, we have~$A_i \in \{D_1,D_2\}$ if~$A_i$ is non-empty. After renaming, we can assume that~$A_i= D_i$ if~$A_i$ is not empty.

Similarly for $w \in W$ consider the order of vertices in~$G-w$ that are not neighborhood-contained by a vertex of larger degree.
Let~$D^w_1$  and $D^w_2$ be the first and last twin equivalence class of vertices of degree at least $\maxdeg(G)-1$ in this order in~$G-w$.

	 \begin{claim} 
	 Let $w$ be a vertex of $W$.
	 If~$w\in N_G[A_i]$, then for some $j \in \{1,2\}$ we have $A_i\subseteq D_j^w$ and~$D_{\neg i}\subseteq D_{\neg j}^w$. 
	 \end{claim}
	 \noindent $\ulcorner$
	 Note first that if~$A_i= D_{\neg i}$ then~$A_1=D_1=D_2=A_2$ and the claim holds. We can thus assume~$A_i\neq  D_{\neg i}$. 
	 
	 Let $w \in N_G[A_i]$. In particular $A_i \neq \emptyset$ and, hence,  $A_i = D_i$ and~$w\in N_G[D_i]$.
	 We argue that~$w\notin N_G[D_{\neg i}]$.
	 If~$A_{\neg i}\neq \emptyset$ this follows from the definition of~$W$ since~$w$ is only adjacent to one twin equivalence class of $\maxdeg(G)$-vertices. Otherwise, if~$w\in N_G[D_{\neg i}]$, then there is only one twin equivalence class of maximum degree by the definition of $W'$.
	  In particular,~$\maxdeg(G_w)<\maxdeg(G)$. In this case however, since~$A_1\neq D_2$, there are some vertices in~$y\in W$ so that~$G-y$ has at least two twin equivalence classes of degree~$\maxdeg(G)-1$, namely the vertices in~$O_i$. It follows that~$w$ is not adjacent to every vertex of degree at least~$\maxdeg(G)-1$. 
	 
	 So far, we have shown~$w\in N_G[D_i]\setminus N_G[D_{\neg i}]$ and $A_i = D_i$. We conclude that vertices of~$D_1\cup D_2$ have degree at least~$\maxdeg(G)-1$ in~$G-w$.
	 Since~$G-w$ is an induced subgraph of~$G$ it follows that the unique order of non-neighborhood-contained vertices is inhered by the vertices in~$G-w$. It follows that vertices~$D_1$ and~$D_2$ are vertices in the first and last equivalence classes of vertices of maximum degree at least~$\maxdeg(G)-1$.
	 \hfill~$\lrcorner$

	\medskip

	For $w \in W$ let
	$X_w \coloneqq \{x \in V(G-w)\colon N_{G-w}[x] \subseteq \Span_{G-w}(D_j)~\text{for some $j \in \{1,2\}$}\}$.
	Note that $|X_w|$ can be reconstructed from~$G_w$.
	We claim that the cards~$G_v$ with~$v\in  Q^{\geq 2}$ are exactly the cards with~$v\in W$ for which~$|X_v|$ is as small as possible.
	For this it suffices to observe that if we
	set~$X\coloneqq \{x \in V(G)\setminus \Vmaxdeg(G)\colon N_G[x] \subseteq \Span(D_i)$\text{~for some~}$i \in \{1,2\} \}$, then 
	 $x\in X\setminus \{w\}$ implies~$x\in X_w$  and 
		if~$w\in O_i$, then~$A_i\in \{D_1,D_2\}$ and~$X_w = X\setminus \{w\}$.
\end{proof}

\begin{lemma}\label{lem:Q:sizes}
Suppose~$O_1$ and~$O_2$ are the (possibly empty) sets of outsiders as defined above.
 Then the multiset~$\leftmset |O_1|,|O_2|\rightmset$ is reconstructible.
\end{lemma}
\begin{proof}By Lemma~\ref{lemma:find:Q2}, we can reconstruct whether~$Q^{\geq 2}$ is empty and by Lemma~\ref{lem:sizes}, we can reconstruct~$\leftmset |F_1|, |F_2| \rightmset$.
Since we assume that $G$ has at most one flank the set~$Q^{\geq 2}$ is empty if and only if $\leftmset |O_1|,|O_2|\rightmset= \leftmset 1,1\rightmset$ or $\leftmset |O_1|,|O_2|\rightmset= \leftmset 0,1\rightmset$. Which of these cases can occur depends on~$|F_G|$, so can be reconstructed.
We can thus assume $Q^{\geq 2} \neq \emptyset$.
If~$G$ has one flank, then~$\leftmset |O_1|,|O_2|\rightmset = \leftmset 0,\max\{1,|Q^{\geq 2}\}|\rightmset$. We can thus assume~$F_G = \emptyset$, i.e., $O_1$ and $O_2$ are non-empty.  

Let~$v\in Q^{\geq 2}$ and say, $v \in O_i$.
Observe that $G-v$ has precisely one flank (containing the vertices of $O_1\setminus \{v\}$) and $O_{\neg i}$ is the set of outsiders of $G-v$ (Lemma~\ref{lem: steal vertex from outsider class}).
Either all cards $G_v$ with $v \in Q^{\geq 2}$ have precisely one outsider and, hence, $\leftmset |O_1|, |O_2| \rightmset = \leftmset 1, |Q^{\geq 2}\}| \rightmset$.
Or we have $Q^{\geq 2} = O_1 \cup O_2$.
If all of the outsider class sizes of cards $G_v$ with $v \in Q^{\geq 2}$ are equal to the same number,~$c$ say, then $\leftmset |O_1|,|O_2|\rightmset= \leftmset c,c\rightmset$. Otherwise, exactly two numbers~$c_1,c_2$ can be observed as outsider class sizes, and~$\leftmset |O_1|,|O_2|\rightmset= \leftmset c_1,c_2\rightmset$.
\end{proof}

In the following, we argue that we can also identify the cards~$G_v$ for which~$v\in O_i$ in case~$|O_i|=1$.

For an interval graph~$H$ of maximum degree~$\maxdeg$, let~$A^\infty(H)$ be the set of vertices~$x$ for which there is a sequence of vertices~$v_1,\ldots,v_t$ with~$v_t=x$,~$v_1\in A(H)$ (so~$v_1$ is a maximum degree vertex with at most one side),~$v_i\in \Vmaxdeg(H)$ and, in case~$t>1$, for each~$i\in \{2,\ldots,t\}$ we have~$|N[v_i] \setminus N[v_{i+1}]|=1$.
Call two elements~$a,a'$ of~$A^\infty(H)$ equivalent 
if there is a sequence~$v_1,\ldots,v_t$ with~$v_1\in A(H)$,~$a=v_i$ for some~$i\in \{1,\ldots,t\}$,~$a'=v_t$, ~$v_i\in \Vmaxdeg(H)$, and~$|N[v_i] \setminus N[v_{i+1}]|\leq 1$ for each~$i\in \{2,\ldots,t\}$.
Note that~$A^\infty(H)$ consists of at most two equivalence classes.
We let~$A^\infty(H)_i$ be the equivalence class of~$A^\infty(H)$ that contains~$A_i(H)$.
(It is possible that~$A^\infty(H)_1= A^\infty(H)_2$.)

If~$|O_1|\geq 2$ and~$|O_2|\geq 2$, then define~$Q^1$ to be the empty set. Otherwise let~$Q^1$ be the set of vertices~$v$ for which
\begin{enumerate}
\item \label{qone:isdegree}$\deg(v)<\maxdeg(G)$,
\item \label{qone:prop:outsider} there is a twin equivalence class $V'$ of $\Vmaxdeg(G)$ with  $V'=N[v]\cap \Vmaxdeg(G)$,
\item $G_v$ has as many flanks as~$G$, and\label{qone:prop:bulk:is:all}
\item \label{qone:many:at:the:border}$\lvert \NTildeZero{G_v}{A^\infty(G_v)_1}  \cup \NTildeZero{G_v}{A^\infty(G_v)_2}\rvert$ is minimal among all cards~$G_v$ satisfying Properties~\ref{qone:isdegree}--\ref{qone:prop:bulk:is:all}.
\end{enumerate}

\begin{lemma} \label{lem:Q1:reconstructible}
We have $v\in Q^1$ if and only if~$v\in O_i$ and~$|O_i|= 1$ for some~$i\in \{1,2\}$. 
Moreover~$\leftmset G_v \colon v\in Q^1 \rightmset$ is reconstructible.
\end{lemma}

\begin{proof}
Since $\leftmset |O_1|,|O_2|\rightmset$ is reconstructible we can reconstruct whether $Q^1 = \emptyset$.

Hence, we may assume~$|O_1|=1$ or~$|O_2|=1$.

($\Leftarrow$) Let~$v\in O_i$ with~$|O_i|=1$.
Properties \ref{qone:isdegree} and~\ref{qone:prop:outsider} are satisfied with~$V'=A_i$ for some~$i$, since $v$ is an outsider.
We argue that~$G_v$ has as many flanks as~$G$. 
By Lemma~\ref{lem:one:flank:means:max:deg:rivet}, there is a vertex~$x\in \NTildeOne{G}{A_i}$ of maximum degree that is adjacent to all vertices of~$N[A_i]\setminus\{v\}$.
 Note that~$x$ has one side fewer in~$G_v$ than it does in~$G$. It follows that~$G_v$ has as many flanks as~$G$.
 
 Regarding Property~\ref{qone:many:at:the:border}, we note that for every~$v\in V(G_v)$ with Properties~\ref{qone:isdegree}--\ref{qone:prop:bulk:is:all} we have for each~$i\in \{1,2\}$ that $\NTildeZero{G_v}{A^\infty(G_v)_i} \subseteq \NTildeZero{G}{A^\infty(G)_i} \setminus\{v\}$.
 On the other hand, we have $\NTildeZero{G_v}{A^\infty(G_v)_i} =\NTildeZero{G}{A^\infty(G)_i}\setminus\{v\}$ in case~$v\in O_j$ with~$j\in \{1,2\}$ and~$|O_j|=1$: here we use that if~$v\in A^\infty(G)_i \cap O_{\neg i}$ then~$A^\infty(G)_i= A^\infty(G)_{\neg i}$.

($\Rightarrow$) For the other direction assume~$v\in Q^1$ and let $V'=N[v]\cap \Vmaxdeg(G)$ be the twin equivalence class of $\Vmaxdeg(G)$ with $N[v]\subseteq  N[V']$.

The fact that~$G_v$ has as many flanks as~$G$ implies that~$v$ is not an outsider in a class of size larger than~$1$.

From the previous part of the proof for an outsider~$o$, we know that  $| \NTildeZero{G_o}{A^\infty(G_o)_1}\cup \NTildeZero{G_o}{A^\infty(G_o)_2}| = |\NTildeZero{G}{A^\infty(G)_1} \setminus \{o\}\cup \NTildeZero{G}{A^\infty(G)_2}\setminus\{o\}|$.

Then Property~\ref{qone:many:at:the:border} implies that~$v\in \NTildeZero{G}{A^\infty(G)_1}\cup \NTildeZero{G}{A^\infty(G)_2}$.

We argue that for each~$i\in \{1,2\}$, there are exactly two vertices~$v$ with~$\deg(v)<\maxdeg(G)$ for which $N[v]\cap \Vmaxdeg(G)$ is a twin equivalence class in~$A^\infty(G)_i$: Indeed, if~$v_1,\ldots,v_t$ is a maximal sequence of vertices in~$A^\infty(G)_i$ containing~$v$ such that~$v_1\in A_i$ and~$|N[v_i] \setminus N[v_{i+1}]|= 1$, then~$v$ is either the single vertex in~$N[v_1] \setminus N[v_{2}]$ or the single vertex in~$N[v_{t-1}] \setminus N[v_{t}]$. Note that the single vertex~$o_i$ in~$N[v_1] \setminus N[v_{2}]$ is the outsider in~$O_i$. Let us call the other vertex~$x_i$.

If~$\NTildeZero{G}{A^\infty(G)_1}= \NTildeZero{G}{A^\infty(G)_2}$ then there are exactly two vertices with Property~\ref{qone:isdegree} and~\ref{qone:prop:outsider} in $\NTildeZero{G}{A^\infty(G)_1}\cup\NTildeZero{G}{A^\infty(G)_2}$, namely the vertices in~$O_1\cup O_2$.

We may thus suppose $\NTildeZero{G}{A^\infty(G)_1}\neq \NTildeZero{G}{A^\infty(G)_2}$.
For each~$i\in \{1,2\}$, we argue that~$x_i$ is not in 
$\NTildeZero{G}{A^\infty(G)_1}\cup\NTildeZero{G}{A^\infty(G)_2}$, which finishes the proof.
Note first that~$x_i\in  \NTildeZero{G}{A^\infty(G)_{\neg i}}$ implies that~$\NTildeZero{G}{A^\infty(G)_1}= \NTildeZero{G}{A^\infty(G)_2}$, which we ruled out.
So it remains to show that~$x_i\notin  \NTildeZero{G}{A^\infty(G)_{i}}$.
The vertices in~$N[x]\cap \Vmaxdeg(G)$ have no neighbor outside of~$\NTildeZero{G}{A^\infty(G)_i}$. We argue that~$x_i$ has a neighbor outside of~$\NTildeZero{G}{A^\infty(G)_i}$. Indeed, otherwise no vertex has a neighbor outside of~$\NTildeZero{G}{A^\infty(G)_i}$, which implies that~$A_{\neg i}$ must be in~$\NTildeZero{G}{A^\infty(G)_i}$ which implies again~$\NTildeZero{G}{A^\infty(G)_1}= \NTildeZero{G}{A^\infty(G)_2}$.

\emph{($Q^1$ is reconstructible)}
To see that the cards in~$Q^1$ are reconstructible, observe that Properties \ref{qone:isdegree} and~\ref{qone:prop:outsider} can be reconstructed for a specific card~$G_v$ and that the number of flanks of~$G$ is reconstructible.
Then the last property can be determined by inspecting all the cards satisfying the first three properties.
\end{proof}

\begin{corollary}
The multiset of cards~$ \leftmset G_v\colon \text{$v$ is an outsider of $G$} \rightmset$ is reconstructible.
\end{corollary}

\begin{proof}
	The set~$Q\coloneqq Q^1\cup Q^{\geq 2}$ contains exactly the outsiders of $G$.
	Since $\leftmset G_v\colon v \in  Q^1 \rightmset$ is reconstructible (Lemma~\ref{lem:Q1:reconstructible}),
	$\leftmset G_v\colon v \in Q^{\geq 2} \rightmset$ is reconstructible (Lemma~\ref{lemma:find:Q2}), and both multisets are disjoint also
	$\leftmset G_v \colon v \in Q^1 \cup Q^{\geq 2} \rightmset$ is reconstructible as the union of the two multisets.
\end{proof}

\section[No flank (|F1|=|F2|=0)]{No flank ($|F_1|=|F_2|=0$)}\label{sec:no:flank}
In the following, we consider the case in which the interval graph~$G$ has no flanks, that is~$|F_1|=|F_2|=0$. 
We may continue to assume that~$G$ is connected, has no universal vertex, and $\maxdeg(G) \geq 3$.
Note that by Lemma~\ref{lem:extreal:max:deg:vertices}, the set~$A$ is composed of exactly two twin equivalence classes~$A_1$ and~$A_2$.

\begin{lemma}
For each~$i\in \{1,2\}$, the triple~$\partition_G(O_i)$ is a clean clique separation.
\end{lemma}
\begin{proof}
By Lemma~\ref{lem: outsider-in-interval-rep}, the assumptions for Lemma~\ref{lem:tilde:sets:give:clean:clique:sep} Part~\eqref{lem:item:no:tilde} are satisfied and, hence, $\partition_G(O_i)$ is a clean clique separation.
\end{proof}

\begin{remark}
	Since we assume that no vertex of $G$ is universal, the condition $\Vmaxdeg(G) \neq A_i$ is always satisfied in case $F_G = \emptyset$.
\end{remark}

Recall that~$H_{O_i}$ is the annotated induced subgraph~$G[O_i\cup N_G(O_i)]$.

\begin{lemma} \label{lem:HO1HO2:reconstructible}
The multiset~$\leftmset H_{O_1},H_{O_2} \rightmset$ is reconstructible (up to isomorphism of annotated graphs).
\end{lemma}

\begin{proof}
Assume first that~$\min \{|O_1|,|O_2|\} >1$.
If~$o\in O_i$, then we obtain from Lemma~\ref{lem: steal vertex from outsider class} that $O_i\setminus \{o\}\subseteq F_{G-o}$ and $O_{\neg i} = O(G-o)$.
No vertex in~$N_G[O_{\neg i}]$ is adjacent to~$o$, so~$H_{O_{\neg i}}$ is the graph induced by $N_{G_o}[O(G_o)]$ in~$G_o$ (including the annotation).

\medskip
If~$|O_1| = |O_2| = 1$, then by Lemma~\ref{lem:Q1:reconstructible}, we may inspect the two cards~$G_{q_1}$ and~$G_{q_2}$ in~$Q^1$.
Let~$i \in \{1,2\}$. We determine~$\deg_G(q_i)$. The neighbors of~$q_i$ are a clique in~$G$ and the multiset of their degrees can be reconstructed (see Lemma~\ref{lem: reconstruct-degree-sequence}). This determines~$H_{O_i}$ as an annotated graph.

\medskip
Finally, assume that~$|O_1|=1$ and~$|O_2|>1$.
With the previous argument, we determine~$H_{O_1}$.
In the following, we determine~$H_{O_2}$.
By Lemma~\ref{lem:Q1:reconstructible}, we can determine the card~$G_{o}$ with~$o\in O_1$. We argue that~$G_{o}$ has two equivalence classes of maximum degree vertices with only one side.
Indeed, it has the twin equivalence class~$A_{2}$ with only one side. It also has a twin equivalence class~$A_1'$ of maximum degree vertices  (Lemma~\ref{lem:one:flank:means:max:deg:rivet}) adjacent to all vertices of~$N[A_1]$ except~$o$.
We now argue that~$A_1'$ is not~$A_{2}$. Indeed, if this was the case, then vertices in~$A_1$ would be adjacent to all vertices of~$G$ except those in~$O_{2}$ and vertices in~$A_{2}$ would be adjacent to all vertices except those in~$O_1$, but this means vertices in~$A_1$ and~$A_2$ have different degrees, which cannot be.

Set $O_1' \coloneqq O(G-o) \cap A_1'$.
The multiset $\leftmset H_{O_1'}, H_{O_{2}} \rightmset$ (up to isomorphism of annotated induced subgraphs) can be reconstructed from $G_o$.
However, we need to determine which of these graphs is~$H_{O_{2}}$. For this it suffices to determine the isomorphism type of~$H_{O_1'}$. This can be done by inspecting a card~$G_{x}$ with~$x\in O_{2}$. Such a card has one flank and one maximum-degree equivalence class with only one side (by Lemma~\ref{lem: steal vertex from outsider class}), which must be~$A_1$. We can find a vertex~$o\in O_1$ by choosing any vertex~$o$ in~$N[A_1]$ with~$N[N[o]]\subseteq N[A_1]$. If we delete~$o$, we can recover~$H_{O_1'}$ as the graph induced by the one class of outsides in~$G_x-o$.
 Here, we use the fact that no vertex in~$H_{O_1'}$ is adjacent to a vertex from~$O_{2}$ (in particular not~$x$). This is the case since~$A_1'\neq A_{2}$.
\end{proof}

Note that with our notation defined before, we have that~$H_{\NTwo{G}{O_i}}$ is the annotated induced subgraph~$G\left[\NTwo{G}{O_i}\cup N_G(O_i)\right]$. We are ready to show that interval graphs without flanks are reconstructible.

\begin{lemma}\label{lem:no:flanks:reconstruct}
If~$|F_1|=0$ and~$|F_2|=0$, then~$G$ is reconstructible.
\end{lemma}

\begin{proof}
We distinguish cases according to whether~$\max \{|O_1|,|O_2|\} >1$. 

\paragraph{Case 1. $\mathbf{\max \{|O_1|,|O_2|\} >1}$.}
By symmetry, we may assume that $|O_2| \geq 2$.
For $v \in O_2$ set~$Y_v \coloneqq {N}_{G-v}(O_2\setminus \{v\})$.
Observe that $Y_v$ can be reconstructed from~$G_v$: there is precisely one non-neighborhood-contained twin equivalence class $A_2'$ in $G_v$ of vertices of degree $\maxdeg(G)-1$ which is extremal with respect to \say{$\prec$} among all non-neighborhood-contained twin equivalence classes of vertices of degree at least $\maxdeg(G)-1$.
The (unlabeled) vertices of $O_2\setminus\{v\}$ are precisely the ones in $\NTildeZero{G_v}{A_2'}\setminus N_{G_v}[\NTildeOne{G_v}{A_2'}]$.

Choose~$x\in O_2$ such that~$\sum_{y\in Y_x} \deg(y)$ is maximal. 
By the choice of~$x$, a vertex in~${N}_G(O_2)$ can only be adjacent to~$x$ if it is also adjacent to all other vertices of~$O_2$. Thus, in~$G$, 
for each vertex~$t$ in~${N}_G(O_2) \cap N_G(x)$, we have~$\deg(t)\geq |O_2|+|{N}_G(O_2)|$ (since ${N}_G(O_2)$ is a clique and there is at least one vertex in $\NTwo{G}{O_2}$ to which every vertex of  ${N}_G(O_2)$ is adjacent). In contrast, for vertices~$m$ in~$O_2$, we have~$\deg(m)\leq |O_2|+|{N}_G(O_2)|-1$. 

Now, we consider the graph~$(G-x)\left[{N}_G(O_2)\cup \NTwo{G}{O_2}\right]$. This graph is essentially $H_{\NTwo{G}{O_2}}= G[{N}_G(O_2)\cup \NTwo{G}{O_2}]$ except that the annotation of some vertices has been reduced by exactly one. This is the case for exactly the vertices that are neighbors of~$x$ in~$G$.
We can alter $(G-x)\left[{N}_G(O_2)\cup \NTwo{G}{O_2}\right]$ into~$H_{\NTwo{G}{O_2}}$ as follows. For each neighbor~$w$ that~$x$ has with degree at least~$|O_2|+|{N}_G(O_2)|$, we pick a vertex~$y\in (G-x)\left[{N}_G(O_2)\cup \NTwo{G}{O_2}\right]$ (arbitrarily) that has annotation~$\deg(w)-1$ and a degree in~$(G-x)\left[{N}_G(O_2)\cup \NTwo{G}{O_2}\right]$ of $\deg(w)-|O_2|$.

We alter its annotation by increasing it by 1 to~$\deg(w)$. (Vertices whose annotation has been altered once will not be altered again.)

Note that since neighborhoods are linearly ordered, two vertices~$w_1,w_2$ in~${N}_G(O_2)\cup \NTwo{G}{O_2}$ which have the same degree in~$G_x\left[\NTildeOne{G}{O_2}\cup \NTildeTwo{G}{O_2}\right]$ and which have the same annotation are twins in~$G_v$. It thus does not matter whether we increase the annotation of the vertex~$w_1$ instead of the vertex~$w_2$.

We have thus managed to reconstruct the graph~$H_{\NTwo{G}{O_2}}$.
By Lemma~\ref{lem:HO1HO2:reconstructible}, also $\leftmset H_{O_1},H_{O_2}\rightmset$ is reconstructible. Note further, as argued in the proof of Lemma~\ref{lem:HO1HO2:reconstructible}, from~$G_x$, we can determine which graph is~$H_{O_{1}}$. We can identify~$H_{O_2}$ up to isomorphism by choosing a graph in $\leftmset H_{O_1},H_{O_2}\rightmset$ so that the other graph is isomorphic to~$H_{O_{1}}$. 

Overall, we have determined~$(H_{O_2},H_{\NTwo{G}{O_2}})$ and by Lemma~\ref{lem:reconst:clique:sep:implies:reconst} and Lemma~\ref{lem:tilde:sets:give:clean:clique:sep}, the graph is reconstructible.

\paragraph{Case 2. $\mathbf{|O_1|= |O_2| =1}$  and $\mathbf{\maxdeg(G)= |V(G)|-2}$.} 
Say~$O_1=\{o_1\}$ and~$O_2= \{o_2\}$. 
In particular, for~$i\in \{1,2\}$, we have~$N_G[A_i]= V(G)\setminus\{o_{\neg i}\}$. Note that $\leftmset G_v\colon v \in A_1 \cup A_2 \rightmset$ is reconstructible since it is the multiset of all cards corresponding to maximum degree vertices of $G$.

We argue that $G$ is reconstructible if there exists a $(\maxdeg(G)-1)$-vertex $x$ in~$G$.
Since $\maxdeg(G) \geq 3$ and $|O_1| = |O_2| = 1$, we have that $x \notin O$.
By Lemma~\ref{lem: properties of o}, we have for $i \in \{1,2\}$  $N_G(o_i) \cap \Vmaxdeg(G) = A_i$ and $N_G[o_i] \subseteq N_G[A_i]$.
Hence, we can locate~$O$ in $G_x$ as the set of vertices not adjacent to two equivalence classes of vertices of degree~$\maxdeg(G)-1$.
 Inspecting the degrees of these vertices we can determine in~$G_x$ whether~$x$ has a neighbor in~$O$.
If~$x$ has no neighbor in~$O$, then we reconstruct~$G$ from~$G_x$ by adding a vertex adjacent to all vertices except those in~$O$.
The other case is that all $(\maxdeg(G)-1)$-vertices of $G$ are adjacent to a vertex of~$O$.
For this case, consider a card~$G_a$ with~$\deg_G(a)=\maxdeg(G)$ in a class~$A_i$ such that~$O_i$ has a neighbor of degree~$\maxdeg(G)-1$.
If~$a$ can be chosen from both~$A_1$ and~$A_{2}$, then we further assume that~$|A_i|\leq |A_{\neg i}|$. Note that we can determine whether~$a$ is contained in~$A_i$ with~$|A_i|\leq |A_{\neg i}|$ by ensuring~$a$ has maximum degree in~$G$ and choosing~$a$ so that it maximizes the size of the largest twin equivalence class of vertices of degree~$\Delta(G)-1$ in~$G_a$.

From~$G_a$, we reconstruct $G$ as follows: Locate a vertex~$x$ with~$\deg_{G_a}(x)=\maxdeg(G)-2$. If possible, pick~$x$ so that it is adjacent to a vertex which itself is not adjacent to the larger equivalence class of $(\maxdeg(G)-1)$-vertices.
There are two vertices not adjacent to~$x$. Insert~$a$ back into the graph and join it to all vertices except a non-neighbor of~$x$ of smallest degree in~$G_a$.

Next we argue that $G$ is reconstructible if~$G-(A_1\cup A_2)$ is not connected:
By the previous paragraph, we may assume that $G$ is free of $(\maxdeg(G)-1)$-degree vertices.
Consider a card $G_x$ with $x \in A_1 \cup A_2$.
Observe that $G_x- \Vmaxdeg(G_x)$ is isomorphic to $G-(A_1 \cup A_2)$ and, hence,
we can reconstruct the isomorphism types of the connected components of~$G-(A_1\cup A_2)$. If we can determine the isomorphism type of the graphs induced by the connected components~$C_1$ and~$C_2$ that contain~$o_1$ and~$o_2$, respectively, and can locate the vertex~$o_1$, respectively~$o_2$, in the component, the we can reconstruct the graph\footnote{Formally we can reconstruct~$\leftmset [(G[C_1],o_1)]_{\cong}, [(G[C_2],o_2)]_{\cong}\rightmset$)}. 
For this note that by deleting a vertex from~$A_i$ (where~$|A_i|\leq |A_{\neg i}|$), we can determine~$o_i$ and the isomorphism type of the connected component of~$G-(A_1\cup A_2)$ that contains~$o_i$ together with the vertex~$o_i$. We can also determine~$|A_i|$. We now consider~$G-a_{\neg i}$ with~$a_{\neg i}\in A_{\neg i}$. If~$|A_i|=|A_{\neg i}|$, we obtain the graph induced by the other component~$C_{\neg i}$ as before. Otherwise, if~$|A_i|<|A_{\neg i}|$, in~$G_{a_{\neg i}}$ there are exactly two vertices not adjacent to all maximum degree vertices (corresponding to~$o_i$ and~$o_{\neg i}$). This gives rise to two connected components each with a vertex singled out. We can identify the one belonging to~$o_i$ or they are isomorphic. In either case we can identify the isomorphism type of the connected components belonging to~$o_{\neg i}$. Overall if~$G-(A_1\cup A_2)$ is disconnected, we can reconstruct~$G$.

From now on, we may assume that~$Z \coloneqq G-(A_1\cup A_2)$ is connected and there is no $(\maxdeg(G)-1)$-vertex in~$G$. 

As usual, the equivalence classes of the vertices in~$Z$ are linearly ordered up to reversal (Lemma~\ref{lem:order:of:max:deg:vertices}). This includes all the vertices of $\Delta(Z)$.

Since~$o_i$ is an outsider of~$G$, by Lemma~\ref{lem: outsider-in-interval-rep}, in the graph~$Z$ there is a unique twin equivalence class of non-neighborhood-contained vertices~$Z_i$ adjacent to~$o_i$. Moreover, the class~$Z_i$ is the first or last equivalence class in the linear ordering of non-neighborhood-contained vertices of~$Z$.

For each~$i$ let~$D_i$ be the bordering equivalence class of~$Z$ that is closest to~$Z_i$. (It is possible that~$Z_i= D_i$. It is also possible that~$Z_i= D_{\neg i}$ but then~$D_i=D_{\neg i}$.)

We will argue now that in~$Z$, for at least one~$j\in \{1,2\}$, we have that~$o_{\neg j}$ is contained in a side~$R_j$ of~$D_j$. In fact, we will argue the following stronger statement:

\begin{claim} \label{claim:someside:contains:two}
In~$Z$, for at least one~$j\in \{1,2\}$, we have that~$o_{\neg j}$ and (at least) one additional vertex is contained in a side~$R_j$ of~$D_j$. 
\end{claim}

\noindent $\ulcorner$ 
For each~$j\in \{1,2\}$, since~$\maxdeg(D_{\neg j}) \leq \maxdeg(G)-2\leq n-4$
in total at least 3 vertices are in the sides of~$D_{j}$.
Thus one side contains at least 2 vertices. That side contains~$o_j$ or~$o_{\neg j}$. If it contains~$o_{\neg j}$ we are done. Otherwise the side contains~$o_j$ and some other vertex~$x$. This implies that some side of~$D_{\neg j}$ contains~$o_j$ and the other vertex~$x$ (since~$D_{\neg j}$ is farther away from~$Z_j$ than~$D_j$).\hfill$\lrcorner$

Let~$i\in \{1,2\}$ now be such that a side~$R_i$ of~$D_i$ contains~$o_{\neg i}$ and an additional vertex.

Note that from~$G_{o_i}$ we can determine whether $o_{\neg i}$  and an additional vertex is contained in a side~$R_i$ of~$D_{i}$ in~$Z$.

We obtain a clean clique separation~$\partition_Z(R_i)$. This separation of~$Z$ extends to a separation~$(R_i,N_Z(R_i)\cup A_1 \cup A_2,\NTwo{Z}{R_i})$ of~$G$ with linearly ordered neighborhoods.

From~$G_{o_i}$ we can recover~$H_{R_i}$ since vertices in~$R_i\cup N_Z(R_i)$ are not adjacent to~$o_i$ (Lemma~\ref{lemma:farthest:away:and:at:least:two:reconstruct:other:side}). (Here we also use that in~$G_{o_i}$ we can locate~$o_{\neg i}$ using the vertices in~$A$ and thus can locate~$R_i$.)

Since~$R_i$ contains at least one other vertex besides~$o_{\neg i}$, from~$G_{o_{\neg i}}$ we can recover~$H_{\NTwo{Z}{R_i}}$ (Lemma~\ref{lemma:farthest:away:and:at:least:two:reconstruct:other:side}). (Here we also use that in~$G_{o_{\neg i}}$ we can locate~$o_{i}$ using the vertices in~$A$ and thus can locate~$R_i\setminus \{o_{\neg i}\}$.)

This implies that, for some~$i\in \{1,2\}$, we can reconstruct~$(H_{R_i}, H_{\NTwo{Z}{R_i}})$ and hence the graph~$G$. 

\paragraph{Case 3. $\mathbf{|O_1|= |O_2| =1}$  and $\mathbf{\maxdeg(G) \neq |V(G)|-2}$. }
Say~$O_1=\{o_1\}$ and~$O_2= \{o_2\}$. In particular~$N_G[A_i]\neq V(G)\setminus\{o_{\neg i}\}$ for one and thus both~$i\in \{1,2\}$. 
By Lemma~\ref{lem:one:flank:means:max:deg:rivet}, besides~$A_1$ and~$A_2$ there is at least one other twin equivalence class of $\maxdeg(G)$-vertices in~$G$.

For $i$ in $\{1,2\}$ by Lemma~\ref{lem:one:flank:means:max:deg:rivet}, we have $F_{G_{o_i}} = \emptyset$.  We define~$O_i^+$ as the outsider class of~$G_{o_i}$ that is not a class of outsiders in~$G$. Set~$\bar O_i^+ \coloneqq \{o_i\} \cup O_i^+$.
We obtain a clean clique separation~$\partition_G(\bar O_i^+)$.

We can reconstruct~$\leftmset G_x\colon x\in O_j^+ \text{ and } |O_j^+|=1 \text{ for some $j \in \{1,2\}$ }\rightmset$ since this is exactly the set of cards $G_x$ with $\maxdeg(G_x) = \maxdeg(G)$ and a flank of size~$1$: indeed,~$F_{G_{o_i}} = \emptyset$ if~$o_i\in O_i$ and for~$x\notin O(G)$, if~$G_x$ has a flank, then that flank contains~$O_1 \cup O_1^+\setminus \{x\}$ or $O_2 \cup O_2^+\setminus \{x\}$.
Note that $|O_i^+|$ is the maximum size of an outsider class in $G_{o_i}$.
\begin{claim}
		If~$|O_i^+|>1$, then from~$G_{o_i}$ we can reconstruct~$H_{\overline{N_G[ \bar O_i^+]}}$ and~$H_{\bar O_{\neg i}^+}$.
\end{claim}
\noindent $\ulcorner$ We can locate~$O_i^+$ in~$G-o_i$ as the outsider class of size larger than one. A vertex adjacent to~$O_i$ is adjacent to all vertices of~$O_i^+$. We can thus apply Lemma~\ref{lemma:farthest:away:and:at:least:two:reconstruct:other:side} to reconstruct~$H_{\overline{N_G[ \bar O_i^+]}}$ and~$H_{\bar O_{\neg i}^+}$.
\hfill $\lrcorner$

\begin{claim}
	If~$|O_i^+|=1$, say $O_i^+ = \{o_i^+\}$, then from~$G_{o_i^+}$ we can reconstruct~$H_{\bar O_{\neg i}^+}$.
\end{claim}
\noindent $\ulcorner$ Observe that~$O_{\neg i}$ is the unique outsider class in~$G-{o_i^+}$ for~$O_i^+= \{o_i^+\}$.
The vertex~$o_i^+$ may have common neighbors with the vertices in~$H_{\bar O_{\neg i}^+}$. However, such vertices have degree~$\maxdeg$ in~$G$, are twins, and adjacent to all of~$\bar O_{\neg i}^+$. This means that from inspecting~$G_{o_i^+}$ we can determine the number of such neighbors of~$\bar O_{\neg i}^+$ and increase the annotation of degree~$\maxdeg-1$ vertices adjacent to all of~$O_{\neg i}^+$ but not to~$O_{\neg i}$ accordingly.
Thus, we may apply Lemma~\ref{lemma:farthest:away:and:at:least:two:reconstruct:other:side} to reconstruct~$H_{\bar O_{\neg i}^+}$.
\hfill $\lrcorner$

Together these two claims imply that if~$|O_i^+|>1$ for some~$i\in\{1,2\}$, we can reconstruct the graph, since we can reconstruct a pair~$(H_{\bar O_{ i}^+},H_{\overline{N_G[ \bar O_i^+]}})$ and apply Lemma~\ref{lem:reconst:clique:sep:implies:reconst}.

It remains to consider the case~$|O_1|=|O_1^+|=|O_2|=|O_2^+|=1$.
Say, $O_i^+ = \{o_i^+ \}$ for $i \in \{1,2\}$.
Define the \emph{outsider length} of $O_i$ 
 to be the length~$s$ of the longest chain of vertices~$(o_i = x_1,x_2,x_3,\ldots,x_s)$ for which~$\{x_i\}$ forms an outsider class in~$G-\{x_1,\ldots,x_{i-1}\}$ but not in~$G-\{x_1,\ldots,x_{i-2}\}$.
Note that $o_i = x_1$ and $o_i^+ = x_2$ in the outsider chain of $O_i$ and hence, both outsider lengths are at least~$2$.
Note that we can reconstruct the maximum length of an outsider chain in $G$ from the two cards $G_{o_1}$ and $G_{o_2}$.
W.l.o.g.,~the outsider length of~$O_{\neg i}$ is not larger than that of~$O_{i}$.

\begin{itemize}

\item 
Let us first assume that for the longest chain $(o_{i},o_i^+, x_3,\ldots,x_s )$ we have the property that $\maxdeg (G-\{o_{i},o_i^+,x_3,\ldots,x_s\}) = \maxdeg(G)$.
In particular, the outsider length of $O_{i}$ in $G-{o_{\neg i}}$ is the same as its length in $G$.
In~$G_{o_{\neg i}}$, we find the outsider class that is not an outsider class in~$G$ (intuitively where~$o_{\neg i}$ was removed), it is the one with shorter outsider length. We can reinsert~$o_{\neg i}$ since every vertex of degree~$d$ adjacent to~$o_{\neg i}$ is a vertex of degree~$d-1$ adjacent to~$O_{\neg i}^+$ in~$G_{o_{\neg i}}$ and all vertices with the same degree adjacent to~$O_{\neg i}^+$ in~$G_{o_{\neg i}}$ are twins.

\item Finally, assume~$\maxdeg (G-\{o_{i},o_i^+,x_3,\ldots,x_s\}) <\maxdeg(G)$. In that case, both outsider lengths of~$G$ decrease when removing a single vertex from~$O$.
As argued before, we can reconstruct~$\leftmset G_x\colon x\in O_i^+\text{ and } |O_i^+|=1\text{ for some $i \in \{1,2\}$}\rightmset=\leftmset G_x\colon x\in O_1^+ \cup O_2^+ \rightmset$. 
If~$s>3$, then in~$G_{o_i^+}$, we can recognize~$o_i$ and thus~$x_3$. We can then reconstruct~$G$ since we know the degrees of vertices adjacent to~$o_i$ and also of vertices adjacent to~$x_3$.

Otherwise~$s=2$ and there are exactly three twin equivalence classes of vertices of maximum degree in~$G$, namely~$A_1$,~$A_2$, and a third class~$A'$. We can reconstruct the cards~$G_a$ with~$a\in A'$ as the cards in which the number of pairs of vertices of degree~$\Delta(G)-1$ that dominate the graph is maximal. In particular, we can reconstruct $|A'|$.

If~$|A'|>1$, we can in~$G_a$ duplicate a non-bordering maximum degree vertex to reconstruct~$G$.
If~$|A'|=1$ then we can reinsert~$a$ in~$G_a$ to reconstruct~$G$ as follows. For each~$A_i$ there are exactly two vertices not adjacent to~$A_i$ in~$G-a$. If they are not twins, the vertex with the larger degree is adjacent to~$A'$. We thereby locate two vertices adjacent to~$A'$. On top of these, all vertices adjacent to both~$A_1$ and~$A_2$ are adjacent to~$A$.\qedhere
\end{itemize}

\end{proof}

\section[One flank (|F1|=0 and |F2|>0)]{One flank ($|F_1|=0$ and~$|F_2|>0$)}\label{sec:one:flank}

In the following, we consider the case in which the graph~$G$ has precisely one non-empty flank~$F_2$.  
We may continue to assume that~$G$ is connected, has no universal vertex, and $\maxdeg(G) \geq 3$.

By Lemma~\ref{lem: E-recognizable} and Lemma~\ref{lem:each:flank:represented:in:E}, the multiset of cards~$\mathcal{E}_2$ is non-empty and can be reconstructed from the deck and $G_v \in \mathcal{E}_2$ implies $v \in F_2$.
By Lemma~\ref{lem:extreal:max:deg:vertices}, there is one twin equivalence class~$A_1$ of vertices of maximum degree in~$G$ which have exactly one side and, hence,~$O(G)=O_1$.

\begin{lemma}\label{lem:one:flank:size:more:than:one:reconstruc}
If~$|F_1|=0$ and~$|F_2|>1$, then~$G$ is reconstructible.
\end{lemma}

\begin{proof}
We recover~$H_{\NTwo{G}{F_2}}$ as follows:
For~$v$ with $G_v \in \mathcal{E}_2$, define~$Y_v \coloneqq N_G[\bulk(G_v)] \cap N_G(F_{G_v})$. 
Choose a card~$G_x$ in~$\mathcal{E}_2$ such that~$\sum_{y\in Y_x} \deg(y)$ is maximal.
The graph~$H_{\NTwo{G}{F_2}}$ can be recovered with annotation by considering the graph~$G_x\left[\NTwo{G_x}{F_2}\cup {N}_{G_x}(F_2)\right]$   and increasing the annotation of certain vertices by one. The annotation of which vertices has to be increased is governed by the degree of the vertices in~$G_x\left[\NTwo{G_x}{F_2}\cup {N}_{G_x}(F_2)\right]$ 
(see the first case of Lemma~\ref{lem:HFi:reconstr} for details).

We argue now that~$H_{F_2}$ is reconstructible.
We distinguish three cases. For this let $v$ be in~$O_1$.

\paragraph{Case 1. $\mathbf{G_v}$ has maximum degree~$\mathbf{\Delta(G)}$ and only one flank.}
In this case then~$H_{F_2}$ is induced by the vertices from the flank and their neighbors. The annotation is the same in~$G_v$ and~$G$ since the vertices of $H_{F_2}$ are not adjacent to~$v$ in~$G$.

\paragraph{Case 2. $\mathbf{G_v}$ has maximum degree~$\mathbf{\Delta(G)}$ and two flanks.}
In particular, there are two annotated graphs induced by the flanks and their neighbors, namely~$H_{F_2}$ and a second graph~$H'_v$. We need to determine which one is which. We consider first all cards~$G_v$ with~$v\in O_1$. We can assume that the graph~$H'_v$ is always the same for all~$v$, otherwise we can identify~$H_{F_2}$ as the graph that is always present in the pair~$(H_{F_2},H'_v)$.

Consider a card~$G_x \in \mathcal{E}_2$.
Since $|F_2|>1$, we have $O_1 = O(G-x)$. Hence, we can reconstruct~$H'_v$ as the graph induced by the vertices in~$O_1\setminus \{v\}$ and their neighbors (where~$v\in O_1$ is arbitrary). We have thereby reconstructed~$H'_v$ and thus~$H_{F_2}$. 

\paragraph{Case 3. $\mathbf{\Delta(G-v) < \Delta(G)}$.}
In this case $\Delta(G-v)= \Delta(G)-1$ (since $v$ is an outsider).
The graph~$G-v$ has a twin equivalence class of $\Delta(G-v)$-vertices with only one side, namely~$A_1$.
If it does not have two such classes, then we can reconstruct~$H_{F_2}$ as graph induced by the vertices not in~$\Span_{G-v}(A_1)$.
If otherwise $G-v$ has a second such class~$A'$, then we reconstruct a multiset of two graphs~$\leftmset H_{F_2}, H'\rightmset$, where $H'$ is the graph induced by the vertices not in $\Span_{G-v}(A')$.
The class~$A'$ is adjacent to at least two vertices of~$F_2$.

Let~$X$ be the set of all vertices satisfying~$G_x \in \mathcal{E}_2$ and $x$ is an outsider of~$A'$ in~$G-v$. 

\begin{claim} \label{claim:recognize:X}
	The multiset~$\leftmset G_x\colon x\in X\rightmset$ is reconstructible.
\end{claim}
\noindent $\ulcorner$ 
Note first that the distance between~$A_1$ and a vertex~$x\notin \Span(A_1)$ is the same in~$G$ and in~$G-v$.

Next we argue that in~$G$ (and thus~$G-v)$ the vertices in~$X$ are exactly the vertices whose distance from~$A_1$ is maximal: by Lemma~\ref{lem: outsider-in-interval-rep}, we may assume that $X = \{v\colon r_v < \min\{\ell_u\colon u \in N_{G-v}(\Span(A')) \}  \}$ in some tidy interval representation of~$G-v$. Thus for every shortest path of maximum length the penultimate vertex is in~$A'$ and the last vertex is in~$X$.

We conclude that the set~$X$ consists exactly of those vertices~$y$ for which 
\begin{enumerate}
\item $M(y)\coloneqq \max\{d_{G-y}(A_1,w)\colon w\in V(G)\}$ is as small as possible, and\label{prop:small:diam},
\item among the cards satisfying Property~\ref{prop:small:diam}, we have that $|\{u\colon d_{G-y}(A_1,u)= M(y)\}|$ is as small as possible.
\end{enumerate}

We conclude that~$X$ can be recognized.\hfill~$\lrcorner$

Fix a vertex $x \in X$.
In~$G-x$ there is one twin equivalence class of $\Delta(G-x)-$vertices. That class is~$A_1$. If there is a non-neighborhood-contained bordering class of degree~$\Delta(G)-2$, then that class is~$A'$.
We can thus assume that $A'$ is neighborhood-contained by another class of degree~$\Delta(G)-1$ in~$G-x$, say~$\widetilde{A}$.
The class~$\widetilde{A}$ dominates at most two equivalence classes of vertices of degree~$\Delta(G)-2$. One of them is~$A'$. Let us call the other one~$A''$. 
There is exactly one vertex~$w'$ adjacent to~$\widetilde{A}$ but not~$A'$.
Similarly, there is exactly one vertex~$w''$  adjacent to~$\widetilde{A}$ but not~$A''$.
The vertex~$w'$ has a neighbor outside of~$\Span_{G-x}(\widetilde{A})$ while~$w''$ does not. We can use this to distinguish~$A'$ from~$A''$. Overall, we can reconstruct~$A'$, thereby reconstruct~$H'$ and thus~$H_{F_2}$.

Overall, we have reconstructed~$(H_{\NTwo{G}{F_2}}, H_{F_2})$ and the result follows with Lemma~\ref{lem:reconst:clique:sep:implies:reconst}.
\end{proof}

\begin{lemma}\label{lem:one:flank:size:one:reconstruc}
If~$|F_1|=0$ and~$|F_2|=1$, then~$G$ is reconstructible.
\end{lemma}

\begin{proof} By the assumptions of this lemma, we have $|F_G| = |F_2|  = 1$, say $F_2 = \{f_2\}$. By Lemma~\ref{lem: E-recognizable} and Lemma~\ref{lem:each:flank:represented:in:E}, we can reconstruct $G_{f_2}$.
There is precisely one non-empty outsider class of $G$, say $O_1 = O(G)$, by Lemma~\ref{lem:extreal:max:deg:vertices} and Lemma~\ref{lem:anonempty-ononempty}.  
	
	\medskip
	\noindent
\textbf{Case 1. $\mathbf{G}$ has only one twin-equivalence class $\mathbf{A}$ of $\mathbf{\Delta(G)}$-vertices.}
Let $a$ be in $A$. If~$|A|>1$, then we can reconstruct $G$ as follows: 
Note first that the vertices in $A\setminus \{a\}$ form the unique twin equivalence class of $(\Delta(G)-1)$-vertices in $G-a$. Hence, we can reconstruct $G$ from~$G_a$ by duplicating a $(\Delta(G)-1)$-vertex.

If~$G_a$ is not connected, then we can reconstruct $G$ as follows: Let $C_1, C_2, \dots, C_k$ be the components of~$G_a$.
Note that $A$ is the unique class of $\Delta(G)$-vertices in $G-f_2$ (since $f_2$ is a flank vertex). In particular, we can reconstruct $G-\{a, f_2\}$ from $G_{f_2}$.
This allows us to reconstruct the isomorphism type of the component of $G-a$ containing $f_2$, w.l.o.g., it is $C_k$. Further, we can locate~$f_2$ inside the component\footnote{This formally means we can reconstruct $[(C_k,f_2)]_{\cong}$.}.
We can reconstruct $G$ from the disjoint union of $C_1, C_2, \dots, C_{k-1}, C_k$ joining $a$ to every vertex of $C_1, C_2, \dots, C_{k-1}, C_k$.

We may assume from now on that~$A= \{a\}$ and~$G_a$ is connected.
By Lemma~\ref{lem:one:flank:means:max:deg:rivet}, we know that~$|O_1|>1$.

\medskip
\noindent \textbf{Case 1.a. $\mathbf{|O_1|=2}$.}
 We prove that there is a $(\Delta(G)-1)$-vertex in $N_G(f_2)$.
 All vertices in~$\Span_G(A)\setminus (O_1)$ have a neighbor in~$N_G(\Span_G(A))$ by the definition of outsiders. Furthermore, $N_G(\Span_G(A)) = N_G(f_2)$.
 From Lemma~\ref{lem:flanks:create:clean:clique:sep}, we conclude that $\partition_G(\Span_G(A)) = (\Span_G(A), N_G(f_2), \{f_2\})$ is a clean clique separation and, hence, has linearly ordered neighborhoods.
 Thus, there is a vertex $x \in N_G(f_2)$ adjacent to every vertex of $\Span_G(A)\setminus (O)$. Altogether $N_G(x) =  \Span_G(A)\setminus O \cup N_G(\Span_G(A))\setminus \{x\} \cup \{f_2\} = (N_G[\Span_G(A)] \cup \{f_2\}) \setminus (O \cup \{x\}) $. Since $|N_G[\Span_G(A)] \cup \{f_2\}| = \maxdeg(G) + 2$, we obtain
 $\deg_G(x) = \maxdeg(G)-1$.

Let~$X$ be the twin equivalence class of~$x$. 
If~$|X|>1$, then we reconstruct $G$ by duplicating a non-neighborhood-contained $(\Delta(G)-2)$-vertex of $G_x$. So we assume~$|X|=1$.

\medskip
If~$G-\{a,x\}$ is  not connected, then we can reconstruct $G$ as follows: We determine the isomorphism types of the connected components of~$G-\{a,x\}$
 together with the information which vertices are adjacent to~$a$ or~$x$ by inspecting cards as follows: In~$G_x$, we locate~$a$ as the only $(\Delta(G)-1)$-vertex. We then locate~$f_2$ as the only vertex not adjacent to~$a$. This means we can determine the isomorphism types of all connected components of~$G-\{a,x\}$ and in the component containing~$f_2$, we can also locate the vertex~$f_2$.
 \begin{itemize}
 
 \item Assume that~$O_1=\{o,o'\}$ with~$\deg_{G}(o)\leq \deg_{G}(o')$. Note that if~$\deg_{G}(o)= \deg_{G}(o')$, then~$o$ and~$o'$ are twins.
 In~$G_{o}$, let~$M$ be the vertices of degree~$\Delta(G)-1$. We can, in~$G_{o}$ locate exactly two connected components of~$G_{o}-M$
  each containing a special vertex that is the only one not adjacent to all vertices of degree~$\Delta(G)-1$. These two special vertices correspond to~$f_2$ and~$o'$. 
 Since we know the isomorphism type of the connected component of~$f_2$ in~$G-\{a,x\}$ (including being able to find the vertex), we can thus determine the isomorphism type of the connected component of~$o'$ in~$G-\{a,x,o\}$. If~$\deg_{G}(o)= |A_1|$ then~$o$ is an isolated vertex in~$G-\{a,x\}$ and we are done. Otherwise, we add~$o$ to the connected component of~$o'$ in~$G-\{a,x,o\}$ by the usual technique of increasing the degree of suitable vertices adjacent to~$o'$ by exactly 1. 
\end{itemize}
We argue that there are no other $(\maxdeg(G)-1)$-vertices.
There can be only two other such types. One such type is adjacent to all vertices except~$f_2$ and one vertex~$o$ from~$O_1$. If it exists, then in~$G_o$, we can distinguish~$f_2$ and the other vertex~$o'$ from~$O_1$. We can thus reconstruct the graph since we have found the vertex~$o'$ that is the neighbor of the extremal vertex~$o$ that was removed.
The other type of $(\maxdeg(G)-1)$-vertex that could exist has~$f_2$ and some other vertex~$f^+$ 
 as non-neighbors. In this case in~$G_{f_2}$, we can find~$f^+$ and then reconstruct the graph.

\medskip
Now we can assume that $Z\coloneqq G-\{a,x\}$ is connected. In $Z$ the vertex~$f\in F$ is adjacent to exactly one non-neighborhood-contained equivalence class of vertices~$Z_f$. 
In $Z$, let~$D$ be the equivalence class of $\Delta(Z)$-vertices closest to~$Z_f$ in the, up to reflection, well defined order of non-neighborhood-contained twin equivalence classes. 

First assume that~$D\neq Z$. Note that we can reconstruct whether~$D\neq Z$ since we can locate~$f$ and~$x$ in~$G_a$.
Let~$S$ be the side of~$D$ that contains~$f$ (which must exist, as otherwise~$D=Z$). 
We obtain a clean clique separation~$\partition_Z(S)$.
This separation extends to a separation~$(S,N_{Z}(S)\cup \{a,x\},\NTwo{Z}{S})$ of~$G$. We obtain~$H_{S}$ from~$G_x$ since we can locate~$f$ and since~$x$ is adjacent to all vertices of~$S\cup N_{Z}(S)$. We can reconstruct~$H_{\NTwo{Z}{S}}$ from~$G_a$ since~$a$ is adjacent to all vertices of~$N_{Z}(S)\cup \{a,x\}\cup \NTwo{Z}{S}$.
Overall, we found a clean clique separation.

If~$D=Z$, the argument is similar except that we use for~$S$ the side of~$D$ that contains~$O_2$ instead of~$f$.

\medskip
\noindent \textbf{Case 1.b. $\mathbf{|O_1|>2}$.}
Let~$x_1,\ldots,x_t$ be a maximal sequence of vertices where~$x_1=f_2$ and for each~$i>1$ the vertex~$x_i$ has degree~$2$ in~$G-a$ (possibly~$t=1$).
If this sequence contains all vertices of~$G-a$, then~$G$ is a path with an added vertex of degree~$n-1$ and is reconstructible.
Otherwise, let~$w$ be a neighbor of~$x_t$ of maximum degree in~$G-a$ and~$w'$ a neighbor of~$x_t$ of minimum degree not in~$\{w,x_{t-1}\}$.

\begin{itemize}
\item If~$N_G[\{x_1,\ldots,x_t,w\}] = V(G)$, then we reconstruct~$G$ as follows:
We find a card~$G_{v}$, with~$v$ a twin of~$w$ or~$w$ itself. Such cards are exactly the cards~$G_v$ with~$v\notin A\cup \{F_2\}$ for which, starting from~$f$, we can find the longest path of degree 2 vertices and~$v$ has maximal degree among them.
We reconstruct~$G$ from~$G_v$ by adding a vertex adjacent to all vertices not on the path of this longest path starting from~$f$ and to the last vertex.
\item If~$N_G[\{x_1,\ldots,x_t,w\}] \neq V(G)$, then we use a card~$G_{v'}$ with~$v'$ a twin of~$w'$ or~$w'$ itself.
Set $V' \coloneqq \{v\colon N_G[v]\subseteq 
N_G[\{x_1,\ldots,x_t,w\}$ and consider the partition $\partition_G(V')$.
 
From $G_{v'}$, we can determine~$H_{\overline{N_G[V']}}$.

Inspecting a card~$G_o$ with~$o\in O_1$, we can locate the vertex~$f\in F_2$ as vertex not adjacent to a vertex of degree~$\Delta(G)-1$. We can thus recover the (unannotated) graph~$H_{V'}$ or determine that~$o\in V'$.
Choosing, among the vertices in~$O_1$ for which~$o\notin V'$, an extremal vertex (i.e., one for which~$\sum_{y\in N(V')} \deg(y)$ is as large as possible), 
we can determine~$H_{V'}$ with annotation.

Overall, we obtain $\leftmset H_{\overline{N_G[V']}}, H_{V'} \rightmset$ and, hence, the graph is reconstructible.
\end{itemize}

\paragraph*{Case 2. $\mathbf{G}$ has more than 1 equivalence class of $\mathbf{\Delta(G)}$-vertices.}

This implies that in~$G-f_2$ there are two classes of outsiders, namely~$O_1$ and a second class, which we call~$\widehat{O_2}$.

\noindent \textbf{Case 2.a.}~$\mathbf{|O_1|> 2}$. In this case, there is a set~$O^*_1 \subseteq V(G)$ containing~$O_1$ such that every graph~$G-o$ with~$o\in O_1$ has two flanks, one of size~$1$ (namely~$F_2$) and the other one being~$O^*_1\setminus o$.
(If the vertices in~$A_1$ have non-equivalent neighbors of maximum degree then~$O^*_1=O_1$, but in other cases~$O^*_1$ may be larger.)

From a card~$G_o$ with~$o\in O_1$ and~$o$ extremal (this means~$\sum_{y\in N_{G-o}(O^*_1\setminus \{o\})} \deg(y)$ as large as possible),
 we can recover~$H_{\NTwo{G}{O^*_1}}$.
For~$f\in F_2$,
we can also recover two graphs from~$G_f$, where one is~$H_{O^*_1}$ and the other is the graph that would be~$H_{O^*_2}$ in~$G-f$. However, the graph~$H_{O^*_2}$ can also be obtained from~$G_o$ by deleting~$f$.
Thus, we can reconstruct~$(H_{O^*_2}, H_{{\NTwo{G}{O^*_2}}})$.

\noindent \textbf{Case 2.b.}~$\mathbf{|O_1|=2}$. In that case, we can reconstruct~$H_{{\NTwo{G}{F_2\cup \widehat{O_2}}}}$ from~$G_f$ or, in case~$|\widehat{O_2}|=2$, we can use~$G_{\widehat{o_2}}$ with~$\widehat{o_2}\in \widehat{O_2}$. Furthermore, from~$G_o$ with~$o\in O_1$ of minimal degree, we can possibly locate~$H_{F_2\cup \widehat{O_2}}$ as annotated graph induced by a flank vertex (namely $f$) and two vertices that become outsiders after deleting that flank vertex. However, there could be another annotated graph~$H'$ obtained similarly by deleting a flank vertex (namely the vertex in~$O_1$ other than~$o$). However, this other annotated graph can be recovered from~$G_f$ (or $G_{\widehat{o_2}}$) as well by deleting the outsider in the class of size~$2$ of smaller degree.
Thus, we can obtain the pair~$(H_{F_2\cup \widehat{O_2}}, H_{{{\NTwo{G}{F_2\cup \widehat{O_2}}}}})$.

\noindent \textbf{Case 2.c.~$\mathbf{|O_1|=1}$ and~$\mathbf{|\widehat{O_2}|\neq 1}$.}
From~$G_o$ with~$o \in O_1$ it is possible to reconstruct~$H_{F_2\cup \widehat{O_2}}$. 
Since~$|\widehat{O_2}|\neq 1$, in~$G_f$, we can distinguish the two classes of outsiders and recover~$H_{{\NTwo{G}{F_2\cup \widehat{O_2}}}}$. 

\noindent \textbf{Case 2.d.~$\mathbf{|O_1|=1}$ and~$\mathbf{|\widehat{O_2}|}= 1$.} Say $O_1= \{o_1\}$. 
If there are more than two twin classes of vertices of maximum degree, then we can consider the class~$O_1^+$ of outsiders in~$G-o_1$. 
We get a triple~$\partition_G( O_1\cup O_1^+)$.
From~$G_{o_1}$, we can extract~$H_{S_1^+}$. From~$G_f$, we can possibly extract~$H_{O_1\cup O_1^+}$ as graph induced by an outsider and an outsider obtained by deleting the outsider. However, it may be the case that there is a second graph~$H'$ that is obtained in this fashion. Nevertheless, this second graph~$H'$ can also be extracted from~$G_{o_1}$. As before, we have reconstructed a clean clique separation of~$G$.

 Finally, suppose there are exactly two twin classes of vertices of maximum degree and~$|O_1|=1$ and~$|\widehat{O_2}|= 1$. We let~$A_1$ be the set of maximum degree neighbors of~$O_1$ and~$\widehat{A_2}$ be the maximum degree neighbors of~$\widehat{O_2}$. 

If~$G-(A_1\cup \widehat{A_2})$ is not connected, then we can reconstruct the graph with the usual technique of considering components up to isomorphism:
For this first note that we can distinguish vertices in~$A_1$ from those in~$\widehat{A_2}$ according to whether they have two sides or one. We can thus also determine from~$G_a$ whether~$a\in A_1$ or~$a\in \widehat{A_2}$.
We can recover the isomorphism type of the connected component that contains~$o_2\in O_2$ and~$f$ by considering~$G_{a_2}$ with~$a_2\in \widehat{A_2}$. Here, in case~$o_2$ and~$f$ are not twins in~$G_{a_2}$, we can also identify~$f$ as the vertex of smaller degree.

We can recover the connected component of~$O_1$ by considering~$G_{a_1}$ with~$a_1\in A_1$: if~$|A_1|>1$, then we can locate~$O_1$ in~$G_{a_1}$. Otherwise if~$|A_1|=1$, then in~$G_{a_1}$ we might not be able to distinguish~$O_1$ from~$f$. However, we already know the isomorphism type of the component that contains~$f$ (and can locate~$f$ in it), so the other component (or either one in case they are isomorphic) yields the isomorphism type of the connected component of~$O_1$.
Overall, we now assume that~$G-(A_1\cup \widehat{A_2})$ is connected.

Consider (similar to the end of Case A.2.) the longest path starting from~$f_2$ with vertices that have degree 2 in~$G-A_1-\widehat{A_2}$. If this covers all of~$G-A_1-\widehat{A_2}$, then~$G$ is reconstructible being a path with two suitably added twin classes. Otherwise define a vertex~$w$ as maximum degree vertex and~$w'$ as minimum degree vertex where we cannot continue anymore.
Note that~$G_o$ with~$o\in O_1$ is reconstructible.
The case follows now from similar arguments to the case where there is only one twin equivalence class of vertices of maximum degree and~$|O_1|>2$.
\end{proof}

\paragraph{Proof of the main theorem.} Having proven all the cases, we conclude with the proof of our main theorem.

\begin{proof}[Proof of Theorem~\ref{main:thm}]
We know that interval graphs are recognizable~\cite{DBLP:journals/dm/Rimscha83}.
For an interval graph~$G$, we can reconstruct the sizes of the (possibly empty) flanks by Lemma~\ref{lem:sizes}. If the graph has two flanks then the graph is reconstructible by Lemma~\ref{lem:two:flanks:reconstruc}. If the graph has no flanks, then it is reconstructible by Lemma~\ref{lem:no:flanks:reconstruct}. If the graph has exactly one flank, then it is reconstructible by Lemmas~\ref{lem:one:flank:size:more:than:one:reconstruc} and~\ref{lem:one:flank:size:one:reconstruc}.
\end{proof}

\textbf{Acknowledgements.} We thank an anonymous reviewer for spotting a mistake in one of our proofs and numerous valuable comments. The research leading to these results has received funding from the European Research Council (ERC) under the European Union’s Horizon 2020 research and innovation programme (EngageS: grant agreement No.\ 820148).

\bibliographystyle{alpha}
\bibliography{main}

\begin{thebibliography}{KNWZ21}

\bibitem[ACKR89]{DBLP:journals/jct/AlonCKR89}
Noga Alon, Yair Caro, Ilia Krasikov, and Yehuda Roditty.
\newblock Combinatorial reconstruction problems.
\newblock {\em J. Comb. Theory, Ser. {B}}, 47(2):153--161, 1989.

\bibitem[AFLM10]{DBLP:journals/arscom/AsciakFLM10}
Kevin~J. Asciak, M.~A. Francalanza, Josef Lauri, and Wendy~J. Myrvold.
\newblock A survey of some open questions in reconstruction numbers.
\newblock {\em Ars Comb.}, 97, 2010.

\bibitem[AKV24]{DBLP:journals/corr/abs-2406-09351}
Vikraman Arvind, Johannes K{\"{o}}bler, and Oleg Verbitsky.
\newblock On the expressibility of the reconstructional color refinement.
\newblock {\em CoRR, arXiv}, abs/2406.09351, 2024.

\bibitem[Bab95]{BabaiHandbook}
L\'{a}szl\'{o} Babai.
\newblock {\em Handbook of Combinatorics (vol. 2)}, chapter Automorphism
  groups, isomorphism, reconstruction, pages 1447--1540.
\newblock MIT Press, Cambridge, MA, USA, 1995.

\bibitem[BFZB23]{DBLP:journals/pami/BouritsasFZB23}
Giorgos Bouritsas, Fabrizio Frasca, Stefanos Zafeiriou, and Michael~M.
  Bronstein.
\newblock Improving graph neural network expressivity via subgraph isomorphism
  counting.
\newblock {\em {IEEE} Trans. Pattern Anal. Mach. Intell.}, 45(1):657--668,
  2023.

\bibitem[BH77]{MR0480189}
John~Adrian Bondy and R.~L. Hemminger.
\newblock Graph reconstruction---a survey.
\newblock {\em J. Graph Theory}, 1(3):227--268, 1977.

\bibitem[BHI07]{DBLP:journals/dam/Bang-JensenHI07}
J{\o}rgen Bang{-}Jensen, Jing Huang, and Louis Ibarra.
\newblock Recognizing and representing proper interval graphs in parallel using
  merging and sorting.
\newblock {\em Discret. Appl. Math.}, 155(4):442--456, 2007.

\bibitem[BM11]{MR2839371}
John~Adrian Bondy and Fabien Mercier.
\newblock Switching reconstruction of digraphs.
\newblock {\em J. Graph Theory}, 67(4):332--348, 2011.

\bibitem[Bol90]{Bollobas1990}
Béla Bollobás.
\newblock Almost every graph has reconstruction number three.
\newblock {\em J. Graph Theory}, 14(1):1--4, 1990.

\bibitem[Bon69]{Bondy69}
John~Adrian Bondy.
\newblock On {U}lam's conjecture for separable graphs.
\newblock {\em Pac. J. Math.}, 31(2):281--288, 1969.

\bibitem[Bon91]{bondy_1991}
John~Adrian Bondy.
\newblock {\em A Graph Reconstructor's Manual}, pages 221--252.
\newblock London Mathematical Society Lecture Note Series. Cambridge University
  Press, 1991.

\bibitem[Bra23]{tuprints26387}
Jendrik Brachter.
\newblock {\em Combinatorial approaches to the group isomorphism problem}.
\newblock PhD thesis, Technische Universit{\"a}t Darmstadt, Darmstadt, December
  2023.

\bibitem[BW99]{Bogart99}
Kenneth~P. Bogart and Douglas~B. West.
\newblock A short proof that `proper = unit'.
\newblock {\em Discret. Math.}, 201(1-3):21--23, 1999.

\bibitem[CG24]{calle2024combinatorialktheoryperspectiveedge}
Maxine~E. Calle and Julian~J. Gould.
\newblock A combinatorial {$K$}-theory perspective on the edge reconstruction
  conjecture in graph theory.
\newblock {\em CoRR, arXiv}, abs/2402.14986, 2024.

\bibitem[CLMS24]{MR4658425}
Alexander Clifton, Xiaonan Liu, Reem Mahmoud, and Abhinav Shantanam.
\newblock Reconstruction and edge reconstruction of triangle-free graphs.
\newblock {\em Discret. Math.}, 347(2):Paper No. 113753, 7, 2024.

\bibitem[CMR21]{DBLP:conf/nips/CottaMR21}
Leonardo Cotta, Christopher Morris, and Bruno Ribeiro.
\newblock Reconstruction for powerful graph representations.
\newblock In {\em Advances in Neural Information Processing Systems 34: Annual
  Conference on Neural Information Processing Systems 2021, NeurIPS 2021,
  December 6-14, 2021, virtual}, pages 1713--1726, 2021.

\bibitem[DHH96]{Deng1996}
Xiaotie Deng, Pavol Hell, and Jing Huang.
\newblock Linear-time representation algorithms for proper circular-arc graphs
  and proper interval graphs.
\newblock {\em SIAM J. Comput.}, 25(2):390--403, 1996.

\bibitem[EH22]{DBLP:journals/jgt/EgrotH22}
Rob Egrot and Robin Hirsch.
\newblock Seurat games on {S}tockmeyer graphs.
\newblock {\em J. Graph Theory}, 99(2):278--311, 2022.

\bibitem[EPY88]{Ellingham1988}
Mark~N. Ellingham, L{\'{a}}szl{\'{o}} Pyber, and Xingxing Yu.
\newblock Claw-free graphs are edge reconstructible.
\newblock {\em J. Graph Theory}, 12(3):445--451, 1988.

\bibitem[FG65]{Fulkerson1965}
Delbert~Ray Fulkerson and Oliver Gross.
\newblock Incidence matrices and interval graphs.
\newblock {\em Pac. J. Math.}, 15(3):835--855, September 1965.

\bibitem[Gar07]{Gardi07}
Fr{\'{e}}d{\'{e}}ric Gardi.
\newblock The {R}oberts characterization of proper and unit interval graphs.
\newblock {\em Discret. Math.}, 307(22):2906--2908, 2007.

\bibitem[Gil74]{Giles1974}
William~B Giles.
\newblock The reconstruction of outerplanar graphs.
\newblock {\em J. Comb. Theory, Ser. {B}}, 16(3):215--226, 1974.

\bibitem[Har74]{MR360368}
Frank Harary.
\newblock A survey of the reconstruction conjecture.
\newblock In {\em Graphs and combinatorics ({P}roc. {C}apital {C}onf., {G}eorge
  {W}ashington {U}niv., {W}ashington, {D}.{C}., 1973)}, volume Vol. 406 of {\em
  Lecture Notes in Math.}, pages 18--28. Springer, Berlin-New York, 1974.

\bibitem[HHRT07]{HEMASPAANDRA2007103}
Edith Hemaspaandra, Lane~A. Hemaspaandra, Stanisław~P. Radziszowski, and Rahul
  Tripathi.
\newblock Complexity results in graph reconstruction.
\newblock {\em Discret. Appl. Math.}, 155(2):103--118, 2007.
\newblock 29th Symposium on Mathematical Foundations of Computer Science MFCS
  2004.

\bibitem[HHSW20]{DBLP:journals/cc/HemaspaandraHSW20}
Edith Hemaspaandra, Lane~A. Hemaspaandra, Holger Spakowski, and Osamu Watanabe.
\newblock The robustness of {LWPP} and {WPP}, with an application to graph
  reconstruction.
\newblock {\em Comput. Complex.}, 29(2):7, 2020.

\bibitem[Hub11]{DBLP:journals/jct/Huber11}
Michael Huber.
\newblock Computational complexity of reconstruction and isomorphism testing
  for designs and line graphs.
\newblock {\em J. Comb. Theory, Ser. {A}}, 118(2):341--349, 2011.

\bibitem[Kel57]{Kelly57}
Paul~Joseph Kelly.
\newblock A congruence theorem for trees.
\newblock {\em Pac. J. Math.}, 7(1):961--968, 1957.

\bibitem[KH94]{DBLP:journals/mst/KratschH94}
Dieter Kratsch and Lane~A. Hemaspaandra.
\newblock On the complexity of graph reconstruction.
\newblock {\em Math. Syst. Theory}, 27(3):257--273, 1994.

\bibitem[KKLV11]{DBLP:journals/siamcomp/KoblerKLV11}
Johannes K{\"{o}}bler, Sebastian Kuhnert, Bastian Laubner, and Oleg Verbitsky.
\newblock Interval graphs: Canonical representations in logspace.
\newblock {\em {SIAM} J. Comput.}, 40(5):1292--1315, 2011.

\bibitem[KLS09]{e85c79e0-d93c-36bd-9f89-07652a5d63a3}
Géza Kós, Péter Ligeti, and Péter Sziklai.
\newblock Reconstruction of matrices from submatrices.
\newblock {\em Math. Comput.}, 78(267):1733--1747, 2009.

\bibitem[KNWZ21]{Kostochka2021}
Alexandr~V. Kostochka, Mina Nahvi, Douglas~B. West, and Dara Zirlin.
\newblock 3-regular graphs are 2-reconstructible.
\newblock {\em Eur. J. Comb.}, 91:103216, 2021.

\bibitem[KR97]{DBLP:journals/jct/KrasikovR97}
Ilia Krasikov and Yehuda Roditty.
\newblock On a reconstruction problem for sequences.
\newblock {\em J. Comb. Theory, Ser. {A}}, 77(2):344--348, 1997.

\bibitem[KSU10]{Kiyomi2010}
Masashi Kiyomi, Toshiki Saitoh, and Ryuhei Uehara.
\newblock Reconstruction of interval graphs.
\newblock {\em Theor. Comput. Sci.}, 411(43):3859--3866, 2010.

\bibitem[KSU12]{kiyomi2012}
Masashi Kiyomi, Toshiki Saitoh, and Ryuhei Uehara.
\newblock Bipartite permutation graphs are reconstructible.
\newblock {\em Discrete Math. Algorithm. Appl.}, 04(03), 2012.

\bibitem[Lau03]{LauriInHandbook}
Josef Lauri.
\newblock {\em Handbook of Graph Theory}, chapter The Reconstruction Problem.
\newblock Discrete Mathematics and Its Applications. Chapman {\&} Hall / Taylor
  {\&} Francis, 2003.

\bibitem[LB62]{Lekkerkerker1962}
Cornelis~Gerrit Lekkerkerker and Jan~Ch. Boland.
\newblock Representation of a finite graph by a set of intervals on the real
  line.
\newblock {\em Fund. Math.}, 51(1):45--64, 1962.

\bibitem[LS16]{Lauri2016}
Josef Lauri and Raffaele Scapellato.
\newblock {\em Topics in Graph Automorphism and Reconstruction}.
\newblock Number 432 in London Mathematical Society Lecture Note Series.
  Cambridge university press, 2 edition, 2016.

\bibitem[McK22]{MR4446370}
Brendan~D. McKay.
\newblock Reconstruction of small graphs and digraphs.
\newblock {\em Australas. J. Combin.}, 83:448--457, 2022.

\bibitem[MRT11]{DBLP:journals/siamdm/MontanariRT11}
Andrea Montanari, Ricardo Restrepo, and Prasad Tetali.
\newblock Reconstruction and clustering in random constraint satisfaction
  problems.
\newblock {\em {SIAM} J. Discret. Math.}, 25(2):771--808, 2011.

\bibitem[MS14]{MR3218273}
Brendan~D. McKay and Pascal Schweitzer.
\newblock Switching reconstruction of digraphs.
\newblock {\em J. Graph Theory}, 76(4):279--296, 2014.

\bibitem[MW14]{doi:10.1179/1743280413Y.0000000023}
Eric Maire and Philip~John Withers.
\newblock Quantitative {X}-ray tomography.
\newblock {\em Int. Mater. Rev.}, 59(1):1--43, 2014.

\bibitem[Myr88]{Myrvold1988}
Wendy~Joanne Myrvold.
\newblock {\em Ally and adversary reconstruction problems}.
\newblock PhD thesis, University of Waterloo, 1988.

\bibitem[Mü76]{Muller1976}
Vladimír Müller.
\newblock Probabilistic reconstruction from subgraphs.
\newblock {\em Commentationes Mathematicae Universitatis Carolinae},
  017(4):709--719, 1976.

\bibitem[OT16]{DBLP:journals/jgt/OliveiraT16}
Igor~C. Oliveira and Bhalchandra~D. Thatte.
\newblock An algebraic formulation of the graph reconstruction conjecture.
\newblock {\em J. Graph Theory}, 81(4):351--363, 2016.

\bibitem[PS17]{PITZ_SUABEDISSEN_2017}
Max~F. Pitz and Rolf Suabedissen.
\newblock A topological variation of the reconstruction conjecture.
\newblock {\em Glasg. Math. J.}, 59(1):221–235, 2017.

\bibitem[Rob69]{Roberts69}
Fred~Stephen Roberts.
\newblock {\em Proof Techniques in Graph Theory}, chapter Indifference Graphs,
  pages 139--146.
\newblock Academic Press, 1969.

\bibitem[Sha93]{DBLP:journals/tnn/Shawe-Taylor93}
John Shawe{-}Taylor.
\newblock Symmetries and discriminability in feedforward network architectures.
\newblock {\em {IEEE} Trans. Neural Networks}, 4(5):816--826, 1993.

\bibitem[Sta85]{MR787322}
Richard~P. Stanley.
\newblock Reconstruction from vertex-switching.
\newblock {\em J. Comb. Theory, Ser. {B}}, 38(2):132--138, 1985.

\bibitem[Sto77]{https://doi.org/10.1002/jgt.3190010108}
Paul~K. Stockmeyer.
\newblock The falsity of the reconstruction conjecture for tournaments.
\newblock {\em J. Graph Theory}, 1(1):19--25, 1977.

\bibitem[Tut79]{MR538033}
William~Thomas Tutte.
\newblock All the king's horses. {A} guide to reconstruction.
\newblock In {\em Graph theory and related topics ({P}roc. {C}onf., {U}niv.
  {W}aterloo, {W}aterloo, {O}nt., 1977)}, pages 15--33. Academic Press, New
  York-London, 1979.

\bibitem[vR83]{DBLP:journals/dm/Rimscha83}
Michael von Rimscha.
\newblock Reconstructibility and perfect graphs.
\newblock {\em Discret. Math.}, 47:283--291, 1983.

\bibitem[Yan88]{Yang88}
Yongzhi Yang.
\newblock The reconstruction conjecture is true if all 2-connected graphs are
  reconstructible.
\newblock {\em J. Graph Theory}, 12(2):237--243, 1988.

\bibitem[ZKT85]{Zemlyachenko1985}
V.~N. Zemlyachenko, Nikolai~M. Korneenko, and Regina~Iosifovna Tyshkevich.
\newblock Graph isomorphism problem.
\newblock {\em J. Sov. Math.}, 29(4):1426--1481, 1985.

\end{thebibliography}

\end{document}